\documentclass[11pt]{amsart}
\usepackage{enumitem}
\usepackage{amsthm}
\usepackage[paperheight=11in,paperwidth=8.5in,left=1in,top=1.25in,right=1in,bottom=1.25in,]{geometry}
 \usepackage{relsize}
\usepackage{rotating}
\usepackage[table]{xcolor}

\usepackage[super]{nth}
\usepackage[none]{hyphenat}
\usepackage{hyperref}
\hypersetup{
	colorlinks,
	citecolor=black,
	filecolor=black,
	linkcolor=black,
	urlcolor=black
}
\usepackage[capitalize, nameinlink, noabbrev, nosort]{cleveref}
\usepackage{graphicx, amssymb}
\usepackage{epstopdf} 
\usepackage{youngtab}
\usepackage[boxsize=.35 em]{ytableau}
\usepackage{comment}
\usepackage{setspace}
\usepackage{multicol}

\usepackage{tikz}   
\usetikzlibrary{arrows,calc,intersections,matrix,positioning,through}
\usepackage{tikz-cd}
\usepackage{comment}
\usepackage{diagbox}
\usepackage{blkarray}
\tikzset{commutative diagrams/.cd,every label/.append style = {font = \normalsize}}

\usepgflibrary{shapes.geometric}
\usetikzlibrary{shapes.geometric}
\usetikzlibrary{backgrounds}

\DeclareMathOperator{\rtop}{{\it k}} 

\DeclareMathOperator{\rt}{{\it k}} 
\DeclareMathOperator{\Gr}{Gr}
\DeclareMathOperator{\Mat}{Mat}

\DeclareMathOperator{\tw}{tw}

\DeclareMathOperator{\AFacet}{Froz}
\DeclareMathOperator{\Froz}{Froz}
\DeclareMathOperator{\refl}{refl}
\DeclareMathOperator{\cyc}{cyc}

\DeclareMathOperator{\pre}{pre}
\DeclareMathOperator{\inc}{inc}
\DeclareMathOperator{\Irr}{\xx}

\DeclareMathOperator{\triv}{triv}

\def\CD{\mathcal{CD}_{n,k}}

\newcommand{\Grk}{\Gr_{k,n}^{\scriptscriptstyle \ge 0}}

\def\AA{\mathcal{A}_{n,k,m}(Z)}
\newcommand{\Ank}{\mathcal{A}_{n, k, 4}({Z})}

\def\tZ{\tilde{Z}}
\def\bcfw{\bowtie}

\def\4biddenprop{4-coindependent}
\newcommand{\lr}[1]{\langle #1 \rangle}
\newcommand{\llrr}[1]{\langle\!\langle #1 \rangle\!\rangle}

\newcommand{\br}{\,|\,}
\newcommand{\gt}[1]{Z_{#1}}
\newcommand{\gto}[1]{Z_{#1}^\circ}
\def\ctop{D_{\rtop}}
\def\rchn{\nearrow}

\newcommand{\rzeta}{\bar{\zeta}}
\newcommand{\ralpha}{\bar{\alpha}}
\newcommand{\rbeta}{\bar{\beta}}
\newcommand{\rgamma}{\bar{\gamma}}
\newcommand{\rdelta}{\bar{\delta}}
\newcommand{\repsilon}{\bar{\varepsilon}}
\DeclareMathOperator{\Mut}{Mut}

\newcommand{\czeta}{{\zeta}}
\newcommand{\calpha}{{\alpha}}
\newcommand{\cbeta}{{\beta}}
\newcommand{\cgamma}{{\gamma}}
\newcommand{\cdelta}{{\delta}}
\newcommand{\cepsilon}{{\varepsilon}}

\newcommand{\szeta}{\hat{\zeta}}
\newcommand{\salpha}{\hat{\alpha}}
\newcommand{\sbeta}{\hat{\beta}}
\newcommand{\sgamma}{\hat{\gamma}}
\newcommand{\sdelta}{\hat{\delta}}
\newcommand{\sepsilon}{\hat{\varepsilon}}

\newcommand{\tzeta}{{\check{\zeta}}}

\newcommand{\tgamma}{{\check{\gamma}}}

\newcommand{\tepsilon}{{\check{\varepsilon}}}

\def\hxx{\hat{\xx}}
\def\hhxx{{\check{\xx}}}
\def\hhQ{\check{Q}}

\newcommand{\rPsi}{\overline{\Psi}}

\newcommand{\chR}{\Gr_{1,5}^{\scriptscriptstyle >0}}
\newcommand{\ZZ}{{L}}
\newcommand{\YY}{{M}}

\newcommand{\PP}{\mathbb{M}}

\newcommand{\C}{\mathbb{C}}

\newcommand{\CCC}{\mathcal{C}}
\newcommand{\A}{\mathcal{A}}

\def\mtx{M}
\def\twmt{M^{\tw}}

\def\O{\mathcal{O}}

\def\Acal{\mathcal{A}}

\def\R{\mathbb{R}}

\def\NN{\mathbb{N}}

\def\Vcal{\mathcal{V}}

\newcommand{\xx}{\mathbf{x}}
\newcommand{\txx}{\widetilde{\mathbf{x}}}

\def\rcp{\mathfrak{r}}
\def\rcpp{\mathfrak{p}}
\def\st{\text{FStep}}

\newtheorem*{theorem*}{Theorem}
\newtheorem{theorem}{Theorem}[section]
\newtheorem{lemma}[theorem]{Lemma}
\newtheorem{proposition}[theorem]{Proposition}
\newtheorem{corollary}[theorem]{Corollary}
\newtheorem{conjecture}[theorem]{Conjecture}
\theoremstyle{definition}
\newtheorem{definition}[theorem]{Definition}
\newtheorem{example}[theorem]{Example}

\newtheorem{remark}[theorem]{Remark}
\newtheorem{notation}[theorem]{Notation}
\newtheorem{claim}[theorem]{Claim}
\newtheorem{heuristic}[theorem]{Heuristic}

\setlist[itemize]{leftmargin=*}
\setlist[enumerate]{leftmargin=*}

\begin{document}
\begin{abstract}  The amplituhedron is a mathematical object which was introduced to provide a geometric origin of scattering amplitudes in $\mathcal{N}=4$ super Yang Mills theory.
It generalizes \emph{cyclic polytopes} and the \emph{positive Grassmannian}, and has a very rich combinatorics with 
connections to cluster algebras.
In this article we provide a series of results about tiles and tilings of the $m=4$ amplituhedron. Firstly, we provide a full characterization of facets of BCFW tiles in terms of cluster variables for $\mbox{Gr}_{4,n}$.
Secondly, we exhibit a tiling of the $m=4$ amplituhedron which involves a tile which does not come from the BCFW recurrence -- the \emph{spurion} tile, which also satisfies all cluster properties.
Finally, strengthening the connection with cluster algebras, 
we show that each standard BCFW tile is the positive part of a cluster variety, which allows us to compute the canonical form of each such tile explicitly in terms of cluster variables for $\mbox{Gr}_{4,n}$. This paper is a companion to our previous paper ``Cluster algebras and tilings for the 
$m=4$ amplituhedron.''
\end{abstract}
	
	\title{A cluster of results on amplituhedron tiles}
	\author[C. Even-Zohar]{Chaim Even-Zohar}
        \address{Faculty of Mathematics, Technion, Haifa, Israel}
	\email{chaime@technion.ac.il}
	\author[T. Lakrec]{Tsviqa Lakrec}
	\address{Institute of Mathematics, University of Zurich, Switzerland}
	\email{tsviqa@gmail.com}
	\author[M. Parisi]{Matteo Parisi}
	\address{CMSA, Harvard University, Cambridge, MA; Institute for Advanced Study, Princeton, NJ;}
 \email{mparisi@cmsa.fas.harvard.edu}
	
 \author[M. Sherman-Bennett]{Melissa Sherman-Bennett}
	\address{Department of Mathematics, MIT, Cambridge, MA}
	\email{msherben@mit.edu}
	\author[R. Tessler]{Ran Tessler}
	\address{Department of Mathematics, Weizmann Institute of Science, Israel}
	\email{ran.tessler@weizmann.ac.il}
	\author[L. Williams]{Lauren Williams}
	\address{Department of Mathematics, Harvard University, Cambridge, MA}
	\email{williams@math.harvard.edu}
	\maketitle
	
	\setcounter{tocdepth}{1}
	\tableofcontents

\section{Introduction}

The amplituhedron is a geometric object which was introduced
in the context of scattering amplitudes in $\mathcal{N}=4$ super Yang Mills theory.
In particular, the fact that the \emph{BCFW recurrence}\footnote{BCFW refers to Britto, Cachazo, Feng, and Witten} computes scattering amplitudes in $\mathcal{N}=4$ super Yang Mills theory
is a reflection of the geometric statement (which we proved in \cite{even2023cluster}) that each BCFW collection of cells in the positive Grassmannian
gives rise to a \emph{tiling} of the $m=4$ amplituhedron.  The $m=4$ amplituhedron also has a close connection to \emph{cluster algebras}:
we proved in \cite{even2023cluster} that each BCFW tile satisfies the \emph{cluster adjacency conjecture}, that is, its facets are cut out by compatible
cluster variables.

In this paper, which is 
a companion paper to \cite{even2023cluster}, we continue our study of the cluster structure and tilings of the $m=4$ amplituhedron. 
In particular, we provide 
a full characterization of the facets of BCFW tiles in terms of cluster variables for $\mbox{Gr}_{4,n}$.
For \emph{standard} BCFW tiles we prove our characterization of facets, see \cref{prop:standard_facets}, extending results of \cite{even2021amplituhedron}. For \emph{general} BCFW cells we state a characterization of facets in \cref{claim:facetsgenBCFW} but omit the proof, which uses the same ideas as the proof of \cref{prop:standard_facets} but is more technical.

While there are many tilings of the amplituhedron which use BCFW tiles, 
we show that there are also tilings that involve other tiles.
In particular, we exhibit the first known tiling of an 
amplituhedron which uses a non-BCFW tile, the \emph{spurion} tile.

Finally, strengthening the connection with cluster algebras, 
we show that each standard BCFW tile is the positive part of a cluster variety, see \cref{prop:birat-to-torus}.  In \cref{sec:canonical} we then use our description of BCFW tiles in terms of cluster
variables for $\mbox{Gr}_{4,n}$ in order to compute the canonical form of each such tile. 
The results of this paper provide computational tools to study BCFW tiles, their cluster structures, canonical forms and tilings.

The structure of this paper is as follows.  
In \cref{background1} and \cref{background2} we provide background on the amplituhedron and cluster algebras.
In \cref{sec:facets} we characterize the facets of BCFW tiles in terms of cluster variables for $\mbox{Gr}_{4,n}$.
In \cref{sec:spurion} we discuss the spurion tiling of the amplituhedron.
In \cref{sec:pospart} we show that each standard BCFW tile can be thought of as the positive part of a cluster variety.
Finally in \cref{sec:canonical} we explain how to compute the canonical form of a BCFW tile from the cluster variables.

\noindent{\bf Acknowledgements:~} 
The authors would like to thank Nima Arkani-Hamed for many
inspiring conversations.
TL is supported by SNSF grant Dynamical Systems, grant no.~188535.
MP is supported by the CMSA at Harvard University and at the Institute for Advanced Study by the U.S. Department of Energy under the grant number DE-SC0009988. 
MSB is supported by the National Science Foundation under Award No.~DMS-2103282.
RT (incumbent of the Lillian and George Lyttle Career Development Chair) was supported by the ISF grant No.~335/19 and 1729/23.
RT would like to thank Yoel Groman for discussions related to this work.
LW is supported by the National Science Foundation under Award No. 
DMS-1854316 and DMS-2152991. Any opinions, findings, and conclusions or recommendations expressed in this material are
those of the author(s) and do not necessarily reflect the views of the National Science
Foundation.  The authors would also like to thank 
Harvard University, 
the Institute for Advanced Study, 
and the `Research in Paris' program at the 
Institut Henri Poincar\'e, 
where some of this 
work was carried out.

\section{Background: the amplituhedron and BCFW tiles}\label{background1}
\subsection{The positive Grassmannian}

The \emph{Grassmannian} $\mbox{Gr}_{k,n}(\mathbb{F})$
is the space of all $k$-dimensional subspaces of
an $n$-dimensional vector space $\mathbb{F}^n$.
Let $[n]$ denote $\{1,\dots,n\}$, and $\binom{[n]}{k}$ denote the set of all $k$-element 
subsets of $[n]$.
We can
 represent a point $V \in
\mbox{Gr}_{k,n}(\mathbb{F})$  as the row-span 
of
a full-rank $k\times n$ matrix $C$ with entries in
$\mathbb{F}$.  
Then for $I=\{i_1 < \dots < i_k\} \in \binom{[n]}{k}$, we let $\lr{I}_V=\lr{i_1\,i_2\,\dots\,i_k}_V$ be the $k\times k$ minor of $C$ using the columns $I$. 
The $\lr{I}_V$ 
are called the {\itshape Pl\"{u}cker coordinates} of $V$, and are independent of the choice of matrix
representative $C$ (up to common rescaling). The \emph{Pl\"ucker embedding} 
$V \mapsto \{\lr{I}_V\}_{I\in \binom{[n]}{k}}$
embeds  $\mbox{Gr}_{k,n}(\mathbb{F})$ into
projective space\footnote{
We will sometimes abuse notation and identify $C$ with its row-span;
we will also drop the subscript $V$ on Pl\"ucker coordinates when it does not cause confusion.} . 
If $C$ has columns $v_1, \dots, v_n$, we may also identify $\lr{i_1\,i_2\,\dots\,i_k}$ with $v_{i_1} \wedge v_{i_2} \wedge \dots \wedge v_{i_k}$, hence e.g. $\lr{i_1\,i_2\,\dots\,i_k}=- \lr{i_2\,i_1\,\dots\,i_k}$. 
In this paper we will often be working with the 
\emph{real} Grassmannian $\mbox{Gr}_{k,n}=
\mbox{Gr}_{k,n}(\mathbb{R})$.
We will also denote by $\mbox{Gr}_{k,N}$ the Grassmannians of $k$-planes in a vector space with basis indexed by a set $N\subset [n]$.

\begin{definition}[Positive Grassmannian]\label{def:positroid}\cite{lusztig, postnikov}
We say that $V\in \Gr_{k,n}$ is \emph{totally nonnegative}
     if (up to a global change of sign)
       $\lr{I}_V \geq 0$ for all $I \in \binom{[n]}{k}$.
Similarly, $V$ is \emph{totally positive} if $\lr{I}_V >0$ for all $I
      \in \binom{[n]}{k}$.
We let $\Grk$ and $\Gr_{k,n}^{>0}$ denote the set of
totally nonnegative and totally positive elements of $\Gr_{k,n}$, respectively.  
$\Grk$ is called the \emph{totally nonnegative}  \emph{Grassmannian}, or
       sometimes just the \emph{positive Grassmannian}.
\end{definition}

If we partition $\Grk$ into strata based on which Pl\"ucker coordinates are strictly
positive and which are $0$, we obtain a cell decomposition of $\Grk$
into \emph{positroid cells} \cite{postnikov}.
Each positroid cell $S$ gives rise to a matroid $\mathcal{M}$, whose bases are precisely
the $k$-element subsets $I$ such that the Pl\"ucker coordinate
$\lr{I}$ does not vanish on $S$; $\mathcal{M}$ 
is called a \emph{positroid}.

One can index positroid cells in $\Grk$  by (equivalence classes of) \emph{plabic graphs} \cite{postnikov}.

\begin{definition}
	Let $G$ be a \emph{plabic graph}, i.e. a planar bipartite graph\footnote{We will always assume that plabic graphs are \emph{reduced} \cite[Definition 12.5]{postnikov}.} embedded in a disk, with black vertices $1, 2, \dots, n$ on the boundary of the disk. An \emph{almost perfect matching} $M$ of $G$ is a collection of edges which covers each internal vertex of $G$ exactly once. The \emph{boundary} of $M$, denoted $\partial M$, is the set of boundary vertices covered by $M$. The positroid associated to $G$ is the collection $\mathcal{M}=\mathcal{M}(G):=\{\partial M: M \text{ an almost perfect matching of }G \}$.
\end{definition}

For more details about plabic graphs relevant for this paper, see e.g. \cite[Appendix A]{even2023cluster}.

Both $\Gr_{k,n}$ and $\Grk$ admit the following set of operations, 
which will be useful to us.

\begin{definition}[Operations on the Grassmannian]\label{def:opGr} We define the following maps on $\Mat_{k,n}$, which descends to maps on  $\Gr_{k,n}$ and $\Grk$, which we denote in the same way:
\begin{itemize}
    \item (cyclic shift) We define the \emph{cyclic shift}
 as the map $\cyc: \Mat_{k, n} \to \Mat_{k,n}$ which sends  $v_1 \mapsto (-1)^{k-1}v_{n}$ and $v_i \mapsto v_{i-1}, 2 \leq i \leq n$, and in terms of Pl\"ucker coordinates: $\lr{I} \mapsto \lr{I-1}$.
 \item  (reflection) We define \emph{reflection} as the map $\refl:\Mat_{k, n} \to \Mat_{k,n}$ which sends $v_i \mapsto v_{n+1-i}$ and rescales a row by $(-1)^{{k\choose 2}}$,  and in terms of Pl\"ucker coordinates: $\lr{I} \mapsto \lr{n+1-I}$.
 \item (zero column) For $J\subseteq [n]$, we define the map $\pre_J:\Mat_{k, [n] \setminus \{i\}} \to \Mat_{k,n}$ which adds  zero columns in positions $J$, and in terms of Pl\"ucker coordinates: $\lr{I} \mapsto \lr{I}$.
\end{itemize}
Here, $I-1$ is obtained from $I \in {[n] \choose k}$ by subtracting $1$ (mod $n$) from each element of $I$ and $n+1-I$ is obtained from $I$ by subtracting each element of $I$ from $n+1$.
\end{definition}

\subsection{The amplituhedron}

Building on \cite{abcgpt,hodges},
Arkani-Hamed and Trnka
\cite{arkani-hamed_trnka}
introduced
the \emph{(tree) amplituhedron}, which they defined as 
the image of the positive Grassmannian under a positive linear map.
Let $\Mat_{n,p}^{>0}$ denote the set of $n\times p$ matrices whose maximal minors
are positive.

\begin{definition}[Amplituhedron]\label{defn_amplituhedron}
Let $Z\in \Mat_{n,k+m}^{>0}$, where $k+m \leq n$. 
    The \emph{amplituhedron map}
$\tilde{Z}:\Gr_{k,n}^{\ge 0} \to \Gr_{k,k+m}$
        is defined by
        $\tilde{Z}(C):=CZ$,
    where
 $C$ is a $k \times n$ matrix representing an element of
        $\Gr_{k,n}^{\ge 0}$,  and  $CZ$ is a $k \times (k+m)$ matrix representing an
         element of $\Gr_{k,k+m}$.
        The \emph{amplituhedron} $\mathcal{A}_{n,k,m}(Z) \subset \Gr_{k,k+m}$ is the image
$\tilde{Z}(\Gr_{k,n}^{\ge 0})$.
\end{definition}

In this article we will be concerned with the case where $m=4$. 

\begin{definition}[Tiles]
\label{defn_tile}
	Fix $k, n, m$ with $k+m \leq n$ and choose
	$Z\in \Mat_{n,k+m}^{>0}$.  
	Given a positroid cell $S$ of 
	$\Gr_{k,n}^{\ge 0}$, we let 
	$\gto{S} := \tilde{Z}(S)$ and 
	$\gt{S}: = \overline{
		\tilde{Z}(S)} = \tilde{Z}(\overline{S})$.
	We call
	$\gt{S}$  and $\gto{S}$
	a \emph{tile}  and an \emph{open tile}
	for $\AA$ 
	if $\dim(S) =km$ and 
	$\tilde{Z}$ is injective on $S$. 
\end{definition}

\begin{definition}[Tilings]\label{def:tiling}
	A \emph{tiling} 
                of $\AA$ is
                a collection
		$\{Z_{S} \ \vert \ S\in \mathcal{C}\}$
                 of tiles, such that their union equals $\AA$ and the open tiles $\gto{S},\gto{S'}$ are pairwise disjoint.
\end{definition}

There is a natural notion of  \emph{facet} of a tile, 
generalizing the notion of facet of a polytope. 

\begin{definition}[Facet of a cell and a tile]\label{def:facet2} 
Given two positroid cells $S'$ and $S$, we say that 
$S'$ is a \emph{facet} of $S$ if 
$S' \subset \partial{S}$ and $S'$ has codimension $1$ in $\overline{S}$.
If $S'$ is a facet of $S$ and $Z_S$ is a tile of $\AA$, we say that $Z_{S'}$ is a
	\emph{facet} 
 of $Z_{S}$ if 
         $\gt{S'} \subset \partial \gt{S}$ and has codimension 1 in $\gt{S}$.
\end{definition}

\begin{definition}[Twistor coordinates]\label{def:tw_coords}
Fix
  $Z \in \Mat_{n,k+m}^{>0}$
   with rows  $Z_1,\dots, Z_n \in \R^{k+m}$.
   Given  $Y \in \Gr_{k,k+m}$
   with rows $y_1,\dots,y_k$,
   and $\{i_1,\dots, i_m\} \subset [n]$,
   we define the \emph{twistor coordinate} $
	\llrr{{i_1} {i_2} \cdots {i_m}}$
        to be  the determinant of the
        matrix with rows
        $y_1, \dots,y_k, Z_{i_1}, \dots, Z_{i_m}$.
\end{definition}

Note that the twistor coordinates are defined only up to a common scalar multiple. An element of $\mbox{Gr}_{k, k+m}$ is uniquely determined by its twistor coordinates \cite{karpwilliams}. Moreover, $\mbox{Gr}_{k,k+m}$ can be embedded into $\mbox{Gr}_{m,n}$ so that the twistor coordinate $\llrr{i_1 \dots i_m}$ is the pullback of the Pl\"ucker coordinate $\lr{i_1, \dots, i_m}$ in $\mbox{Gr}_{m,n}$.
\begin{definition}\label{def:functionary}
We refer to a homogeneous polynomial in twistor coordinates as a \emph{functionary}. For $S \subseteq \Grk$, we say a functionary $F$ has a definite sign $s \in \{\pm 1\}$ (or vanishes) on $\gto{S}$ if for all $Z\in \Mat_{n,k+4}^{>0}$ and for all $Y \in \gto{S}$, $F(Y)$ has sign $s$ (or $0$, respectively).  A functionary is \emph{irreducible} if it is the pullback of an irreducible function on $\Gr_{m,n}$.
\end{definition}

We will use functionaries to describe amplituhedron tiles and to connect with cluster algebras.
\subsection{BCFW cells and BCFW tiles}
In this section we review the operation of \emph{BCFW product} used to build BCFW cells, following \cite[Section 5]{even2023cluster}.  We then define BCFW cells and  tiles.

\begin{notation}\label{not:LR_cluster}
Choose integers $1\leq a<b<c<d<n$ with $a,b$ and $c,d,n$ consecutive.
	Let\footnote{
Note that we will overload the notation and let $n$
	index an element of a vector space basis for different
	vector spaces; however, in what follows, the meaning should
	be clear from context.} $N_L = \{n,1,2,\dots,a,b\}, N_R = \{b, \dots, c, d, n\}$ and $B=(a,b,c,d,n)$\footnote{The `B' stands for ``butterfly.''}. Also fix $k \leq n$ and two nonnegative integers $k_L \leq |N_L|$ and $ k_R\leq |N_R|$ such that $k_L + k_R +1=k$. 
\end{notation}

\begin{remark}
While it is convenient to state our results in terms of $[n]$ and $\mbox{Gr}^{\geq 0}_{k,n}$, our results hold if we replace $[n]$ by 
 any set of indices $N \subset [n]$, and replace $1$ and $n$ by the smallest and largest elements of $N$, respectively.
\end{remark}

\begin{definition}[BCFW product]\label{def:butterfly}
Let $S_L \subseteq \mbox{Gr}^{\geq 0}_{k_L,N_L}, S_R \subseteq \mbox{Gr}^{\geq 0}_{k_R,N_R}$ be as in \cref{not:LR_cluster}, with  $G_L, G_R$ the respective plabic graphs, and let $B=(a,b,c,d,n)$ as in 
\cref{not:LR_cluster}.  The \emph{BCFW product} of $S_L$ and $S_R$ is the positroid cell $S_L \bcfw S_R \subseteq \mbox{Gr}^{\geq 0}_{k,n}$ corresponding to the plabic graph in the right-hand side of \Cref{fig:butterfly}.

\begin{figure}[h]
\centering
\includegraphics[width=\linewidth]{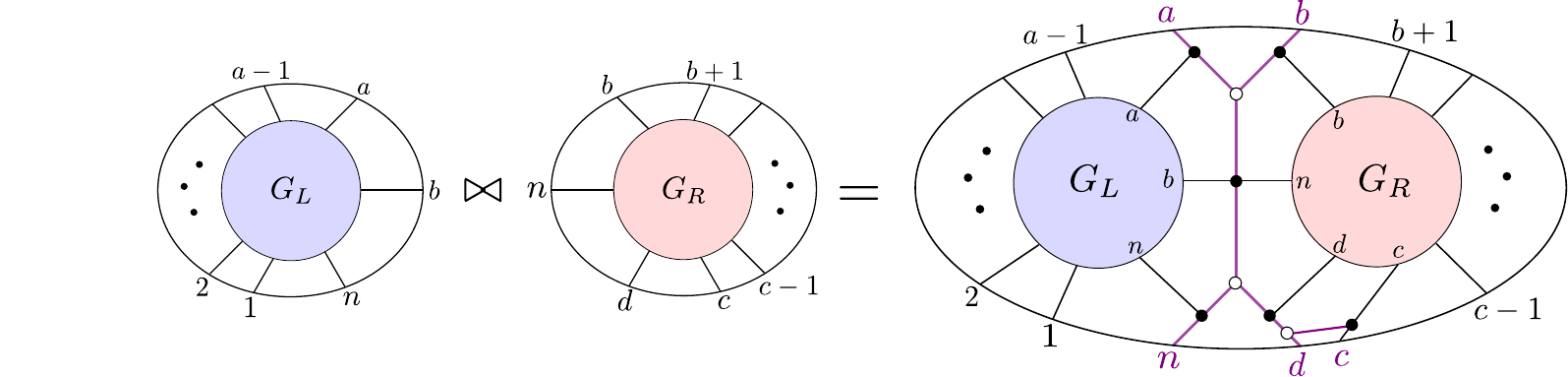}
\caption{The BCFW product $S_L \bcfw S_R$ of $S_L$ and $S_R$ in terms of their plabic graphs. 
Note that $G_L$ and $G_R$ are joined along the purple graph associated to 
$B=(a,b,c,d,n)$; we call it a `butterfly graph' since it resembles a butterfly.}
\label{fig:butterfly}
\end{figure}

\end{definition}
When it is not clear from the context, we will say $\bcfw$ is performed `with indices $B$'.

We now introduce the family of \emph{BCFW cells} to be the set of positroid cells which is closed under the operations in \Cref{def:opGr,def:butterfly}:

\begin{definition}[BCFW cells]\label{def:BCFW_cell}
The set of \emph{BCFW cells} is defined recursively. For $k=0$, let the trivial cell $\Gr^{\scriptscriptstyle>0}_{0,n}$ be a BCFW cell. This is represented by a plabic graph with black lollipops at each of the boundary vertices.  If $S$ is a BCFW cell, so is the cell obtained by applying $\mbox{cyc}, \mbox{refl}, \mbox{pre}$ to $S$. If $S_L,S_R$ are BCFW cells, so is their BCFW product $S_L \bcfw S_R$.
\end{definition}
\begin{remark}
It follows from the definition that the plabic graph of a BCFW cell is built by glueing together
a collection of (possibly rotated or reflected) `butterfly graphs.'  We could therefore refer to the plabic graph of a BCFW cell as a \emph{kaleidoscope}\footnote{A group of butterflies is officially
called a \emph{kaleidoscope}.}. 
\end{remark}

The \emph{standard} BCFW cells, 
which we define below, are 
a particularly nice subset of BCFW cells.
The images of the standard BCFW cells 
yield a tiling of the amplituhedron~\cite{even2021amplituhedron}. 

\begin{definition}[Standard BCFW cells]\label{def:std-bcfw-cells}
 The set of \emph{standard BCFW cells} is defined recursively. For $k=0$, let the trivial cell $\Gr^{\scriptscriptstyle>0}_{0,n}$ be a BCFW cell. If $S$ is a BCFW cell, so is the cell obtained by adding a zero column using $\mbox{pre}$ in the penultimate position. If $S_L,S_R$ are BCFW cells, so is their BCFW product $S_L \bcfw S_R$.
\end{definition}

\begin{figure} 
\includegraphics[width=0.6\textwidth]{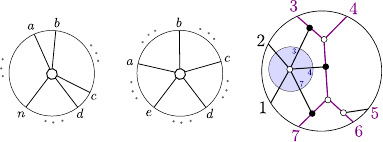} 
\caption{  The plabic graph of a standard BCFW cell (left) and general BCFW cell (center) in $\Gr^{\scriptscriptstyle\geq 0}_{1,n}$, where the $\ldots$ denote black lollipops in the remaining indices; the plabic graph of a BCFW cell $S_{ex} \subset \mbox{Gr}^{\geq 0}_{2,7}$ (right).}\label{fig:bcfwcellsk1}
\end{figure}

\begin{example}\label{ex:bcw_tile}
For $k=1$, each BCFW cell in $\Gr_{1,n}^{\geq 0}$ has a plabic graph of the form shown in \cref{fig:bcfwcellsk1} (middle). The Pl\"ucker coordinates $\lr{a},\lr{b},\lr{c},\lr{d},\lr{e}$ are positive, and all others are zero. 
    In \cref{fig:bcfwcellsk1} (right), $S_{ex} \subset \mbox{Gr}^{\geq 0}_{2,7}$ is obtained as $S_L \bowtie S_R$, with $S_L, S_R$ BCFW cells in $\mbox{Gr}^{\geq 0}_{1,N_L}, \mbox{Gr}^{\geq 0}_{0,N_R}$ respectively, with  $N_L=\{7,1,2,3,4\}, N_R=\{4,5,6,7\}$ and $B=(3,4,5,6,7)$. The standard BCFW cells for $k=1$ are  those BCFW cells where $a,b$ and $c,d$ are consecutive and $e=n$, as shown in 
    \cref{fig:bcfwcellsk1} (left). For $k=n-4$, the \emph{totally} positive Grassmannian $\Gr_{n-4,n}^{\scriptscriptstyle>0}$ is the only BCFW cell.
\end{example}

In \cite[Section 7]{even2023cluster} we showed that the 
amplituhedron map is injective on each BCFW cell.  We can therefore define \emph{BCFW tiles}.

\begin{definition}[BCFW tiles and standard BCFW tiles] \label{def:BCFWtile}
We define a \emph{BCFW tile} to be 
the (closure of the) image of a BCFW cell under the amplituhedron map.  In other words, each BCFW tile has the form 
$Z_{\rcp}:=\overline{\tilde{Z}(S_{\rcp})},$ where $\rcp$ is a recipe.  We define a \emph{standard BCFW tile} to be a BCFW tile that comes from a standard BCFW cell.  
\end{definition}

\subsection{Standard BCFW cells from chord diagrams}
\label{sec:quord}

In this section we introduce \emph{chord diagrams}, and show 
how each gives an algorithm for constructing  a standard BCFW cell.
In \cref{sec:recipes} we then give a  generalization of this algorithm, called a \emph{recipe}, for 
constructing a general BCFW cell.  

\begin{definition}[Chord diagram \cite{even2021amplituhedron}]\label{def:cd} 
Let $k,n \in \mathbb{N}$. A~\emph{chord diagram} $D \in\mathcal{CD}_{n,k}$ is a set of $k$~quadruples named \emph{chords}, of integers in the set $\{1,\dots,n\}$ named \emph{markers}, of the following form:
	$$ D \;=\; \{(a_1,b_1,c_1,d_1),\dots,(a_k,b_k,c_k,d_k)\} \;\;\text{ where }\;\;
	b_i=a_i+1 \text{ and }d_i=c_i+1$$ such that every 
	chord $D_i=(a_i,b_i,c_i,d_i) \in D$ satisfies
$ 1 \;\leq\; a_i \;<\; b_i \;<\; c_i \;<\; d_i \;\leq\; n-1 $
and \emph{no} two chords $D_i,D_j \in D$ satisfy
$ a_i \;=\; a_j$ or $a_i \;<\; a_j \;<\; c_i \;<\; c_j.$
\end{definition}

The number of different chord diagrams with $n$ markers and $k$ chords is the  
Narayana number ${N(n-3,k+1)}$:
$ \left|\mathcal{CD}_{n,k}\right| \;=\; \frac{1}{k+1}\binom{n-4}{k}\binom{n-3}{k}$.

\begin{figure}
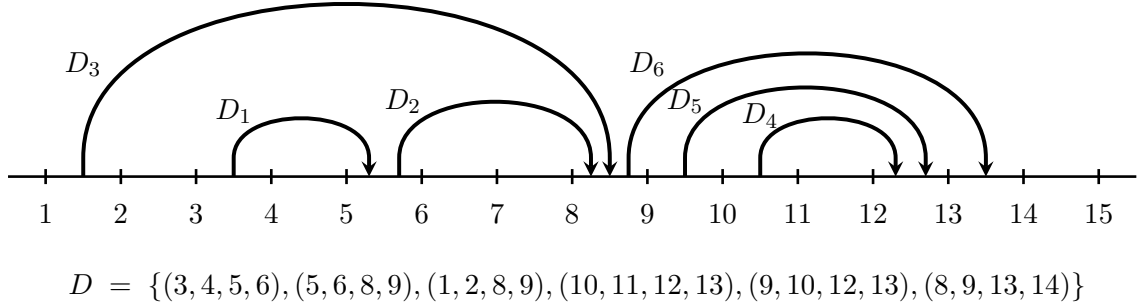

\begin{center}
\tikz[line width=1,scale=1]{
\draw (0.5,0) -- (15.5,0);
\foreach \i in {1,2,...,15}{
\def\x{\i}
\draw (\x,-0.1)--(\x,+0.1);
\node at (\x,-0.5) {\i};}
\foreach \i/\j in {1/8, 3/4.8, 5.2/7.75, 8.25/13, 9/12.2, 10/11.8}{
\def\x{\i+0.5}
\def\y{\j+0.5}
\draw[line width=1.5,-stealth] (\x,0) -- (\x,0.25) to[in=90,out=90] (\y,0.25) -- (\y,0);
}
\node at(1.5,1.5) {$D_3$};
\node at(3.5,0.875) {$D_1$};
\node at(5.75,1) {$D_2$};
\node at(9,1.5) {$D_6$};
\node at(9.5,1) {$D_5$};
\node at(10.5,0.8125) {$D_4$};
}
\vspace{0.5em}
$$ D \;=\; \{(3,4,5,6),(5,6,8,9),(1,2,8,9),(10,11,12,13),(9,10,12,13),(8,9,13,14)\} $$
\end{center}
\caption{
A chord diagram $D$ with $k=6$ chords $n=15$ markers.
} 
\label{cd-example}
\end{figure}

See Figure~\ref{cd-example}, where we visualize such a chord diagram $D$ in the plane as a horizontal line with $n$ markers labeled $\{1,\dots,n\}$ from left to right, and $k$ nonintersecting chords above it, whose \emph{start} and \emph{end} lie in the segments $(a_i,b_i)$ and $(c_i,d_i)$ respectively. The definition imposes restrictions on the chords: they cannot start before $1$, end after $n-1$, or start or end on a marker. Two chords cannot start in the same segment $(s,s+1)$, and one chord cannot start and end in the same segment, nor in adjacent segments. Two chord cannot cross.

We say that a chord is a \emph{top chord} if there is no chord above it, e.g. $D_3$ and $D_6$ in Figure~\ref{cd-example}. One natural way to label
the chords is by 
$D_1,\dots,D_k$ such that for all $1 \leq j \leq k$, $D_j$ is the rightmost top chord among the set of chords $\{D_1,\dots,D_j\}$ as in \cref{cd-example}. This is equivalent to sorting 
the chords according to their ends.

\begin{definition}[Terminology for chords]
\label{cd-terminology}
A~chord is a~\emph{top} chord 
if there is no chord above it,
 and otherwise it is a \emph{descendant} of the chords above it, called its \emph{ancestors}, and in particular a~\emph{child} of the chord immediately above it, which is called its~\emph{parent}. 
Two chords are \emph{siblings} if they are either top chords or children of a common parent. 
Two chords are \emph{same-end} 
if their ends occur in a common segment $(e,e+1)$, are \emph{head-to-tail} if the first ends in the segment where the second starts, and are \emph{sticky} if their 
starts lie in consecutive segments $(s,s+1)$ and~$(s+1,s+2)$.
\end{definition}

\begin{example}
 Consider the chord diagram in \cref{cd-example}. $D_4$ has parent $D_5$ and ancestors $D_5$ and $D_6$. $D_1$ and $D_2$ are siblings, and $D_3$ and $D_6$ are siblings. Chords $D_2$ and $D_3$ are same-end, chords $D_1$ and 
	$D_2$ are head-to-tail, and chords $D_5$ and $D_6$ are sticky.
\end{example}

\begin{remark}
\label{index-sets}
The definition of a chord diagram naturally extends to the case of a finite set of markers 
$N \subset \{1,\dots,n\}$ rather than $\{1,\dots,n\}$, and a set $K$ of chord indices rather than $\{1,\dots,k\}$. 
We will always have that the largest marker is 
$n\in N$, the starts and ends of chords will be consecutive pairs in~$N$ (and also $\mathbb{N}$) and the rightmost top chord will be denoted by $D_{\rtop} = D_{\max K}$. 
The notion of chord subdiagram in \cref{def:leftright} is an example of this extended notion of chord diagram.
\end{remark}

\begin{definition}[Left and right subdiagrams]
\label{def:leftright}
Let $D$ be a chord diagram in $\mathcal{CD}_{n,k}$.
A \emph{subdiagram} is obtained by restricting to a subset of the chords and a 
subset of the markers which contains both these chords and the marker~$n$.
Let
$\ctop = (a,b,c,d)$ be the rightmost top chord of~$D$, where $1\leq a<b<c<d<n$, and moreover $a,b$ and $c,d$ are consecutive. 

In the case that $d,n$ are consecutive as well
we define $D_L$, the \emph{left subdiagram} of $D$, on the markers $N_L=\{1,2,\dots,a,b,n\}$
and the \emph{right subdiagram} $D_R$ on~$N_R=\{b,\dots,c,d,n\}$. The subdiagram $D_L$ contains all chords that are to the left of $D_{\rtop}$, and $D_R$ contains the descendants of~$D_{\rtop}$. 
\end{definition}

\begin{example}
\label{right-left-diagrams}
For the chord diagram $D$ in \cref{cd-example}, the rightmost top chord is $D_6 = (8,9,13,14)$, so $N_L = \{1,\dots,9,15\}$ and $D_L = \{D_1,D_2,D_3\}$, while $N_R = \{9,\dots,15\}$ and $D_R = \{D_4,D_5\}$.
\end{example}

\begin{definition}[Standard BCFW cell from a chord diagram]
\label{def:standardfromCD}
Let $D$ be a chord diagram with $k$ chords on a set of markers $N$.
We recursively construct from $D$ a standard BCFW cell $S_D$ in ${\Gr}^{\scriptscriptstyle\ge0}_{k, N}$ as follows:
\begin{enumerate}[align=left]
\itemsep0.125em
\item 
If $k=0$, then the BCFW cell is the trivial cell $S_D:=\Gr^{\scriptscriptstyle\ge0}_{0,N}$.
\item Otherwise, let $D_{\rt}=(a,b,c,d)$ be the rightmost top chord of~$D$ and let $p$ denote the penultimate marker in $N$.
\begin{enumerate}
\itemsep0.125em
\item 
If $d\neq p$, let $D'$ be the subdiagram on 
$N\setminus \{p\}$
with the same chords as $D$, and let $S_{D'}$ be the standard BCFW cell associated to $D'$.  Then, we define $S_D := \pre_{p} S_{D'}$, which denotes the standard BCFW cell obtained from $S_{D'}$ by inserting a zero column in the penultimate position~$p$. 
\item If $d=p$, let $S_L$ and $S_R$ be the standard BCFW cells on $N_L$ and $N_R$ associated to 
the left and right subdiagrams	$D_L$ and $D_R$ of $D$.  Then, we let $S_D := S_L \bcfw S_R$, the standard BCFW cell which is their BCFW product as in \cref{def:butterfly}.
\end{enumerate}
\end{enumerate}
\end{definition}

\begin{figure}
	\includegraphics[width=0.7\textwidth]{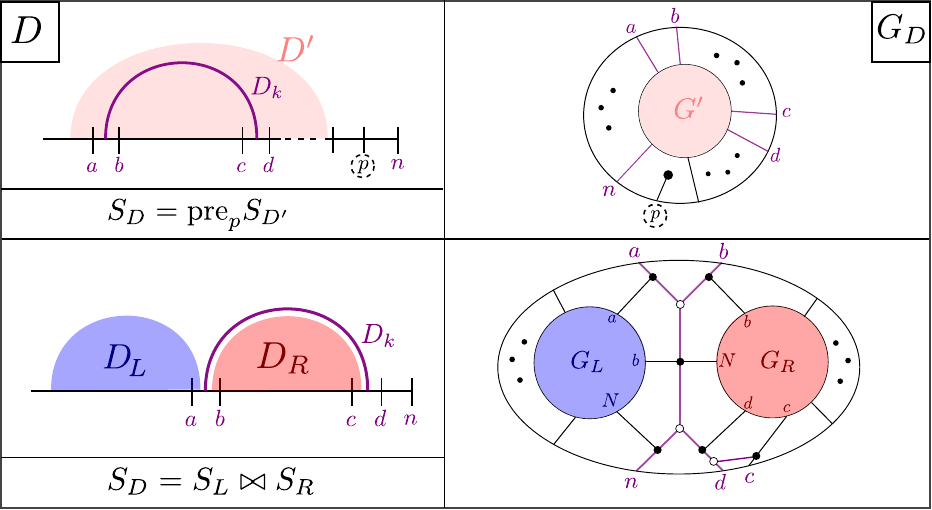}
	\caption{Recursive construction of a standard BCFW cell from a chord diagram as in \cref{def:standardfromCD}. 
 Top left (right): construction of $D$ ($G_D$) from $D'$ ($G'$) as in $(1a)$; bottom left (right) construction of $D$ ($G_D$) from $D_L,D_R$ ($G_L,G_R$) as in $(1b)$.}
	\label{fig:bcfw_chord}
\end{figure}

\begin{example}\label{ex:chordcell}
The standard BCFW cell $S_D$ of the chord diagram $D$ in \cref{cd-example} is $S_L \bcfw S_R$ where the chord subdiagrams $D_L,D_R$ are as in \cref{right-left-diagrams}. One can keep applying the recursive definition and obtain:
\begin{align*}
S_L \;& =\;
\Gr_{0,\{1,2,15\}}
\bcfw
\left( 
\left(
\Gr_{0,\{2,3,4,15\}}
\bcfw
\Gr_{0,\{4,5,6,15\}}
\right)
\bcfw
\Gr_{0,\{6,7,8,9,15\}}
\right)
\\
S_R \;& =\;
\pre_{14} 
\left(
\Gr_{0,\{9,10,15\}}
\bcfw
\left(
\Gr_{0,\{10,11,15\}}
\bcfw
\Gr_{0,\{11,12,13,15\}}
\right)
\right)
\end{align*}
\end{example}

\subsection{BCFW cells from recipes}\label{sec:recipes}
In this section, we review the conventions for labeling general BCFW cells from \cite[Section 6]{even2023cluster}. 
Each general BCFW cell may be specified 
by a list of operations from \cref{def:BCFW_cell}. 
The class of general BCFW cells includes the standard BCFW cells, but 
is additionally closed under the operations
of cyclic shift, reflection, and inserting a zero column anywhere (cf. \cref{def:BCFW_cell})
at any stage of the recursive generation.  
Since any sequence of these operations can be expressed as $\pre_I$ followed by $\cyc^r$ followed by $\refl^s$ for some $I, r, s$,
we can specify in a concise form which ones take place after each BCFW product. We will record the generation of a BCFW cell using the formalism of \emph{recipe} in \cref{def:recipe}.

\begin{definition}[General BCFW cell from a recipe] \label{def:recipe} 
A \emph{step-tuple} on 
a finite index set $N\subset \NN$ is a $4$-tuple 
\[((a_i, b_i, c_i, d_i, n_i),\pre_{I_i}, \cyc^{r_i}, \refl^{s_i}),\]
where $I_i \subseteq N$ such that $n_i$ is the largest element in $N \setminus I_i$, 
$a_i<b_i$ and $c_i<d_i<n_i$ are both consecutive in $N\setminus I_i$,  $0 \leq r_i < |N|$, and $s_i \in \{0,1\}$. 
A step-tuple records in order: a BCFW product of two cells using indices $(a_i, b_i, c_i, d_i, n_i)$; zero column insertions in positions $I_i$; applying the cyclic shift $r_i$ times; applying reflection $s_i$ times. Note that some of these operations may be the identity. Each operation in a step-tuple which is not the identity is called a~\emph{step}. 

A \emph{recipe} $\rcp$ on $N$ is either the empty set (the \emph{trivial recipe} on $N$, denote $\rcp^{\triv}_N$),
or a recipe $\rcp_L$ on $N_L$ followed by a recipe $\rcp_R$ on $N_R$ followed by a step-tuple
$((a_k, b_k, c_k, d_k, n_k),\pre_{I_k}, \cyc^{r_k}, \refl^{s_k})$ on $N$,
where $N_L = (N\setminus I_k) \cap \{n_k, \dots, a_k,b_k\}$ and 
$N_R = (N\setminus I_k) \cap \{b_k,\dots, c_k,d_k,n_k\}$.
We let $S_{\rcp}$ denote the general BCFW cell on $N$ obtained by applying the sequence of 
	operations specified
by $\rcp$. If $\rcp$ consists of $k$ step-tuples, then $S_{\rcp} \subset \Gr_{k, N}^{\scriptscriptstyle\geq 0}$.
\end{definition}

\begin{example} \label{ex:recipe}
Consider the recipe $\rcp$ consisting of the following sequence of $4$ step-tuples:
\begin{equation*}
    ((3,4,5,6,12),\pre_{2}), ((1,2,5,6,12), \cyc^{2}, \refl)), ((6,7,8,9,11),\pre_{10,12}), ((5,6,10,11,12), \cyc^{4}, \refl).
\end{equation*}
\cref{fig:bcfw_tile} shows the plabic graph of the general BCFW cell $S_\rcp$ obtained from $\rcp$ following \cref{def:recipe}. 
\end{example}

\begin{figure}
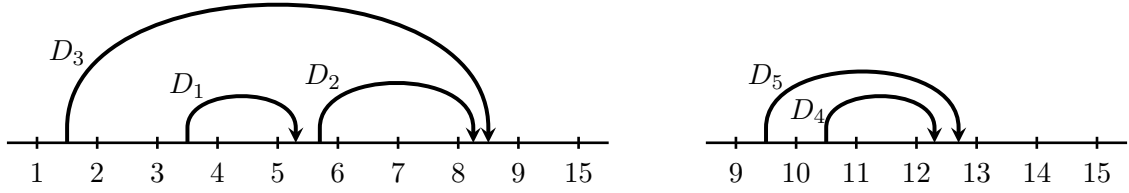

\centering
\begin{center}
\tikz[line width=1,scale=0.8]{
\draw (0.5,0) -- (10.5,0);
\foreach \i/\j in {1/1,2/2,3/3,4/4,5/5,6/6,7/7,8/8,9/9,10/15}{
\def\x{\i}
\draw (\x,-0.1)--(\x,+0.1);
\node at (\x,-0.5) {\j};}
\foreach \i/\j in {1/8, 3/4.8, 5.2/7.75}{
\def\x{\i+0.5}
\def\y{\j+0.5}
\draw[line width=1.5,-stealth] (\x,0) -- (\x,0.25) to[in=90,out=90] (\y,0.25) -- (\y,0);
}
\node at(1.5,1.5) {$D_3$};
\node at(3.5,0.95) {$D_1$};
\node at(5.75,1.1) {$D_2$};
}
\hspace{1cm}
\tikz[line width=1,scale=0.8]{
\draw (8.5,0) -- (15.5,0);
\foreach \i in {9,10,...,15}{
\def\x{\i}
\draw (\x,-0.1)--(\x,+0.1);
\node at (\x,-0.5) {\i};}
\foreach \i/\j in {9/12.2, 10/11.8}{
\def\x{\i+0.5}
\def\y{\j+0.5}
\draw[line width=1.5,-stealth] (\x,0) -- (\x,0.25) to[in=90,out=90] (\y,0.25) -- (\y,0);
}
\node at(9.5,1.1) {$D_5$};
\node at(10.2,0.53) {$D_4$};
}
\end{center}
\caption{The left diagram $D_L$ and the right diagram $D_R$ for the chord diagram $D$ in \cref{cd-example}.}
\label{fig:d_l}
\end{figure}

\begin{remark}[Recipe from a chord diagram]\label{rem:recipechord}
We now explain how a chord diagram $D$ gives rise to a recipe
$\rcp(D)$.  
Let $D$ be a chord diagram with $k$ chords on a set of markers $N$.  If $k=0$, $\rcp(D)$ is the trivial recipe on $N$.
Otherwise, let 
$(a_k, b_k, c_k, d_k)$ denote the rightmost top chord, 
let $n:=\max N$, and let 
$I_k:= \{p\in N\ \vert \ d_k<p<n\}.$ 
Let $\overline{D}$ be the chord diagram obtained from $D$ by removing the markers in $I_k$, and let $D_L$ and $D_R$ be the 
left and right subdiagrams of $\overline{D}$, on marker sets
 $N_L\subseteq N\setminus I_k$ and $N_R\subseteq N\setminus I_k$, respectively. 
Then the recipe $\rcp(D)$ from $D$ is recursively constructed
as the recipe $\rcp(D_L)$ followed by the recipe $\rcp(D_R)$ followed by the step-tuple $((a_k, b_k, c_k, d_k, n),
\pre_{I_k})$ on $N$.
\end{remark}

\begin{example}
We now illustrate \cref{rem:recipechord} on the chord diagram $D_L$ of \cref{ex:chordcell}, which is pictured in \cref{fig:d_l}.  In this case we obtain the recipe
$$\rcp^{\triv}_{\{1,2,15\}},\rcp^{\triv}_{\{2,3,4,15\}}, \rcp^{\triv}_{\{4,5,6,15\}}, ((3,4,5,6,15)), \rcp^{\triv}_{\{6,7,8,9,15\}}, ((5,6,8,9,15)),((1,2,8,9,15)).$$
\end{example}

\begin{figure}
	\includegraphics[width=\textwidth]{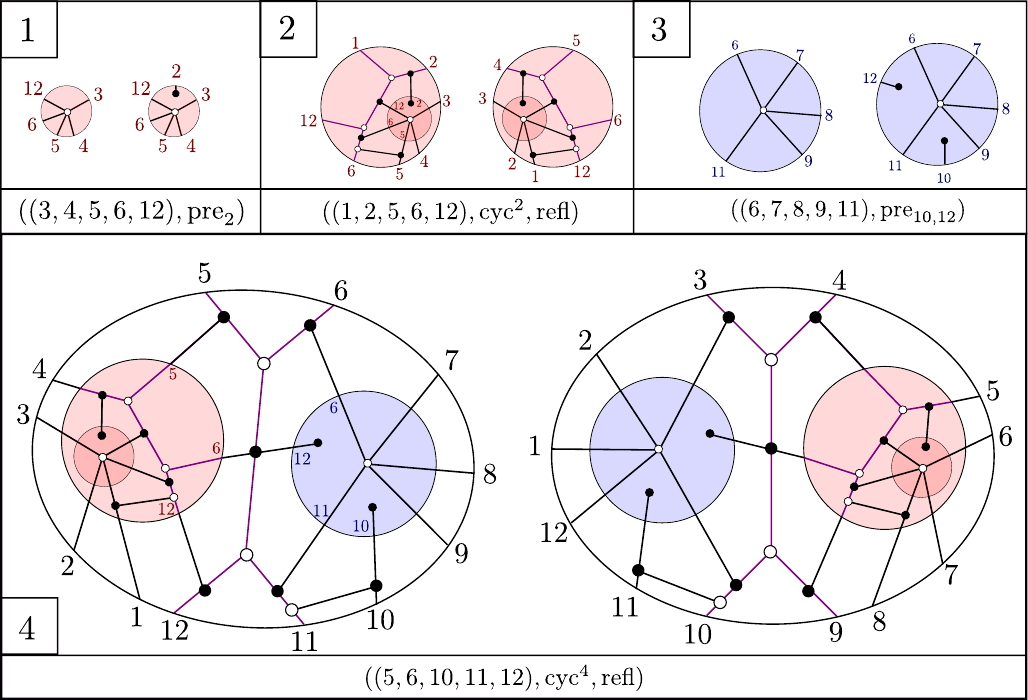}
	\caption{Illustration of building up a BCFW cell using the recipe $\rcp$ of \cref{ex:recipe}. Box $i$ shows the result after the first $i$ step-tuples. The result of the step $(a_i,b_i,c_i,d_i,n_i)$ is shown on the left in each box, and the results of the steps $\mbox{pre}_{I_i}, \mbox{cyc}^{r_i}$ and $\mbox{refl}^{s_i}$ are shown on the right.}
 
	\label{fig:bcfw_tile}
\end{figure}

Because our arguments are frequently recursive, we need some notation for the BCFW cells obtained by deleting the final step of a recipe. We use the following notation throughout.

\begin{notation}\label{not:L-and-R}
Let $\rcp$ be a recipe for a BCFW cell $S \in \Gr_{k, N}^{\scriptscriptstyle\geq 0}$. Let $\st$
denote the final step, which is either $(a_k,b_k,c_k,d_k,n_k), \pre_{I_k}, \cyc$ or $\refl$. If $\st \neq (a_k, b_k, c_k, d_k, n_k)$, then we let $\rcpp$ denote the recipe obtained by replacing $\st$ with the identity. Note that $S_{\rcpp}$ is again a BCFW cell. If $\st=(a_k, b_k, c_k, d_k, n_k)$, let $\rcp_L$ and $\rcp_R$ denote the recipes on $N_L$ and $N_R$ as in \cref{def:recipe}. Then $\rcp_L, \rcp_R$ are recipes for BCFW cells $S_L \subset \Gr_{k_L, N_L}^{\scriptscriptstyle\geq 0}$ and $S_R\subset \Gr_{k_R, N_R}^{\scriptscriptstyle\geq 0}$ and $S = S_L \bcfw S_R$. Note that to avoid clutter, we will usually use $L,R$ as subscripts rather than writing $S_{\rcp_L}, S_{\rcp_R}$.
\end{notation}

\begin{remark}
In contrast with the bijective correspondence between standard BCFW cells and chord diagrams, multiple recipes could give rise to the same general BCFW cell. 
Even the sets of 5 indices that are involved in the BCFW products are not uniquely determined by the resulting cell.
\end{remark}

\section{Background: cluster algebra and BCFW tiles}\label{background2}
In this section we review some of the connections between BCFW tiles and the cluster algebra of the Grassmannian $\Gr_{4,n}$. See e.g. \cite[Section 3]{even2023cluster} for a relevant review on cluster algebras.

\subsection{Product promotion}
A key ingredient for connecting BCFW tiles to cluster algebras is \emph{product promotion} -- a map which is the algebraic counterpart of the BCFW product.
\begin{definition}\label{def:product_promotion}
Using \cref{not:LR_cluster}, \emph{product promotion} is the homomorphism 
	$$\Psi_{B} = \Psi: \C(\widehat{\Gr}_{4,N_L})\times \C(\widehat{\Gr}_{4,N_R}) \to \C(\widehat{\Gr}_{4,n}),$$ 
induced 
by the following substitution:
\begin{equation*}
\text{on $\widehat{\Gr}_{4,N_L}$: } b \;\mapsto\; \frac{(ba)\cap (cdn)}{\lr{a\,c\,d\,n}}, 
\end{equation*}
\begin{equation*}
\text{on $\widehat{\Gr}_{4,N_R}$: } n \;\mapsto\;
\frac{(ba)\cap (cdn)}{\lr{a\,b\,c\,d}} 
\label{eq:promotionvectors2}, \, d \;\mapsto\; \frac{(dc)\cap (abn)}{\lr{a\,b\,c\,n}}.
\end{equation*}
\end{definition}
The vector $(ij)\cap (rsq):=v_i \lr{j \, r \, s\, q}-v_j \lr{i \, r \, s\, q}= - v_r \lr{i \, j \, s\, q}+v_s \lr{i \, j \, r\, q}-v_q \lr{i \, j \, r\, s} $ is in the intersection of the $2$-plane and the $3$-plane spanned by $v_i,v_j$ and $v_r,v_s,v_q$, respectively.

 \cref{thm:promotion2} below says\footnote{We will sometime omit the dependence on the indices $B=\{a,b,c,d,n\}$ in $\Psi$ (and $\rPsi$) for brevity.} that $\Psi$ is a \emph{quasi-homomorphism} from the cluster algebra\footnote{$\C[\widehat{\Gr}_{4,N_L}^{\circ}] \times \C[\widehat{\Gr}_{4,N_R}^{\circ}]$ is a cluster algebra where each seed 
is the disjoint union of a seed of each factor.} $\C[\widehat{\Gr}_{4,N_L}^{\circ}] \times \C[\widehat{\Gr}_{4,N_R}^{\circ}]$ to the cluster algebra $\C[\widehat{\Gr}_{4,n}^{\circ}]$.   See \cite[Definition 3.23]{even2023cluster} or \cite[Definition 3.1, Proposition 3.2]{Fraser} for the definition of a quasi-homomorphism.

\begin{theorem}\label{thm:promotion2}\cite[Theorem 4.7]{even2023cluster}
Product promotion 
$\Psi$ is a quasi-homomorphism of cluster
algebras. In particular, $\Psi$ maps a cluster variable (respectively, cluster)
of  $\C[\widehat{\Gr}_{4,N_L}^{\circ}] \times \C[\widehat{\Gr}_{4,N_R}^{\circ}]$, to a cluster variable (respectively, sub-cluster) of 
$\C[\widehat{\Gr}_{4,n}^{\circ}]$, up to multiplication by Laurent monomials in $\mathcal{T'}:=\{ 
\lr{a\,b\,c\,n},
\lr{a\,b\,c\,d},
\lr{b\,c\,d\,n}, 
\lr{a\,c\,d\,n}\}$. 
\end{theorem}

\begin{remark}
\cref{def:product_promotion} and \cref{thm:promotion2} extend also to the degenerate cases, e.g. for $a=1$ (\emph{upper promotion}), where $\Psi:\C(\widehat{\Gr}_{4,N_R}) \to \C(\widehat{\Gr}_{4,n})$, see \cite[Section 4.3]{even2023cluster}. 
\end{remark}

\begin{definition}\label{def:rPsi} Let $x$ be a cluster variable of $\C[\widehat{\Gr}_{4,N_L}^{\circ}]$ or $\C[\widehat{\Gr}_{4,N_R}^{\circ}]$. 
 We define the \emph{rescaled product promotion} $\rPsi(x)$ of $x$ to be the cluster variable of $\Gr_{4,n}$ obtained from $\Psi(x)$ by removing\footnote{If $x= \lr{bcdn}$, then $\rPsi(x)=\Psi(x)= x$.} the Laurent monomial in $\mathcal{T'}$ (c.f. \Cref{thm:promotion2}).
\end{definition}

The fact that product promotion is a cluster quasi-homomorphism may be of independent interest
in the study of the cluster structure on $\Gr_{4,n}$.
Much of the work thus far on the cluster structure of the Grassmannian
has focused on cluster variables which are polynomials in Pl\"ucker coordinates
with low degree; by contrast, the cluster variables we obtain can have arbitrarily high degree
in Pl\"ucker coordinates. We introduce the following notation:
\begin{equation}
	\lr{a\,b\,c \br d\,e \br f\,g\,h}  := 
\lr{a \, b \, c \, (d \, e) \cap (f \, g\, h)}=\lr{a\,b\,c\,d}\,\lr{e\,f\,g\,h} - \lr{a\,b\,c\,e}\,\lr{d\,f\,g\,h} \label{eq1:quadratic}.
\end{equation}
More generally, we consider polynomials called \emph{chain polynomials} of degree~$s+1$ as follows (see \cite[Definition 2.5]{even2023cluster}):
\begin{align}\label{eq:chainpols}
\begin{split}
	&  \lr{a_0\,b_0\,c_0 \br d_{1,0}\,d_{1,1} \br b_1\,c_1 \br d_{2,0}\,d_{2,1} \br b_2\,c_2 \br \dots \br d_{s,0}\,d_{s,1} \br b_s\,c_s\,a_s} \\[0.25em] 
	& 
	\;=\; \sum_{t \in \{0,1\}^s}\;(-1)^{t_1+\dots+t_s} \,\lr{a_0\,b_0\,c_0\,d_{1,{t_1}}}\, \lr{d_{1,{1-t_1}}\,b_1\,c_1\,d_{2,{t_2}}} \, \lr{d_{2,{1-t_2}}\,b_2\,c_2\,d_{3,{t_3}}} \, \cdots \, \lr{d_{s,{1-t_s}}\,b_s\,c_s\,a_s}
 \end{split}
\end{align}

\begin{example}\label{ex:psi}
For $N_L$ and $N_R$ as in \cref{ex:bcw_tile}, the only Pl\"ucker which changes is: $
  \Psi(\lr{1 \, 2 \, 4 \, 7})= \lr{1 \, 2 \, 7 | 3 \, 4| 5\, 6\, 7}/ \lr{3 \, 4\, 6\, 7}$, and $\rPsi(\lr{1 \, 2 \, 4 \, 7})=\lr{1 \, 2 \, 7 | 3 \, 4| 5\, 6\, 7}$ which is a quadratic cluster variable in $\mbox{Gr}_{4,7}$, e.g. obtained by mutating $\lr{2367}$ in the rectangle seed $\Sigma_{4,7}$ (see \cite[Definition 3.12]{even2023cluster}).   
\end{example}

\subsection{Coordinate cluster variables}
Using rescaled product promotion and  \cref{def:opGr}, 
we associate to each recipe ${\rcp}$ a collection of compatible cluster variables $\Irr(\rcp)$ for $\Gr_{4,n}$.  This will allow us to describe each (open) tile as the subset of the Grassmannian $\Gr_{k,k+4}$ where these cluster variables
take on particular signs.

\begin{definition}[Coordinate cluster variables of BCFW cells]\label{def:generalcluster}
Let $S_\rcp \subset \Grk$ be a BCFW cell. 
We use \cref{not:L-and-R}.  
The \emph{coordinate cluster variables} $\Irr(\rcp):=\{\rzeta^\rcp_i\}$ for $S_\rcp$ are defined recursively as follows:
\begin{itemize}
\itemsep0.25em
\item If $\st= (a, b, c, d, n)=:B$, then we define 
\vspace{1em}
\begin{equation*}
 \ralpha_k^{\rcp} \,:=\, \lr{b\, c \,d\, n}, \quad  \rbeta_k^{\rcp} \,:=\, \lr{a\, c\, d\, n}, \quad \rgamma_k^{\rcp} \,:=\,  \lr{a\, b\, d\, n}, \quad \rdelta_k^{\rcp} \,:=\, \lr{a\, b\, c\, n}, \quad \repsilon_k^{\rcp} \,:=\, \lr{a\, b\, c\, d}   \vspace{1em}
\end{equation*} 
 and for $i \neq k$, \quad
$\rzeta^\rcp_i \;:=\; \begin{cases}
\rPsi_{B}(\rzeta_i^{L})\\
\rPsi_{B}(\rzeta_i^{R})
\end{cases}$ if the $i$th step-tuple is in $\begin{cases}
  \rcp_L\\
  \rcp_R\\
\end{cases}.$
\vspace{1em}
\item If $\st=$ 
$\begin{cases}
    \refl \\
    \cyc \\
    \pre_{I_k}
\end{cases}$ then $\rzeta^\rcp_i := \begin{cases}
    \refl^* \rzeta^{\rcpp}_i  \\
    \cyc^{-*} \rzeta^{\rcpp}_i  \\
    \rzeta^{\rcpp}_i 
\end{cases}$.
\end{itemize}
Note that $\Irr(\rcp)$ depends on the recipe $\rcp$ rather than just the BCFW cell.
\end{definition} 

\begin{notation}\label{not:pluckfunc}
Given a cluster variable $x$ in $\Gr_{4,n}$, we will denote by $x(Y)$ the functionary on $\Gr_{k,k+4}$ obtained by identifying Pl\"ucker coordinates $\lr{I}$ in $\Gr_{4,n}$ with twistor coordinates $\llrr{I}$ in $\Gr_{k,k+4}$ (cf. \cref{def:tw_coords}). 
\end{notation} 

Interpreting each cluster variable as a functionary,
we describe each BCFW tile  as the semialgebraic subset of
$\Gr_{k,k+4}$ where the coordinate cluster variables take on particular signs. This appears as Corollary 7.12 in \cite{even2023cluster}:

\begin{theorem}[Sign description for general BCFW tiles]\label{cor:cluster-sign-description} 
	Let $\gt{\rcp}$ be a general BCFW tile. For each element $x$ of $\Irr(\rcp)$, the functionary $x(Y)$ has a definite sign $s_x$ on $\gto{\rcp}$ and 
	\[\gto{\rcp}= \{Y \in \Gr_{k,k+4}: s_x \, x(Y) >0 \text{ for all } x \in \Irr(\rcp) \}.\]
\end{theorem}

\begin{example}[Coordinate cluster variables] \label{ex:coord_clust}
The coordinate cluster variables for $S_{\rcp}$ in \cref{fig:bcfw_tile} are obtained by applying the recursion in \cref{def:generalcluster}:\vspace{1em}\\ 
\setlength{\tabcolsep}{2pt}
\hspace{-1em} 
\begin{tabular}{|c|| c| c| c|c|c|}
 \hline
 $i$ & $\ralpha_i$ & $\rbeta_i$  & $\rgamma_i$ & $\rdelta_i$ & $\repsilon_i$\\ [0.5ex] 
 \hline \hline
 $1$ & $ \lr{7\,8\,9 \br 4 \,3 \br 9\,A\,B}$ & $ \lr{6\,8\,9 \br 4\,3 \br 9 \,A \, B}$ & $ \scriptstyle \lr{9\,A\,B \br 3\,4 \br 6\, 7 \br 8\,9 \br 3 \,4 \, 5}$ & $
    \scriptstyle \lr{6\,7\,8 \br 4\,5 \br 8\, 9 \br 3\,4 \br 9 \,A \, B}$ & $ \lr{6\,7\,8\,9}$ \\
 \hline
 $2$ & $\lr{5\,8\,9 \br 4 \,3 \br 9\,A\,B}$ & $
      \lr{3\,4\,8\,9}$ & $
     \lr{3\,4\,5\,9}$ & $
   \lr{3\,4\,5\,8}$ & $
 \lr{4\,5\,8\,9}$ \\
 \hline
 $3$ & $\lr{1\,2\,C \br B \,A \br 3\,4\,9}$ & $
     \lr{1\,3\,C \br A\,B \br 3\,4\,9}$ & $
     \lr{2\,3\,C \br A \,B \br 3\,4\,9}$ & $
    \lr{1\,2\,3 \br B \,A \br 3\,4\,9}$ & $
 \lr{1\,2\,3\,C}$\\
  \hline
 $4$ & $\lr{3 \, 9\,A \,B}$ & $ \lr{4 \, 9\,A \,B} $ & $\lr{3 \, 4 \,9\,A}$ & $ \lr{3 \, 4\,9 \,B}$ & $ \lr{3 \, 4\,A \,B}$\\
  \hline 
\end{tabular}
\vspace{1em} \\ See \cite[Example 7.4]{even2023cluster} for more details.
\end{example}

\subsection{BCFW tiles}\label{sec:bcfw_tiles}

In \cite[Section 7]{even2023cluster} we proved that 
BCFW cells give tiles of the amplituhedron $\mathcal{A}_{n,k,4}(Z)$ by explaining how  to invert the amplituhedron map $\tilde{Z}$ on the image $\gto{\rcp} = \tZ(S_{\rcp})$ of each BCFW cell $S_{\rcp}$.  For each point $Y \in \gto{\rcp}$, the pre-image $\tilde{Z}^{-1}(Y)$ is a point in $\Grk$ represented by the \emph{twistor matrix} $\twmt_\rcp(Y)$, whose entries are expressed in terms of ratios of the \emph{coordinate functionaries} $\{\czeta^\rcp_i (Y)\}_{i=1}^{5k}$ of $S_{\rcp}$, see \cite[Definition 7.1]{even2023cluster}. The coordinate functionaries are defined recursively in a similar way as in \cref{def:generalcluster} using product promotion. Moreover, they can be used to give a semilagebraic description of the tile. This is summarized in the theorem below, which appears as \cite[Theorem 7.7]{even2023cluster}.

\begin{theorem}[General BCFW cells give tiles] \label{thm:BCFW-tile-and-sign-description}
	Let 
		$S_\rcp$ be a general BCFW cell with recipe
		$\rcp$.  Then
	for all $Z \in \Mat_{n,k+4}^{>0}$, $\tZ$ is injective on 
		$S_\rcp$ and thus $\gt{\rcp}$ is a tile. 
		In particular,
		given $Y \in \tZ(S_\rcp)$,  the unique preimage of $Y$ in $S_\rcp$ is given by (the rowspan of) of the twistor matrix $\twmt_\rcp(Y)$. 
		Moreover,
		$$\gto{\rcp}= \{Y \in \Gr_{k, k+4}: \czeta_i^\rcp(Y)>0 \text{ for all coordinate functionaries of }S_\rcp\}.$$
\end{theorem}

For functionaries, we can introduce a similar notation as for the chain polyonmials in \cref{eq1:quadratic}:
\begin{equation}\label{eq:quadratic}
\llrr{a\,b\,c \br d\,e \br f\,g\,h} \;=\; 
\llrr{a\,b\,c\,d}\,\llrr{e\,f\,g\,h} - \llrr{a\,b\,c\,e}\,\llrr{d\,f\,g\,h}. 
	\end{equation}
More generally, we define \emph{chain functionaries} of degree~$s+1$ to be the polynomials obtained from \cref{eq:chainpols} by replacing Pl\"ucker coordinates $\lr{I}$ by twistor coordinates $\llrr{I}$. See \cite[Definition 2.19]{even2023cluster}.

\begin{example}[Coordinate functionaries] \label{ex:coord_func}
The coordinate functionaries for $S_{\rcp}$ in \cref{fig:bcfw_tile} are:

\setlength{\tabcolsep}{2pt}
\hspace{-1em} 
\begin{tabular}{|c|| c| c| c|c|c|}
 \hline
 $i$ & $\alpha_i(Y)$ & $\beta_i(Y)$  & $\gamma_i(Y)$ & $\delta_i(Y)$ & $\epsilon_i(Y)$\\ [0.5ex] 
 \hline \hline
 $1$ & $ \scriptstyle \frac{\llrr{7\,8\,9 \br 4 \,3 \br 9\,A\,B}}{\llrr{3\,4\,9\,A}}$ & $\scriptstyle -\frac{\llrr{6\,8\,9 \br 4\,3 \br 9 \,A \, B}}{\llrr{4\, 9 \, A \, B}}$ & $ \scriptstyle  \frac{\llrr{9\,A\,B \br 3\,4 \br 6\, 7 \br 8\,9 \br 3 \,4 \, 5}}{\llrr{3\,4\,5 \,8} \llrr{4\, 9 \, A \, B}}$ & $
    \scriptstyle -\frac{\llrr{6\,7\,8 \br 4\,5 \br 8\, 9 \br 3\,4 \br 9 \,A \, B}}{\llrr{4\,5\,8 \,9} \llrr{4\, 9 \, A \, B}}$ & $\llrr{6\,7\,8\,9}$ \\
 \hline
 $2$ & $\scriptstyle -\frac{\llrr{5\,8\,9 \br 4 \,3 \br 9\,A\,B}}{\llrr{4\,9\,A\,B}}$ & $
    \llrr{3\,4\,8\,9}$ & $
    \llrr{3\,4\,5\,9}$ & $
   -\llrr{3\,4\,5\,8}$ & $
 \llrr{4\,5\,8\,9}$ \\
 \hline 
 $3$ & $\scriptstyle -\frac{\llrr{1\,2\,C \br B \,A \br 3\,4\,9}}{\llrr{3\,4\,9\,B}}$ & $ \scriptstyle -\frac{\llrr{1\,3\,C \br A\,B \br 3\,4\,9}}{\llrr{3\,4\,9\,B}}$ & $ \scriptstyle \frac{\llrr{2\,3\,C \br A \,B \br 3\,4\,9}}{\llrr{3\,4\,9\,B}}$ & $ \scriptstyle -\frac{\llrr{1\,2\,3 \br B \,A \br 3\,4\,9}}{\llrr{3\,4\,9\,B}}$ & $
 \llrr{1\,2\,3\,C}$\\
  \hline
 $4$ & $-\llrr{3 \, 9\,A \,B}$ & $ \llrr{4 \, 9\,A \,B} $ & $\llrr{3 \, 4 \,9\,A}$ & $ -\llrr{3 \, 4\,9 \,B}$ & $ \llrr{3 \, 4\,A \,B}$\\
  \hline 
\end{tabular}

See \cite[Example 7.2]{even2023cluster} for more details.
\end{example}

For a standard BCFW tile $\gt{D}$, we call the coordinate cluster variables \emph{domino cluster variables} or simply \emph{domino variables}, and denote them as $\xx(D)=\{\ralpha_i, \rbeta_i, \rgamma_i, \rdelta_i, \repsilon_i\ \vert \ 1 \leq i \leq k\}$. See
\cite[Theorem 8.4]{even2023cluster} for explicit formulas for the domino variables. The formulas have different cases depending on whether certain chords are head-to-tail siblings, same-end parent and child, or sticky parent and child (cf. terminology in \cref{cd-terminology}).

\begin{example}[Domino cluster variables] 
\label{domino-variables-formulas}
The domino cluster variables $\xx(D)$ for the chord diagram $D$ in \cref{cd-example} are as follows. We will denote $(10,11,12,13,14,15)$ as $(A,B,C,D,E,F)$. \vspace{1em} \\
\setlength{\tabcolsep}{2pt}

\begin{tabular}{|c|| c| c| c|c|c|}
 \hline
 $i$ & $\ralpha_i$ & $\rbeta_i$  & $\rgamma_i$ & $\rdelta_i$ & $\repsilon_i$\\ [0.5ex] 
 \hline \hline
 $1$ & $\lr{4\,5\,6\,|\,2\,1\,|\,8\,9\,F}$ & $\lr{3\,5\,6\,|\,2\,1\,|\,8\,9\,F}$ & $\lr{F\,8\,9\,|\,2\,1\,|\,3\,4\,|\,5\,6\,|\,8\,9\,F}$ & $\lr{3\,4\,5\,|\,2\,1\,|\,8\,9\,F}$ & $\lr{3\,4\,5\,6}$ \\
 \hline
 $2$ & $\lr{6\,8\,9\,F}$ & $\lr{5\,8\,9\,F}$ & $\lr{F\,1\,2\,|\,5\,6\,|\,8\,9\,F}$ & $\lr{5\,6\,8\,|\,2\,1\,|\,8\,9\,F}$ & $\lr{5\,6\,8\,9}$ \\
 \hline
 $3$  & $\lr{2\,8\,9\,F}$ & $ \lr{1\,8\,9\,F}$ & $ \lr{F\,1\,2\,|\,8\,9\,|\,D\,E\,F}$ & $ \lr{1\,2\,8\,F}$ & $ \lr{1\,2\,8\,9}$\\
  \hline
 $4$ & $\lr{B\,C\,D\,|\,9\,8\,|\,D\,E\,F}$ & $ \rbeta_4 = \ralpha_5$ & $ \lr{8\,9\,A\,B}$ & $ \lr{9\,A\,B\,C}$ & $ \lr{A\,B\,C\,D}$\\
  \hline
$5$ & $\lr{A\,C\,D\,|\,9\,8\,|\,D\,E\,F}$ & $ \lr{8\,9\,C\,D}$ & $ \lr{8\,9\,A\,D}$ & $ \lr{8\,9\,A\,C}$ & $ \lr{9\,A\,C\,D} $ \\
\hline
 $6$ &  $\lr{9\,D\,E\,F}$ & $\lr{8\,D\,E\,F}$ & $\lr{8\,9\,E\,F}$ & $\lr{8\,9\,D\,F}$ & $\lr{8\,9\,D\,E}$\\
  \hline
\end{tabular}
\vspace{1em} \\
See \cite[Example 8.5]{even2023cluster} for more details.

\end{example}

\begin{definition}[Mutable and frozen domino variables]\label{def:frozmut}
Let $D \in \mathcal{CD}_{n,k}$ be a chord diagram, corresponding to a
standard BCFW tile $Z_D$ in $\A_{n,k,4}(Z)$.
Let $\AFacet(Z_D)$ denote
the  following collection of domino cluster variables:
        \begin{itemize}
                \itemsep0.25em
                \item
                $\ralpha_i$ unless $D_i$ has a sticky child
                \item
                 $\rbeta_i$ unless $D_i$ starts where another chord ends or $D_i$ has a same-end sticky parent. 
\item   $\rgamma_i$ in all cases. 
   \item
                $\rdelta_i$ unless $D_i$ has a same-end child.
                \item
                $\repsilon_i$ unless $D_i$ has a same-end child.
        \end{itemize}
Let $\Mut(Z_D)$ denote the complementary set of domino variables, i.e. $\Mut(Z_D)=\xx(D) \setminus \AFacet(Z_D)$.
\end{definition}

\begin{remark}\label{rmk:8_2}
One can show (see \cite[Remark 8.2]{even2023cluster}
)
that if $D_i$ has a same-end sticky parent $D_p$, then $\rbeta_i=\ralpha_p$.    
\end{remark}

\begin{example}[Mutable and frozen domino variables]
Let $Z_D$ be the tile with the chord diagram~$D$ from 
\cref{cd-example} and domino variables as in \cref{domino-variables-formulas}. Among those, the mutable variables
are:
$$ \ralpha_5,\, \ralpha_6,\, \rbeta_2,\, \rbeta_4,\, \rbeta_6,\, \rdelta_3,\, \rdelta_5,\, \repsilon_3,\, \repsilon_5 \;\in\; \Mut(Z_D). $$
Hence
$\AFacet(Z_D)$ consists of the
remaining $21$ domino variables. Note that $\ralpha_5 = \rbeta_4$ by \cref{rmk:8_2}.
\end{example}

\begin{definition}[The seed $\Sigma_D$ of a BCFW tile $Z_D$]
\label{def:seed}
Let $D \in \CD$ be a chord diagram,
and $\gt{D}$ the corresponding BCFW tile.
We define a seed $\Sigma_D=(\xx(D), Q_D)$
as follows.
The extended cluster $\xx(D)$ has the sets $\Mut(Z_D)$ of mutable cluster variables and $\AFacet(D)$ of frozen variables (recall \cref{def:frozmut}).
To obtain the quiver $Q_D$, we consider each chord $D_i$ in turn, 
check if it satisfies any 
of the conditions in the table below, and if so, we draw the corresponding arrows.
\vspace{0.5em}
\begin{center}
\begin{tabular}{|c|c|c|c|}
\hline
\rotatebox{90}{Condition \quad\;}&
\tikz[line width=1,scale=1]{
\def\r{1}
\draw (1*\r,0) -- (1.5*\r,0);
\draw[dashed] (1.5*\r,0) -- (2.75*\r,0);
\draw (2.75*\r,0) -- (3.5*\r,0);
\draw[dashed] (3.5*\r,0) -- (5*\r,0);
\draw (5*\r,0) -- (5.5*\r,0);
\foreach \i in {6,7}{
\def\x{\i/2*\r}
\draw (\x,-0.1)--(\x,+0.1);
}				
\foreach \i/\j in {2/5.833,6.166/10}{
\def\x{\i/2*\r+0.25*\r}
\def\y{\j/2*\r+0.25*\r}
\draw[line width=1,-stealth] (\x,0) -- (\x,0.15) to[in=90,out=90] (\y,0.15) -- (\y,0);}
\node at(2.2*\r,1*\r) {$j$};
\node at(4.3*\r,1*\r) {$i$};
}&
\tikz[line width=1,scale=1]{
\draw (0.75,0) -- (1.25,0);
\draw[dashed] (1.25,0) -- (2.25,0);
\draw (2.25,0) -- (2.75,0);
\draw[dashed] (2.75,0) -- (4,0);
\draw (4,0) -- (5,0);
\foreach \i in {8.5,9.5}{
\def\x{\i/2}
\draw (\x,-0.1)--(\x,+0.1);
}
\foreach \i/\j in {1/4.5+0.066, 2.5/4.5-0.066}{
\def\x{\i}
\def\y{\j}
\draw[line width=1,-stealth] (\x,0) -- (\x,0.15) to[in=90,out=90] (\y,0.15) -- (\y,0);}
\node at(2.4,0.6) {$j$};
\node at(2.7,1.5) {$i$};
}&
\tikz[line width=1,scale=1]{
\draw (0.5,0) -- (1.75,0);
\draw[dashed] (1.75,0) -- (3.25,0);
\draw (3.25,0) -- (3.75,0);
\draw[dashed] (3.75,0) -- (4.75,0);
\draw (4.75,0) -- (5.25,0);
\foreach \i in {1.5,2.5,3.5}{
\def\x{\i/2}
\draw (\x,-0.1)--(\x,+0.1);
}
\foreach \i/\j in {1/5, 1.5/3.5}{
\def\x{\i}
\def\y{\j}
\draw[line width=1,-stealth] (\x,0) -- (\x,0.15) to[in=90,out=90] (\y,0.15) -- (\y,0);}
\node at(3,1) {$j$};
\node at(2,1.5) {$i$};
}\\
& head-to-tail left sibling $D_j$
& same-end child $D_j$
& sticky child $D_j$
\\[0.5em]
\hline
\rotatebox{90}{\quad Arrows \quad\;\;}&
\tikz[line width=0.75,scale=1,minimum size=18pt,inner sep=0pt, outer sep=0pt,fill=lightgray!25]{
\node        (c1) at (5,5) {$\rgamma_j$};
\node        (d1) at (6,5) {$\rdelta_j$};
\node        (a2) at (5,4) {$\ralpha_i$};
\node[fill,draw,circle] (b2) at (6,4) {$\rbeta_i$};
\path[very thick,->] (b2) edge (d1);
\path[very thick,->] (c1) edge (b2);
\path[very thick,->] (b2) edge (a2);
}
&
\tikz[line width=0.75,scale=1,minimum size=18pt,inner sep=0pt, outer sep=0pt,fill=lightgray!25]{
\node        (c5) at (12,2) {$\rgamma_i$};
\node[fill,draw,circle] (d5) at (13,2) {$\rdelta_i$};
\node[fill,draw,circle] (e5) at (15,2) {$\repsilon_i$};
\node        (c4) at (12,1) {$\rgamma_j$};
\node        (d4) at (13,1) {$\rdelta_j$};
\node        (e4) at (15,1) {$\repsilon_j$};
\path[very thick,->] (e5) edge (d5);
\path[very thick,->] (d5) edge (c5);
\path[very thick,->] (c4) edge (d5);
\path[very thick,->] (d4) edge (e5);
\path[very thick,->] (d5) edge (d4);
\path[very thick,->] (e5) edge (e4);
}
&
\tikz[line width=0.75,scale=1,minimum size=18pt,inner sep=0pt, outer sep=0pt,fill=lightgray!25]{
\node[fill,draw,circle] (a5) at (9,2) {$\ralpha_i$};
\node (b5) at (10,2) {$\rbeta_i$};
\node (e5) at (6.5,2) {$\repsilon_i$};
\node (a4) at (10,1) {$\ralpha_j$};
\node (e4) at (6.5,1) {$\repsilon_j$};
\node[gray] (text) at (7.5,2.5) {if same-end};
\path[very thick,->,dotted] (a5) edge (e5);
\path[very thick,->] (e4) edge (a5);
\path[very thick,->] (b5) edge (a5);
\path[very thick,->] (a5) edge (a4);
}
\\[0.5em]
\hline
\end{tabular}
\end{center}
\vspace{0.5em}
If $D_i$ has sticky same-end child $D_j$ then the dotted arrow from $\ralpha_i$ to $\repsilon_i$ appears, along with the usual arrows of the ``sticky'' and ``same-end'' cases. In view of \cref{rmk:8_2}, 
in this case $\ralpha_i$ stands also for $\rbeta_j$ as they are equal.   
\end{definition}

\begin{figure}[h]
\begin{center}
\tikz[line width=1,scale=1]{
\draw (0.5,0) -- (15.5,0);
\foreach \i in {1,2,...,15}{
\def\x{\i}
\draw (\x,-0.1)--(\x,+0.1);
\node at (\x,-0.5) {\i};}
\foreach \i/\j in {1/8, 3/4.8, 5.2/7.75, 8.25/13, 9/12.2, 10/11.8}{
\def\x{\i+0.5}
\def\y{\j+0.5}
\draw[line width=1.5,-stealth] (\x,0) -- (\x,0.25) to[in=90,out=90] (\y,0.25) -- (\y,0);
}
\node at(1.5,1.5) {$D_3$};
\node at(3.5,0.875) {$D_1$};
\node at(5.75,1) {$D_2$};
\node at(9,1.5) {$D_6$};
\node at(9.5,1) {$D_5$};
\node at(10.5,0.8125) {$D_4$};
}
\\
\vspace{1em}
\includegraphics[width=0.5\textwidth]{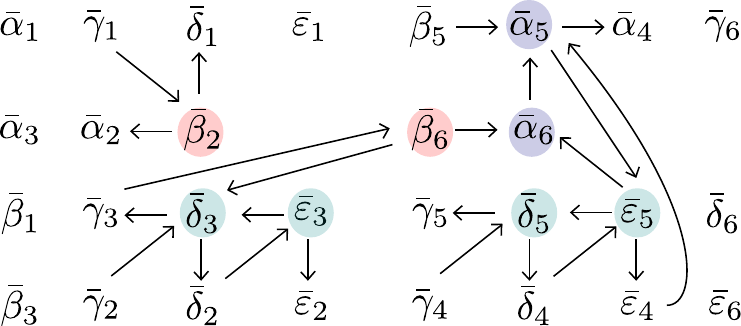}
\end{center}

\caption{The seed $\Sigma_D$ associated to the chord diagram $D$ above (also in \cref{cd-example}). The variables $\xx(D)$
	 are as in \cref{domino-variables-formulas}. 
	 The mutable variables $\Mut(Z_D)$ are circled; 
	 the other variables are the frozen variables $\AFacet(Z_D)$.
	 The colors (red, green, blue) indicate the different cases of \cref{def:seed}. 
\label{cd-example-again}}
\vspace{1em}
\end{figure}
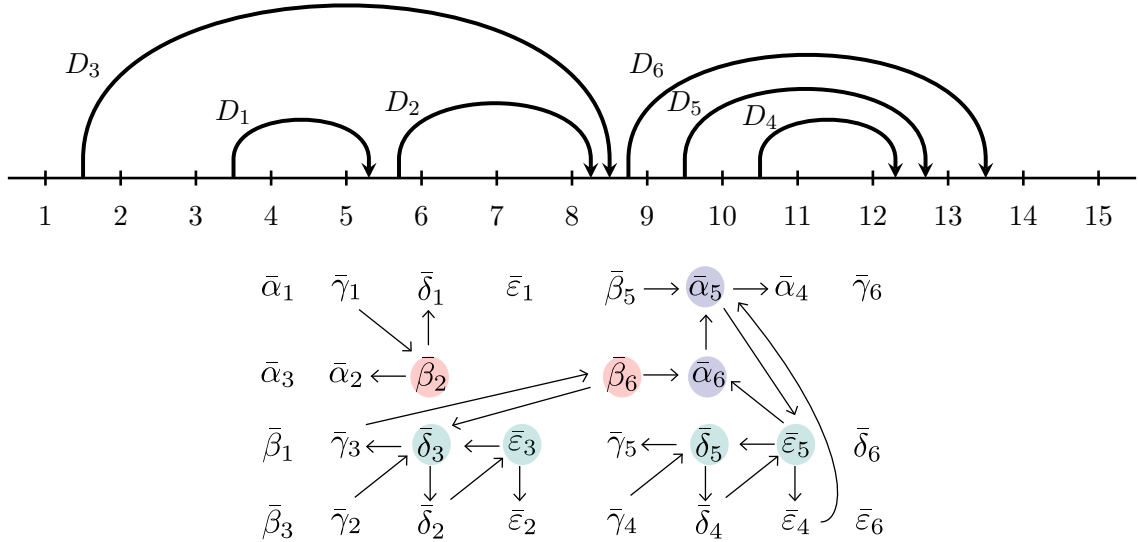

\begin{example}[Seed of a standard BCFW tile]
The seed $\Sigma_D$ from \cref{cd-example-again} is built from \cref{def:seed} by applying the rules for the following conditions. Head-to-tail left siblings: $(i,j) \in \{(2,1),(6,3)\}$; same-end child: $(i,j) \in \{(3,2),(5,4)\}$; sticky child: $(i,j) \in \{(6,5),(5,4)\}$.
\end{example}

\cref{thm:quiver}
appears as Theorems 9.10 in
\cite{even2023cluster}.
	\begin{theorem}[The seed of a standard BCFW tile  is a subseed of a $\Gr_{4,n}$ seed]
	\label{thm:quiver}
Let $D \in \CD$. The seed $\Sigma_D=(\xx(D), Q_D)$ is a subseed of a seed for $\Gr_{4,n}$.
Hence every cluster
variable (respectively, exchange relation) of 
$\Acal(\Sigma_D)$ is a cluster variable (resp., exchange relation) for
$\Gr_{4,n}$.

\end{theorem}

The following theorem characterizes the open BCFW tile
$\gto{D}$ in terms of \emph{any} extended cluster of $\Acal(\Sigma_D)$. It generalizes \cref{cor:cluster-sign-description} for standard BCFW tiles and it appears as Theorem 9.11 in \cite{even2023cluster}.
\begin{theorem}
	[Positivity tests for standard BCFW tiles]
	\label{thm:sign-definite}
Let $D \in \CD$. 
Using \cref{not:pluckfunc}, every cluster and frozen variable $x$ in 
 $\A(\Sigma_D)$ is such that $x(Y)$ has a definite sign $s_x \in \{1, -1\}$ on the open BCFW tile $\gto{D}$, and
 \begin{equation}
\gto{D} = \{Y \in \Gr_{k, k+4}: s_x \cdot x(Y)>0 \text{ for all }x \text{ in any fixed extended cluster of } \A(\Sigma_D)\}.
\end{equation}
\end{theorem}
The signs of the domino variables in \cref{thm:sign-definite} are given by \cite[Proposition 8.10]{even2023cluster}.

\begin{example}[Positivity test for a standard BCFW tiles] \label{ex:stdBCFW_postest}
    For the tile $Z_D$ with chord diagram $D$ in \cref{cd-example-again} and $\xx(D)$ as in \cref{domino-variables-formulas}:
    \begin{equation*}
        \gto{D} = \{Y \in \Gr_{6,10}: s_x \cdot x(Y)>0 \text{ for all }x \in \xx(D)\},
    \end{equation*}
    where the signs $s_x$ are negative if $x$ is among:
$ \ralpha_2,\; \ralpha_3,\;  
\ralpha_5 = 
\rbeta_4,\;  \rbeta_1,\;  \rbeta_6,\;  \rgamma_2,\;  \rdelta_1,\;  \rdelta_5,\;  \rdelta_6. $
Otherwise, $s_x$ is positive.
\end{example}

The following result appears as \cite[Theorem 7.16]{even2023cluster}.
\begin{theorem}[Cluster adjacency for general BCFW tiles]
	\label{thm:clusteradjacency}
	Let $\gt{\rcp}$ be a general BCFW tile of 
	$\Ank$. Each facet $\gt{S}$ of $\gt{\rcp}$ lies on a hypersurface cut out by a functionary $F_S(\llrr{I})$ such that $F_S(\lr{I}) \in \Irr(\rcp)$. Thus $\{F_S(\lr{I}): \gt{S} \text{ a facet of }\gt{\rcp}\}$ consists  of compatible cluster variables of $\Gr_{4,n}$.
\end{theorem}

\section{Facets of BCFW tiles}\label{sec:facets}

The main goal of this section is 
to prove 
 \cref{prop:standard_facets}, which characterizes the facets 
of standard  BCFW tiles; this proof is in 
 \cref{sec:standardfacets} and \cref{sec:trans}.
Then in \cref{sec:BCFWfacets}
we also state (without proof) a characterization of the
facets of general BCFW tiles.

\subsection{Facets of standard BCFW tiles}\label{sec:standardfacets}

\begin{theorem}[Frozen variables as facets]\label{prop:standard_facets}
Let $D \in \mathcal{CD}_{n,k}$ be a chord diagram, corresponding to a
standard BCFW tile $Z_D$ in $\Ank$.
Then for each cluster variable $\rzeta_i\in \AFacet(Z_D)$ (cf. \cref{def:frozmut})
there is a unique facet of $\gt{D}$
which lies in the zero locus of the functionary $\rzeta_i(Y)$; the plabic graph
of this facet is constructed in \cref{thm:showingreduced}.  
Moreover, for 
any $Z$, there are no other facets of $\gt{D}$.
\end{theorem}

We need several lemmas in order to 
prove \cref{prop:standard_facets}.
The first two are consequences of the Cauchy-Binet formula for the twistors (see, e.g., \cite[Lemma 2.16]{even2023cluster}). 
We recall the notion of \emph{coindependence} (\cite[Definition 5.5]{even2023cluster})
\begin{definition}\label{def:coind}
Let $V \in \Gr^{\geq 0}_{k,n}$. A subset $I \subseteq [n]$ is \emph{coindependent for $V$} if $V$ has a nonzero Pl\"ucker coordinate $\lr{J}_V$, such that $J \cap I = \emptyset$. If $k=0$ we declare all subsets to be coindependent. If $S$ is a positroid cell in $ \Gr^{\geq 0}_{k,n}$, then $J$ is \emph{coindependent for} $S$ if $J$ is coindependent for the elements of $S$. 
\end{definition}

\begin{lemma}\label{lem:coin1}
Let $I=\{i_1,\dots,i_m\} \in {[n] \choose m}$.
If       $\llrr{CZ, Z_{i_1}, \dots ,Z_{i_m}}\neq 0$, then 
 $I$ must be coindependent for $C \in  \Grk$.
\end{lemma}
\begin{proof}
If $\llrr{CZ, Z_{i_1}, \dots ,Z_{i_m}}\neq 0$, then 
 by the second equation of \cite[Lemma 2.16]{even2023cluster}, there must be 
some $J$ such that
$\lr{J}_C \neq 0$ and $J \cap I = \emptyset$.
This means that $I$ is coindependent for $C$.
\end{proof}

\begin{definition}[{\cite[Definition 11.1]{even2023cluster}}]
We say that functionary $F$ has a \emph{strong sign} on a positroid cell $S$ if there exists an expansion of $F(\tZ(C)),$ for $C\in S,$ as a sum of monomials in the Pl\"ucker coordinates of $C$ and the minor determinants of $Z$ all of whose coefficients have the same sign.    
\end{definition}

\begin{lemma}\label{lem:coin2}
Let $I \in {[n] \choose 4}$, and let $S$ be a cell of $\Grk$.
Suppose that $\llrr{I}$ has a strong sign on $\gto{S}$, but for some
cell $S' \subset \overline{S}$, we have
$\llrr{I} = 0$ on $Z_{S'}$.  Then for each
$J \in {[n] \choose k}$ disjoint from $I$, we must have
$\lr{J}_C=0$ for all $C\in S'$.  In other words,
$I$ is not coindependent for $S'$.
\end{lemma}

\begin{proof}
Since 
$\llrr{I}$ has a strong sign on $Z_S$, all nonzero terms of \cite[Lemma 2.16]{even2023cluster}, which necessarily come from $J$ for which $J$ and $I$ are disjoint, must 
have the same sign.
Since $\llrr{I} = 0$ on $Z_{S'}$, all the above nonzero terms must 
vanish when we go to the cell $S'$ in the boundary of $S$.
	But this means that all Pl\"ucker coordinates $\lr{J}$, with 
	$J$ disjoint from $I$, must vanish on $S'$.
\end{proof}

\begin{lemma}\label{lem:coin3}
Let $S_L \subset {\Gr}^{\ge0}_{k_L, N_L}$ and $S_R \subset {\Gr}^{\ge0}_{k_R, N_R}$ be positroid cells, with plabic graphs $G_L$ and $G_R$.
Let $G = G_L \bcfw G_R$.
 If 
$\{a, b, n\}$ fails to be coindependent for $S_L$
or
$\{b,c,d,n\}$ fails to be coindependent for $S_R$, then 
for each $I \in \binom{\{a,b,c,d,n\}}{4}$,
	we have $\llrr{I}=0$ on $\gt{G}$.
\end{lemma}
\begin{proof}
We will prove the contrapositive.
Suppose that 
for some $I \in \binom{\{a,b,c,d,n\}}{4}$,
we have $\llrr{I}\neq 0$ on $Z_{S_G}$.
Then by \cref{lem:coin1}, $I$ must be coindependent
for the cell $S_G$.  Then by \cite[Remark 5.6]{even2023cluster}, 
the plabic 
graph $G$ must have a perfect orientation $\O$ where
all boundary vertices in $I$ are sinks.  But now it is a simple
exercise to check that if in the graph $G_L \bcfw G_R$ which appears in \cref{fig:extra1} (ignoring the arrows) 
we put sinks at the
(outer) boundary vertices $I$, then there is a \emph{unique}
way to complete this to a perfect orientation of the 
``butterfly'' portion of the graph.  And in particular, 
this orientation will include the directed edges shown in 
\cref{fig:extra1}.  But then the perfect orientation $\O$,
restricted to $G_L$ and $G_R$, must have sinks
 at vertices $a,b,n$ of $G_L$, and at vertices $b,c,d,n$ of $G_R$.
	But then $\{a,b,n\}$ 
	and $\{b,c,d,n\}$ 
	must be coindependent for $S_L$
	and $S_R$, respectively.
\begin{figure}[h]
\centering
\includegraphics[height=2in]
{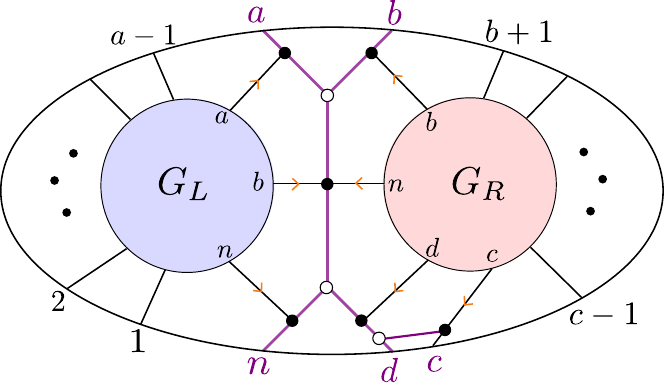}
\caption{}
\label{fig:extra1}
\end{figure}
\end{proof}

\begin{lemma}\label{lem:maps_to_bdry}
For every cell $S \subseteq \partial S_D$ 
in the boundary of a standard BCFW cell $S_D,$
\[\gt{S}\subseteq\partial\gt{D}.\]  
So $\gto{D}$ is the interior of $\gt{D}$ and 
$\partial \gto{D} = \partial \gt{D} = \tZ(\partial S_D)$.
\end{lemma}
\begin{proof} 
The second and third statements follows from the first, using 
	\cite[Corollary 11.17]{even2023cluster}.

We now focus on proving the first statement.
It is enough to prove it for facets, since images of boundary cells of higher codimensions are contained in the closure of the images of facets.  By \cite[Proposition 7.10]{even2021amplituhedron}, each facet $S$ of $S_D$ 
	  is either a facet of another BCFW cell $S_{D'}$ or its image $\gt{S}$ lies in the zero locus of a twistor coordinate $\llrr{i,i+1,j,j+1}$ for some $i, j$.

In the former case it follows that for every $p\in\gto{S},$ every open neighborhood of $p$ intersects both $\gto{D}$ and $\gto{D'}.$  
By \cite[Theorem 1.4]{even2021amplituhedron}, which shows that the images of different standard BCFW cells do not intersect, we have that
	$\gto{D}\cap\gto{D'}=\emptyset$. 
	Therefore $\gto{S}$ is indeed in the topological boundary of $\gt{D}.$

For the latter case, \cite[Proposition 8.1]{even2021amplituhedron} shows that the intersection of the hypersurface $\{\llrr{i,i+1,j,j+1}=0\}$ with $\Ank$ is contained in the topological boundary $\partial \Ank$. Hence if $\gt{S}$ lies on this hypersurface, $\gt{S}$ must also be contained in the topological boundary of $\gt{D}.$
\end{proof}

\begin{lemma}\label{lem:mutable_vars_give_codim_2}
	Let $D$ be a standard BCFW cell, and let $\xi_1,\xi_2 \in \Irr(D)$ be two different domino cluster  variables for $Z_D$. Then the intersection of zero loci of $\xi_1(Y),\xi_2(Y)$ (the natural identification between functionaries and homogenous polynomials in Pl\"ucker coordinates is explained in \cite[Notation 7.11]{even2023cluster}) meets $\gt{D}$ in codimension greater than $1$. It follows 
that for each mutable cluster variable $\xi \in \Mut(D)$, the zero locus of $\xi(Y)$ intersects $\gt{D}$ in codimension greater than one. 
\end{lemma}

The proof of Lemma \ref{lem:mutable_vars_give_codim_2} is postponed to the next subsection.

\begin{theorem}\label{thm:twistorvanish}
Let $S = S_L \bcfw S_R$ be a BCFW cell, and 
suppose $I \in {\{a,b,c,d,n\} \choose 4}$.
Then there is at most one 
facet $S'$ of $S,$ 
such that among the five twistor coordinates coming from
${\{a,b,c,d,n\} \choose 4}$, only 
$\llrr{I}$ vanishes on $Z_{S'}$.
To construct the potential facet, 
we start from the graph in \cref{fig:extra2} and 
remove the edge labeled by $x_{12}$ (respectively,
$x_{10}$, $x_{6}$, $x_{8}$, $x_1$), obtaining a graph 
$G^{(i)}$ corresponding to a cell $S^{(i)}$ (for $1 \leq i \leq 5$)
such that 
$\llrr{abcd}$ (respectively,  $\llrr{abdn}$,  
$\llrr{bcdn}$, $\llrr{acdn}$, $\llrr{abcn}$) 
is the unique twistor coordinate coming from ${\{a,b,c,d,n\} \choose 4}$
	which vanishes on 
	$\tilde{Z}(S^{(i)})$.
Moreover, we can realize the elements of $S^{(1)}$
using path matrices which have a row whose support
is precisely $\{a,b,c,d\}$ (and similarly for the other
$S^{(i)}$). 
	If $G^{(i)}$ is reduced, then $S^{(i)}$ is the desired facet $S'$.
\end{theorem}
\begin{proof}
Let $G_L$ and $G_R$ be reduced plabic graphs corresponding to $S_L$
and $S_R$. By \cite[Theorem 18.5]{postnikov} (see also \cite[Theorem B.14]{even2023cluster}
), any cell $S'$ of 
codimension $1$ in $\overline{S}$ comes from a plabic graph $G'$ 
obtained by removing an edge $e$ from $G_L \bcfw G_R$.
Such an edge could be in $G_L$ or $G_R$ or in the ``butterfly.''
Choose $I$ from  ${\{a,b,c,d,n\} \choose 4}$.
We first claim that if 
$\llrr{I}$  is the unique twistor coordinate among
${\{a,b,c,d,n\} \choose 4}$ which 
	vanishes on $Z_{S'}$, 
then edge $e$ must come from the butterfly.

Suppose $e$ does not come from the butterfly.
Then $G' = G'_L \bcfw G'_R$, where either $G'_L=G_L$ and 
$G'_R$ is obtained from $G_R$ by removing an edge $e$,
or vice versa. Since we are assuming the twistor coordinates from ${\{a,b,c,d,n\} \choose 4}$ which are not $\llrr{I}$ do not vanish on $\gt{S'}$,  \cref{lem:coin3} implies that
$\{a,b,n\}$ is coindependent for the cell of $G'_L$,
and $\{b,c,d,n\}$ is coindependent for the cell of $G'_R$.
Hence $G'_L$ and $G'_R$ have perfect orientations where 
$\{a,b,n\}$ 
and $\{b,c,d,n\}$ are sinks.
But now by \cite[Lemma 10.4]{even2023cluster}, 
all elements of 
${\{a,b,c,d,n\} \choose 4}$ are coindependent for 
$S'$, the cell associated to $G'_L \bcfw G'_R$.  
Meanwhile we know by \cite[Lemma 11.6]{even2023cluster} 
that $\llrr{I}$ has a strong sign on $Z_S$. 
	Therefore by \cref{lem:coin2}, $I$ is \emph{not}
	coindependent for $S'$. This is a contradiction.

Now we know that 
if $\llrr{I}$ is the unique twistor coordinate among
${\{a,b,c,d,n\} \choose 4}$ which 
vanishes on $Z_{S'}$,
then $S'$ has a plabic graph
which is obtained from $G_L \bcfw G_R$ by removing an edge $e$
from the butterfly. 
Let us choose 
perfect orientations of  $G_L$ and $G_R$ where 
$\{a,b,n\}$ 
and $\{b,c,d,n\}$ are sinks.  
We can then complete this to a perfect orientation
	of $G = G_L \bcfw G_R$ with a source at $d$, as in \cref{fig:extra2}.
\begin{figure}[h]
\centering
\includegraphics[height=3in]{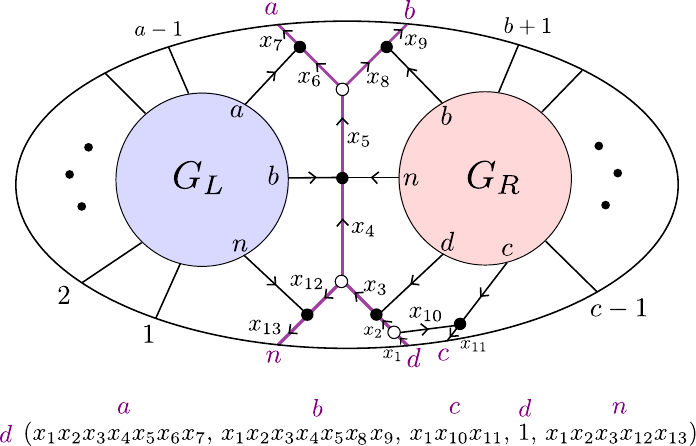}
\caption{A perfect orientation of the butterfly, and the nonzero
	entries of row $d$ in the associated path matrix.}
\label{fig:extra2}
\end{figure}

Then the path matrix $C$ associated to this perfect orientation
has a row indexed by $d$ with exactly five nonzero entries
in positions $a, b, c, d, n$.  If we weight the edges
	of $G$ as in \cref{fig:extra2}, the row $d$ of the path
	matrix is exactly as shown in the bottom of \cref{fig:extra2}.

Now notice that if we delete the edge $e$ labeled by $x_{12}$, i.e. if 
$x_{12}=0$, then our perfect orientation restricts to a perfect
orientation of the remaining subgraph, and 
when we construct the path matrix $C'$, row $d$ will have support $\{a,b,c,d\}$.
	Thus the path matrix $C'$, representing points of a cell $S^{(1)}$,
will fail to be coindependent at 
$\{a,b,c,d\}$ and hence the twistor coordinate
	$\llrr{a,b,c,d}$ will vanish on $Z_{S^{(1)}}$.
However, we can still find perfect orientations of the 
	``butterfly $\setminus \{e\}$'' with sinks at the other four elements
of ${\{a,b,c,d,n\} \choose 4}$, which all include $n$. 
	So these other four twistor coordinates will not vanish on $Z_{S^{(1)}}$.

Similarly, if we delete the edge $e$ labeled by $x_{10}$, then
row $d$ will have support $\{a,b,d,n\}$, and the analogous argument
shows that the associated cell $S^{(2)}$ will fail to be coindependent
at $\{a,b,d,n\}$.  Moreover $\llrr{a,b,d,n}$ will be the unique twistor
among 
 ${\{a,b,c,d,n\} \choose 4}$ which vanishes on $Z_{S^{(2)}}$.
Meanwhile, if we delete the edge $e$ labeled by $x_6$ (respectively, $x_8$), we 
get a cell $S^{(3)}$ (respectively, $S^{(4)}$) for which 
 $\llrr{b,c,d,n}$ 
 (respectively, $\llrr{a,c,d,n}$) is the unique twistor
among 
 ${\{a,b,c,d,n\} \choose 4}$ which vanishes on the image of the cell under $\tilde{Z}$.

In order to discuss what happens when we delete the edge labeled by $x_1$, we first
need to construct a new perfect orientation $\O'$, by reversing the directed path from $d$ to $n$.
Then when we delete the edge labeled by $x_1$, $\O'$ restricts to a perfect orientation,
and the associated path matrix has a row indexed by $n$ whose support is $\{a,b,c,n\}$.
As before $\llrr{a,b,c,n}$ will be the unique twistor 
among 
 ${\{a,b,c,d,n\} \choose 4}$ which vanishes on $Z_{S^{(5)}}$.

This constructs the plabic graphs $G^{(i)}$ corresponding 
to the cells $S^{(i)}$ (for $1\leq i \leq 5$) whose existence the theorem predicts. If $G^{(i)}$ is reduced, then $S^{(i)}$ is a facet of $S$, as desired.

To show that no other cells have the desired properties,
we show that if we delete any other edge of the butterfly, we get 
a cell $S'$ such that 
at least two twistors
coordinates among ${\{a,b,c,d,n\} \choose 4}$  
 vanish on $Z_{S'}$.
For example if we delete the edges labeled $x_2$ or $x_4$, 
we still have a perfect orientation but now row $d$ of the path matrix $C'$
has support at most three, which means that at least two twistor 
 coordinates among ${\{a,b,c,d,n\} \choose 4}$  
will vanish on $C'Z$.
To analyze what happens if we delete any of the other 
edges we have to change the perfect orientation, but in all cases
our path matrix $C'$ will have a row whose support is a 
$1$, $2$, or $3$-element subset of $\{a,b,c,d,n\}$,
which means that at least two twistor coordinates 
 among ${\{a,b,c,d,n\} \choose 4}$  
will vanish on $C'Z$.
\end{proof}

\begin{lemma}\label{lem:standardpermutation}
Let $S$ be a standard BCFW cell, and let $\pi$ be its trip permutation.
	Then $\pi(n) \notin \{1, n-1, n-2\}$, 
	and $\pi(1) \neq n-1$.
\end{lemma}
\begin{proof}
This follows from the Le-diagram description of standard BCFW cells
from \cite[Definition 6.2]{karp2020decompositions}, or 
 the related $\oplus$-diagram description given in 
\cite[Definition 2.24]{even2021amplituhedron}.
\end{proof}

\begin{theorem}[Plabic graphs for potential facets of standard BCFW tile] \label{thm:showingreduced}
Let $G = G_L \bcfw G_R$ be a reduced plabic graph
for 
the standard BCFW cell 
$S = S_L \bcfw S_R$ associated to a chord diagram
$D$ with top chord $D_{\rtop}$.  
Use the notation of \cref{thm:twistorvanish}
and 
\cref{fig:extra2}, and identify the labels of edges of $G$ with the edges
themselves.
\begin{itemize}
\item[($\alpha$).] 
	If $D_{\rtop}$ does not have a sticky child, then 
		$G\setminus \{x_{6}\}$ is reduced.\footnote{The converse may not be true.}
\item[($\beta$).]  
	$D_{\rtop}$ does not start where another chord 
		ends if and only if $G\setminus \{x_8\}$ is reduced.
\item[($\gamma$).] 
	The graph $G\setminus \{x_{10}\}$ is reduced.
\item[($\delta$).]  
	$D_{\rtop}$ does not have a same-end child 
		if and only if 
		$G\setminus \{x_1\}$ is reduced.
\item[($\epsilon$).] 
	$D_{\rtop}$ does not have a same-end child
		if and only if $G\setminus \{x_{12}\}$ is reduced.
\end{itemize}
\end{theorem}

Before proving the theorem, we recall a useful lemma.
\begin{lemma}\cite[Lemma 18.9]{postnikov}\label{lem:edgeremoval}
Let $G$ be a reduced plabic graph with trip permutation $\pi$,
 let $e$ be an edge of $G$,
and let $T_1: i \to \pi(i)$ and $T_2: j \to \pi(j)$ 
be the two trips in $G$ that pass through $e$ 
(the trips will pass through this edge in two different directions).
 Then $G \setminus \{e\}$ is 
       reduced if and only if the pair $(i,\pi(i))$
       and $(j,\pi(j))$ is a simple crossing in $\pi$.
\end{lemma}

\begin{proof}[Proof of \cref{thm:showingreduced}]
\begin{figure}[h]
		\includegraphics[height=3.5cm]{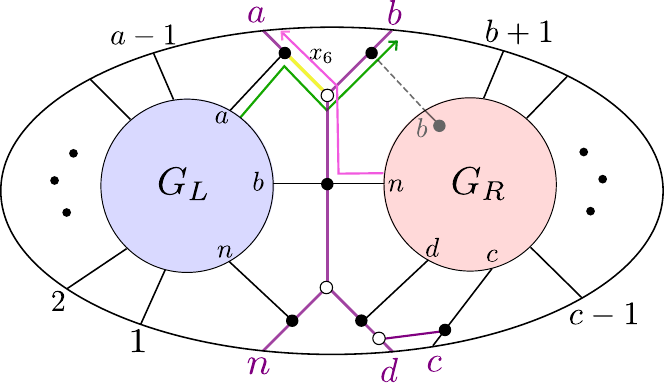}
		\caption{
	If $D_{\rtop}$ does not have a sticky child, then 
		$G\setminus \{x_{6}\}$ is reduced.}
		\label{fig:alpha}
\end{figure}

Case  ($\alpha$).  If $D_{\rtop}$ does not 
have a sticky child, then $G_R$ has a black lollipop at $b$. This means
that in $G$, the edge connecting vertex $b$ in $G_R$ to the ``butterfly''
can be contracted.  
The trips going through edge $x_6$ are shown in \cref{fig:alpha}.
Since these two trips end at adjacent boundary vertices,
they must be part of a simple crossing. 
	Therefore by \cite[Lemma 18.9]{postnikov},
$G\setminus \{x_6\}$ is reduced.
	
	\begin{figure}[h]
		\includegraphics[height=3.5cm]{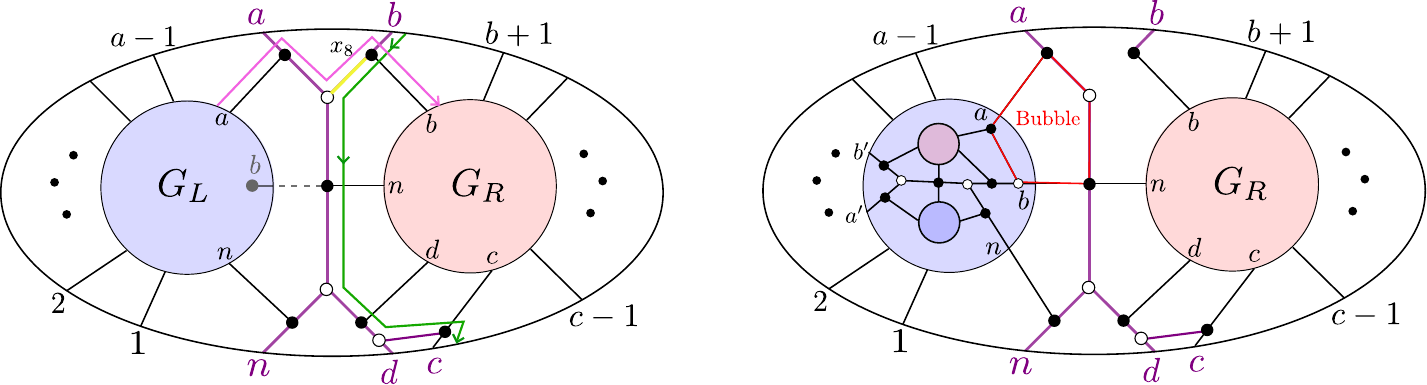}
		\caption{Left:
	if $D_{\rtop}$ does not start where another chord 
		ends then $G\setminus \{x_8\}$ is reduced.
       Right: if 
	 $D_{\rtop}$ starts where another chord 
		ends then $G\setminus \{x_8\}$ is non-reduced.}
		\label{fig:beta}
\end{figure}

Case ($\beta$). 
	Suppose that $D_{\rtop}$ does not start where 
another chord ends.  Then $G_L$ has a black lollipop at vertex $b$,
	which means that the edge (shown dashed in \cref{fig:beta})
	connecting that vertex to the butterfly
	can be contracted.  The two 
trips which pass through $x_8$ are shown in pink and green in \cref{fig:beta}.  
By \cref{lem:standardpermutation}, $\pi_{G_L}(n) \neq a$  and so the
pink trip in $G$ must start at the left part of the graph, i.e. 
at some element in $\{1,2,\dots,a-1\}$.  We also claim that the 
pink trip in $G$ must end at the right part of the 
graph, i.e. at some element in $\{b+1,b+2,\dots,c-1\}$, otherwise
the pink and green trips would have a \emph{bad double crossing}
and $G$ would fail to be reduced \cite[Theorem 13.2]{postnikov}.
But now it is clear that the pink and green trips must form a simple 
crossing, because there is no other trip in $G$ that starts
	at an element of $\{1,2,\dots,a\}$ and ends at an element
	of $\{b+1,b+2,\dots,c-1\}$.
	Therefore by \cite[Lemma 18.9]{postnikov},
$G\setminus \{x_8\}$ is reduced.
	
Now suppose that $D_{\rtop}$ starts where 
another chord ends.  Then $G_L$ has the form shown at the right of 
 \cref{fig:beta}: in particular, the vertices $a$ and $b$ of $G_L$
are connected by a black-white bridge. But then 
when we delete edge $x_8$, the resulting graph
has a configuration of vertices which is 
move-equivalent to a bubble (cf \cite[Definition B.2]{even2023cluster})
, as shown in 
	the right of \cref{fig:beta}.  Therefore $G\setminus \{x_8\}$
	is not reduced.

	\begin{figure}[h]
		\includegraphics[height=3.5cm]{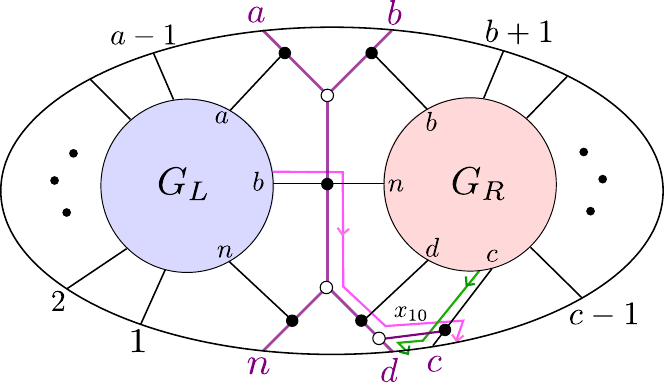}
		\caption{
			The graph $G\setminus \{x_{10}\}$ is reduced.}
		\label{fig:gamma}
\end{figure}

Case ($\gamma$).
The two trips passing through $x_{10}$ are shown in \cref{fig:gamma}.
Since these two trips end at $c$ and $d$, there cannot be another trip 
ending between $c$ and $d$, hence they represent a simple crossing.
Therefore 
	 by \cref{lem:edgeremoval},
$G\setminus \{x_{10}\}$ is reduced.

\begin{figure}[h]
\includegraphics[height=3cm]{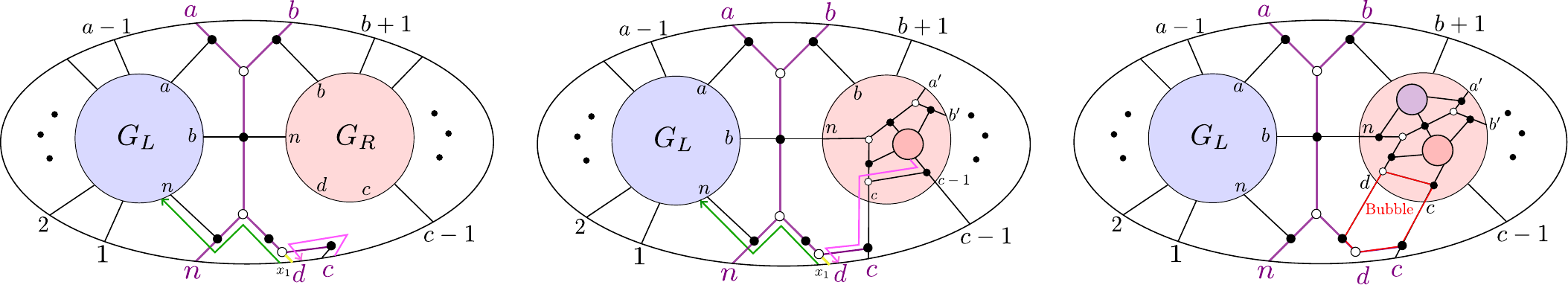}
\caption{Left: If $D_{\rtop}$ does not have a same-end child
and $D$ has no chord ending at $(c-1,c)$, then $G\setminus \{x_1\}$
is reduced.  Middle: If $D_{\rtop}$ does not have a same-end child
and $D$ does have a chord ending at $(c-1,c)$, then $G\setminus \{x_1\}$ is reduced.
Right: if $D_{\rtop}$ has a same-end child then $G\setminus \{x_1\}$ is not reduced.}
	\label{fig:delta}
\end{figure}

Case ($\delta$).
Suppose that $D_{\rtop}$ does not have a same-end child.
Then $D$ does not have another chord ending at $(c,d)$, and hence
in $G_R$, the vertex $d$ will be a black lollipop that can be contracted.
First suppose there is no chord in $D$ ending at $(c-1,c)$,
then there is  also a lollipop in $G_R$ at $c$, and $G$ looks as shown at the 
left of \cref{fig:delta}. Then one of the trips through edge $x_1$
goes from $c$ to $d$, so the two trips passing through $x_1$ must form a simple
crossing.  Therefore
	 by \cite[Lemma 18.9]{postnikov},
		$G\setminus \{x_1\}$ is reduced.
Now suppose there \emph{is} a chord in $D$ ending at $(c-1,c)$.
Then $G$ looks as shown in the middle of \cref{fig:delta}.
By \cref{lem:standardpermutation}, $\pi_{G_R}(1) \neq n-1$  and 
$\pi_{G_R}(n) \neq n-1$, so the pink trip must start at an element of 
$\{b+1,\dots,c-1\}$.  Similarly, by \cref{lem:standardpermutation},
$\pi_{G_L}(n) \neq n-2$ and $\pi_{G_L}(n) \neq n-1$, so the 
green trip must end at an element of $\{1,2,\dots,a-1\}$. But now the pink
and green trips must form a simple crossing, because there is no other 
trip that can start at an element of $\{b+1,\dots,c-1\}$ and end at an element
of $\{1,2,\dots,a-1\}$.  Therefore 
		$G\setminus \{x_1\}$ is reduced.

Now suppose that $D_{\rtop}$ has a same-end child.
Then $G_R$ has a black-white bridge at vertices $c,d$, 
and when we delete $\{x_1\}$, $G\setminus \{x_1\}$  looks as in the right of 
\cref{fig:delta}.  We obtain a face which is move-equivalent to a bubble,
so $G \setminus \{x_1\}$ is not reduced.

	\begin{figure}[h]
		\includegraphics[width=\textwidth]{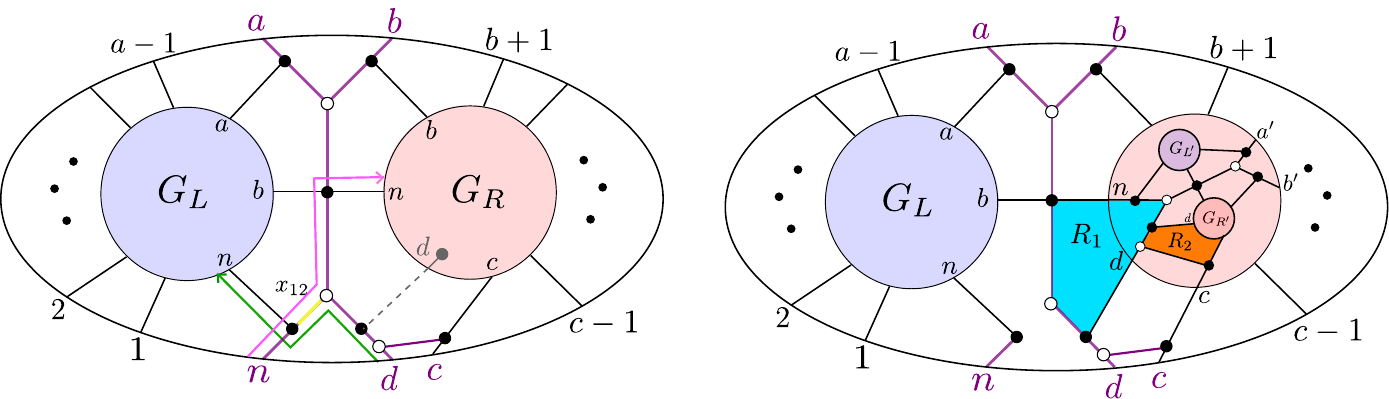}
		\caption{
			Left: 
			if $D_{\rtop}$ does not have a same-end child
		then $G\setminus \{x_{12}\}$ is reduced. 
		Right: 
			if $D_{\rtop}$ has a same-end child
		then $G\setminus \{x_{12}\}$ is not reduced. 
	}
		\label{fig:epsilon}
\end{figure}

Case ($\epsilon$). Suppose that $D_{\rtop}$ does not have a same-end child.
Then $G_R$ has a black lollipop (which can be contracted), and hence the two 
trips passing through $x_{12}$ are as shown at the left of \cref{fig:epsilon}.
Since these two trips start at adjacent vertices $d$ and $n$, they must form
a simple crossing.  Therefore by 
	  \cite[Lemma 18.9]{postnikov}, 
	  $G\setminus \{x_{12}\}$ is reduced.

Now suppose that $D_{\rtop}$ does have a same-end child. 
Then $G_R$ has a black-white bridge, as shown in the 
right of \cref{fig:epsilon}.
$G_R$ is itself the plabic graph of a standard BCFW cell, so we can 
write it as $G_R = G_{L'} \bcfw G_{R'}$.  If $d$ is a black 
lollipop in $G_{R'}$, then we can contract the edge joining that lollipop
to the butterfly in $G_r$, and then we find that region $R_1$
in \cref{fig:epsilon} is move-equivalent to a bubble.
On the other hand, if $d$ is not a black lollipop in $G_{R'}$,
then $D_{\rtop}$ has a same-end grandchild, so $G_{R'}$ has a black-white
bridge.  Then one can do a square move at $R_2$ which turns 
$R_1$ into a bubble.  Therefore 
	  $G\setminus \{x_{12}\}$ is not reduced.
\end{proof}

\begin{proof}[Proof of \cref{prop:standard_facets}]
	By \cref{lem:maps_to_bdry}, all facets of $S_D$ map to the boundary of $\gt{D}$, so any cell in $\partial S_D$ whose image is codimension 1 in $\gt{D}$ is a facet of $\gt{D}$. \cref{thm:clusteradjacency} shows that all facets of $\gt{D}$ lie 
in the zero locus of a cluster variable in $\Irr(D).$ 
By  \cref{lem:mutable_vars_give_codim_2}, 
no facet is contained in the zero locus of a mutable cluster variable
	 $\Mut(Z_D)$. 
Thus, we are left to show the following.
\begin{claim}\label{claim:main}
For each frozen variable $\rzeta$ in $\AFacet(D)$, there is exactly 
one cell $S$ of codimension $1$ in $\overline{S}_D$ 
such that $\gt{S}$ is codimension 1 in $\gt{D}$ and $\gt{S}$ lies in the zero locus of $\rzeta$.
\end{claim}

In \cite[Section 7]{even2021amplituhedron} it was shown that each facet $S'$ of a standard BCFW cell $S_D$ either:
\begin{enumerate}
\item
maps to the interior of $\Ank$, in which case it 
maps injectively \cite[Proposition 8.2]{even2021amplituhedron}, 
and lies in the zero locus of a 
coordinate functionary,\footnote{\cite[Section 7]{even2021amplituhedron} is phrased using entries and 2-by-2 minors of the domino matrix, which are the same as our BCFW coordinates.} or
\item maps to the boundary of $\Ank$, in which case
$Z_{S'}$ lies in the zero locus of a domino cluster variable of the form
$\llrr{i,i+1,j,j+1}$.
\end{enumerate}

In the first case, 
 \cref{claim:main}
follows from results of \cite{even2021amplituhedron}, as we now explain.
Those facets of $S_D$
which map injectively to the interior of the 
amplituhedron 
are in bijection with the elements of $\Froz(Z_D)$
which do not have the form $\llrr{i,i+1,j,j+1}$, and 
can be explicitly constructed using the BCFW recursion,
but with one parameter set to $0$  
\cite[Lemma 7.9]{even2021amplituhedron}.
	Then using the arguments from the proof
 of \cref{thm:clusteradjacency}
 , one can see that if $S'$ is a facet 
of $S_D$ where a {single} BCFW coordinate $\zeta_i$ vanishes, 
then $Z_{S'}$ lies in the zero locus of the corresponding cluster variable $\rzeta_i.$
Moreover, 
for every BCFW parameter, there is at most one facet of $S_D$ where {only}
that parameter vanishes
 (cf. \cite[Lemmas 7.9, 7.13, 7.14, 7.15]{even2021amplituhedron}). 

We now show that 
\cref{claim:main} holds for frozen domino variables of the form
$\lr{i,i+1,j,j+1}$, using 
results of \cite[Section 7]{even2021amplituhedron}
as well as
	\cref{thm:twistorvanish} and \cref{thm:showingreduced}.
We use the notation of \cite{even2021amplituhedron}
which are close to the ones used in this paper, but not identical.

\emph{Step 1: constructing the facets.}
Since we are concerned only with facets of $Z_D$ where 
a boundary twistor $\llrr{i,i+1,j,j+1}$ vanishes, 
we can use \cref{thm:twistorvanish} 
	and \cref{thm:showingreduced} to build the plabic graph $G$ corresponding to the facet (we will show in Step 3 below that the image of the cell $S_G$ has codimension 1 in $Z_D$). 
Concretely,
in order to construct the graph $G$ corresponding to the facet of $Z_D$ where 
$\rzeta_i$ vanishes (where $\rzeta_i$ is a boundary twistor), 
we follow the procedure for constructing $S_D$,
but at the $i$th step we remove the edge of the 
butterfly dictated by \cref{thm:twistorvanish}.

\emph{Step 2: Uniqueness of facets where a given cluster variable vanishes.}
We use induction to show that for each 
$\llrr{i,i+1,j,j+1} \in \Froz(Z_D)$, 
there is at most one facet of 
a tile $Z_D$ in its zero locus.
From  \cite[Lemma 10.5]{even2023cluster}
, we know that 
each facet $Z_{S'}$ of a BCFW tile $Z_D$ either 
(1) lies in the vanishing locus of a domino variable $\rzeta_k$ of the $k$th chord (which is a twistor coordinate with indices in $\{a,b,c,d,n\}$),
or 
(2) the cell $S'$ is the BCFW product of a BCFW cell and a facet of another BCFW cell. 
By induction, the tiles coming from Case (2) lie in the vanishing locus 
of distinct cluster variables; and these cluster variables must all be
different from the twistor coordinates of the $k$th chord. (The only case when a coordinate cluster variable from $S_L$ or $S_R$ promotes to a twistor coordinate for the top chord is the case of $\rbeta_i$ where $D_i$ is a sticky same-end child of $D_k$; in this case, $\rbeta_i=\ralpha_k=\llrr{bcdn}$ which is not a boundary twistor since $D_k$ has a child.)
In Case (1), \cref{thm:twistorvanish}
 shows that there is at most one facet $Z_{S'}$ of $Z_{S_D}$ which 
 lies in the zero locus of a single chord twistor of the $k$th chord. 
But now by \cref{lem:mutable_vars_give_codim_2}, if two cluster variables vanish
on $Z_{S'}$, it must have codimension at least $2$, 
so all facets of $Z_{D}$ must lie in the vanishing locus of distinct cluster variables.

\emph{Step 3: Injectivity of the amplituhedron map.} 
In light of  \cref{thm:twistorvanish},
we can alternatively construct the facets by 
following the recipe of
 \cref{def:std-bcfw-cells}
 , but 
setting exactly one of the 
BCFW parameters 
$\{\alpha_i, \beta_i, \delta_i, \gamma_i, \epsilon_i\}$ for $1\leq i \leq k$ 
 equal to $0$ at the appropriate BCFW step.
Using slightly different conventions, such a 
construction\footnote{This construction 
was called the \emph{$(D,\star)-$extended domino form} for $\star\in\{\alpha_i,\beta_i,\gamma_i,\delta_i,\varepsilon_i,\eta_{ij},\theta_{ij}\}$.} 
was given in 
 \cite[Definition 7.6 and Lemma 7.7]{even2021amplituhedron}
for most facets,
building each facet in terms 
 of the operations $\pre_i,\inc_i,x_i(\R_+),y_i(\R_+).$

Now we need to show that the amplituhedron map restricted to 
 $S',$ the facet of $S$ 
obtained by setting a particular BCFW parameter $\star$ to $0$, is injective. The proof is similar to the proof of \cite[Theorem 7.7]{even2023cluster}. 
The positroid cell $S'$ is constructed by a sequence of adding zero columns, BCFW products, and a single ``degenerate" BCFW product.

As in the proof of \cite[Theorem 7.7]{even2023cluster} the proof of injectivity follows by showing that injectivity persists through the different steps of the construction of $S'.$ The treatment in the cases of adding a zero column, and doing a BCFW product is identical to the treatment in \cite[Theorem 7.7]{even2023cluster}
, 
relying on \cite[Theorem 11.3]{even2023cluster} 
(as before  we need to verify that $\{b_i,c_i,d_i,n\}$ is
coindependent at the time of the $i$th BCFW step). 
The treatment in the single degenerate BCFW product is also completely analogous to that of\cite[Theorem 7.7]{even2023cluster}
, and this proves the injectivity. 

Note, however, that in the application of \cite[Lemma 11.13]{even2023cluster} 
for the degenerate step, the coordinate $\star$ turns out to be $0,$ while the other four keep the same sign they would have had on the BCFW cell at that stage. This twistor will be promoted, according to \cite[Theorem 11.3]{even2023cluster} to a functionary vanishing on this facet. 
The same argument used in the proof of \cref{thm:clusteradjacency} 
shows that each facet lies in the zero locus of the corresponding \emph{reduced} functionary. In light of the uniqueness discussion above, we see that each such reduced boundary functionary corresponds to a unique facet. 
It also follows that the facet is characterized as the locus where the corresponding functionary vanishes, but the other coordinate functionaries keep their signs.
\cref{lem:maps_to_bdry} shows that the facets indeed map to the boundary of the tile.

\end{proof}
Note that the uniqueness in the above proof follows from two facts. First, if a facet in the domain has image which is not a facet at some time of the cell construction process, then the BCFW product of this facet with a standard BCFW cell will also have image which is not a facet. Second, when a new facet in the domain (which corresponds to the rightmost top chord at a given time of the process) maps to a facet of the tile, it is the maximal face in the domain, among those which map into the zero locus of the corresponding chord twistor, hence other components in this zero locus are of lower dimension already in the domain.

\subsection{Proof of \cref{lem:mutable_vars_give_codim_2}}\label{sec:trans}
The proof of the lemma will use the notion of transversality. For this we recall some notions and facts.
 \begin{definition}
	 Let $X$ be an $n$ dimensional manifold with an atlas $\{(U_\alpha,\phi_\alpha:U_\alpha\to\R^n)\}_{\alpha\in A}.$ We say that a set $\ZZ\subseteq X$ \emph{is of measure $0$,}  if for every $\alpha\in A,$ the set $\phi_\alpha(\ZZ \cap U_\alpha)$ is of Lebesgue measure $0$ in~$\R^n.$
	If $\YY\subset X$ is the complement of a measure $0$ subset, we say that 
	 \emph{almost every $x\in X$ belongs to $\YY.$}
\end{definition}

	\begin{definition} 
		Let $f:X\to \YY$ be a smooth map between smooth manifolds $X,\YY$. Let $\ZZ$ be a smooth submanifold of $\YY.$ We say that $f$ is \emph{transverse} to $\ZZ,$ and write $f\pitchfork \ZZ$ if for every $x\in f^{-1}(\ZZ)$
		\[df_x(T_xX)+T_{f(x)}\ZZ=T_{f(x)}\YY,\]where $T_xX$ denotes the tangent space of $X$ at $x\in X,$ and $df_x$ is the differential map at $x,$ which maps $T_xX$ into $T_{f(x)}\YY.$
	\end{definition}
	
	\begin{theorem}[Thom's Parametric Transversality Theorem]\label{thm:thom} 
		Let $X$ be a smooth manifold, let $B,\YY$ be smooth manifolds and let $\ZZ$ be a submanifold of $\YY$. Let $f:X\times B\to \YY$ be a smooth map.
		Suppose that $f\pitchfork \ZZ$.
		Then for almost every $b\in B$ the map
		\[f(-,b):X\times\{b\}\to \YY\]
		is transverse to $\YY.$
	\end{theorem}
 We first prove a general ``almost-every $Z$" result. 
	\begin{lemma}\label{lem:high_codim_promotes_to_high_codim}
		The zero locus in the amplituhedron $\Ank$ of two different irreducible functionaries (as in \cref{def:functionary}) is of codimension 
		at least $2$ for almost all $Z.$
	\end{lemma}

We know from \cite[Theorem 1.3]{GLS} 
 that all cluster variables are irreducible; therefore,
in light of \cref{def:functionary}, 
functionaries which correspond to cluster variables of $\Gr_{4,n}$ are irreducible.

 \begin{proof} 
		We will prove the lemma in the B-amplituhedron 
		(cf. {\cite[Definition 3.8]{karpwilliams}} (see also \cite[Definition 2.20]{even2023cluster})
		$\mathcal{B}_{n,k,4}(W)$, where $W$ is the column span of $Z$.
		This will imply the result for $\Ank$, since the map $f_Z$ of \cite[Proposition 2.21]{even2023cluster} (which combines \cite[Lemma 3.10 and Proposition 3.12]{karpwilliams}) is a diffeomorphism from a neighborhood of the $B$-amplituhedron to a neighborhood of $\Ank$. 
		The map between the two spaces takes the zero locus of an irreducible functionary to the zero locus of an irreducible polynomial in the Pl\"ucker coordinates of $\Gr_{4,n},$ and we consider its intersection with $\mathcal{B}_{n,k,4}(W)$.
		It will be enough to show that its intersection with $\Gr_4(W),$ for a generic $W\in \Gr_{k+4,n}$ is of codimension $2.$ We will use Thom's transversality.
		Let $\YY=\Gr_{4,n},$ and $\ZZ$ the intersection of zero loci of the two functions. Then $\ZZ$ is of codimension $2.$ Let $B$ be a small ball around $W\in \Gr_{k+4,n},$ and $X=\Gr_4(W).$ Identify the fiber bundle $F\to B$ whose fiber over $W'\in B$ is $\Gr_4(W')$ with $X\times B.$ This can be done since the two spaces are diffeomorphic, for $B$ small enough. The map $f:X\times B\to \YY$ is defined by
		\[f(V,W')=V,\]
		where $W'\in B,~V\in\Gr_4(W')$ and in the right hand side $V$ is considered as an element of $\Gr_{4,n}.$ Clearly $df_{V,W'}(T_{V,W'}X\times B)=T_{V}\Gr_{4,n},$ so that the assumption of Theorem \ref{thm:thom} is met. Thus, for almost every $W'\in B,$ the intersection $\Gr_4(W')\cap \ZZ$ is of codimension $2,$ hence the intersection with $\ZZ$ of the $B-$amplituhedron, for almost every $W,$ is of codimension at least $2.$
  \end{proof}
\begin{proof}[Proof of \cref{lem:mutable_vars_give_codim_2}]
The last statement follows from the first one, since
if $\xi$ is a mutable variable for $\gt{D},$ then the mutation relation has the form
 \[\xi\xi'=A+B,\]where $\xi$ is the variable of interest, and $A,B$ are products of other cluster variables. 
Moreover, by \cite[Proposition 9.27]{even2023cluster}, 
$A,B$ have the same sign on $\gto{D}$. Thus, the vanishing of $\xi$ implies the vanishing of at least one more cluster variable.

Every facet of $\gt{D}$ lies in the zero locus of a cluster variable, by \cref{thm:clusteradjacency}. By \cite[Theorem 11.3]{even2023cluster} we know that the cluster variables of $\gt{D}$ have a strongly positive expression, hence every such functionary either vanishes identically on a given boundary $\gt{S},$ for all positive $Z,$ or never vanishes there, for all positive $Z.$ Let $S_1,\ldots,S_N$ be the facets of $S_D$ which map to the zero locus of a single cluster variable.

From the previous lemma it follows that for almost all positive $Z$ the remaining faces of $S_D$ map to the union of finitely many codimension $2$ submanifolds of $\gt{D}.$ These submanifolds are contained in $\partial\gt{D},$ using \cref{lem:maps_to_bdry} and the fact that no cluster variable of $Z_D$ vanishes on $\gto{D}.$

 Denote by $\ZZ(\xi_1,\xi_2)\subset\gt{D}$ the vanishing locus of $\xi_1$ and $\xi_2$ in $\gt{D}.$ Let $S'_1,\ldots,S'_M$ be the faces of $D$  which map to $\ZZ(\xi_1,\xi_2).$
 Note that
 \begin{equation}\label{eq:containment_1}\ZZ(\xi_1,\xi_2)\subseteq\gt{D}\setminus\left(\gto{D}\sqcup_{i=1}^N\gto{S_i}\right)\subseteq\partial\gt{D}.\end{equation}
 For almost all positive $Z$, $\ZZ(\xi_1,\xi_2)$ is of codimension at least $2.$ 
	We will now show that for almost all positive $Z$ 
 \[\ZZ(\xi_1,\xi_2)\subseteq\sqcup_i\gt{S_i},\]
	together with \eqref{eq:containment_1} this implies, 
	that for almost all positive $Z,$ and every $j=1,\ldots, M$
 \begin{equation}\label{eq:containment_2_skel}
	 \gt{S'_j} \subseteq \sqcup_{i=1}^N \sqcup_{S'~\text{is a face of }S_i} \gt{S'},
 \end{equation}
 that is, the union of images of faces of $D$ of codimension at least $2.$

In order to show \eqref{eq:containment_1}, take an arbitrary $p\in \ZZ(\xi_1,\xi_2).$ We will show that every neighborhood $U$ of $p$ contains a point from $\bigcup_{i=1}^N\gto{S_i}.$
Indeed, assume without loss of generality that $U$ is connected, since $p$ belongs to the boundary of $\gt{D},$ we can find two points $q_0\in \gt{D}\cap U,~q_1\in U\setminus\gt{D}.$
We can find a path $(q_t)_{t\in[0,1]}\subset U$ from $q_0$ to $q_1$ in $U$ not passing throw the intersection of zero loci of any two different cluster variables, which we assume to be of codimension $2$ or more (see, e.g., the proof of \cite[Proposition 8.5]{even2021amplituhedron}). Let $t$ be the last time where $q_t\in \gt{D}.$ Then $q_t$ must be in the zero locus of a single cluster variable, hence in some $\gt{S_i}.$ 

Now, since \eqref{eq:containment_2_skel} holds for almost every positive $Z,$ and both its left hand and right hand are images of compact sets, it holds in fact for every positive $Z.$ Indeed, if $Z$ is the limit of $(Z_i)_{i=1}^\infty$ where for each $Z_i$ \eqref{eq:containment_2_skel} holds, it also holds for $Z.$

\end{proof}

\subsection{Facets of general BCFW tiles}\label{sec:BCFWfacets}
We now describe, without proof, the facets of general BCFW tiles in \cref{claim:facetsgenBCFW}. Instead of the recipe in \cref{def:recipe}, it is convenient to use a slightly different indexing set for BCFW tiles.

\begin{definition} \label{def:5tuples}
    Let $\rcp$ be a recipe with $k$ step-tuples, which is composed by a recipe $\rcp_L$ followed by a recipe $\rcp_R$ followed by a step-tuple
$((a_k, b_k, c_k, d_k, n_k),\pre_{I_k}, \cyc^{r_k}, \refl^{s_k})$. We introduce the following collection of $5$-tuples $\tilde{D}=\{(\tilde{a}_i,\tilde{b}_i,\tilde{c}_i,\tilde{d}_i,\tilde{n}_i)\}_{i=1}^k=\tilde{D}_L \cup \tilde{D}_R \cup \tilde{D}_k$ we call \emph{generalized chords} defined recursively as:
    \begin{itemize}
        \item $\tilde{D}_k=(\tilde{a}_k,\tilde{b}_k,\tilde{c}_k,\tilde{d}_k,\tilde{n}_k) = \refl^{s_k} \circ \cyc^{r_k} (a_k,b_k,c_k,d_k,n_k)$,
        \item $\tilde{D}_L =\refl^{s_k} \circ \cyc^{r_k} \tilde{D}'_L$ and $\tilde{D}_R =\refl^{s_k} \circ \cyc^{r_k} \tilde{D}'_R$,
    \end{itemize}
    where $\tilde{D}'_L$ (resp. $\tilde{D}'_R$) are the generalized chords for the recipe $\rcp_L$ (resp. $\rcp_R$).
\end{definition}

\begin{notation} \label{not:5tuples}
Given a BCFW cell $S_{\rcp}$, we will sometime label it as $S_{\tilde{D}}$ in terms of the corresponding generalized chords $\tilde{D}$. We denote by $\tilde{D}^{(j)}_L \cup \tilde{D}^{(j)}_R \cup \tilde{D}_j$ the generalized chords of the recipe $\rcp^{(j)}$ obtained from $\rcp$ by performing only the first $j$ step-tuples. Here $\tilde{D}^{(j)}_L$ (resp. $\tilde{D}^{(j)}_R$) are the generalized chords of $\rcp^{(j)}_L$ (resp. $\rcp^{(j)}_R$). 
\end{notation}

\begin{example} \label{ex:5tuples}
Consider the BCFW cell $S_{\rcp}$ of \cref{fig:bcfw_tile}.
Its generalized chords are: \\ $\tilde{D}=\{(6,7,8,9,3),(4,5,8,9,3),(3,2,1,12,10),(4,3,11,10,9)\}$. Its plabic graph is as in \cref{fig:5tuple}. 
\end{example}

\begin{figure}[h]
\centering
\includegraphics[height=2.2in]{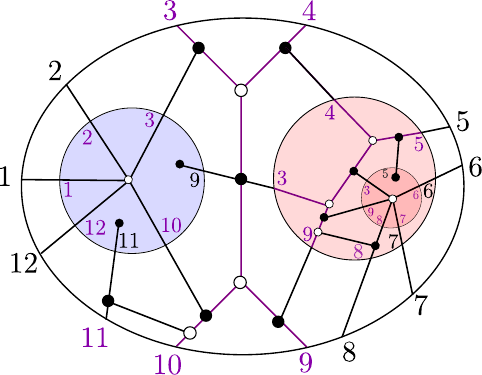}
\caption{}
\label{fig:5tuple}
\end{figure}

We introduce the definition of \emph{condensability} and \emph{condensations} of a BCFW cell $S_{\rcp}$ as follows.

\begin{definition}\label{def:condensability}
Let $S_{\tilde{D}} \subseteq\Gr^{\ge0}_{k,n}$ be a BCFW cell, and $\tilde{D}=\{\tilde{D}_i\}_{i=1}^k$ the corresponding generalized chords.
For $i\in[k], \tilde{f}_i\in\{\tilde{a}_i,\tilde{b}_i,\tilde{c}_i,\tilde{d}_i,\tilde{n}_i\}$ we say that $S_{\tilde{D}}$ is \emph{$\tilde{f}_i$-condensable} if either $\tilde{f}_i=\tilde{c}_i$ or 
\begin{equation*}
\tilde{f}_i=\begin{cases}
\tilde{a}_i\\
\tilde{b}_i\\
\tilde{d}_i\\
\tilde{n}_i
\end{cases} 
\mbox{and for all } \tilde{D}_j \in 
\begin{cases}
\tilde{D}^{(i)}_R\\
\tilde{D}^{(i)}_L\\
\tilde{D}^{(i)}_R\\
\tilde{D}^{(i)}_R
\end{cases},
\begin{cases}
\{\tilde{b}_i,\tilde{n}_i\}\\
\{\tilde{b}_i,\tilde{a}_i\}\\
\{\tilde{c}_i,\tilde{d}_i\}\\
\{\tilde{d}_i,\tilde{n}_i\}
\end{cases}
\not \subset \tilde{D}_j,
\end{equation*}
where we used \cref{not:5tuples}.
\end{definition}

\begin{example}\label{ex:condensable}
Consider the BCFW cell $S_{\rcp}$ of \cref{fig:bcfw_tile} and its generalized chords as in \cref{ex:5tuples}. The cell $S_{\rcp}$ is $\tilde{f}_i$-condensable for all $\tilde{f}_i$ except for $\tilde{f}_i\in \{\tilde{d}_2, \tilde{n}_2, \tilde{b}_4\}$. For example, the cell is not $\tilde{d}_2$-condensable because $\{\tilde{c}_2,\tilde{d}_2\}=\{8,9\} \subset \tilde{D}_1$, and $\tilde{D}_1$ is in $\tilde{D}^{(2)}_R$.
 
\end{example}

\begin{definition}\label{def:condensation}
  Let $S_{\rcp} \subseteq\Gr^{\ge0}_{k,n}$ be a BCFW cell, and $\tilde{D}=\{\tilde{D}_i\}_{i=1}^k$ the corresponding generalized chords. We define the \emph{$\tilde{f}_i$-condensation} $\partial_{\tilde{f}_i} S_{\rcp}$ of $S_{\rcp}$ to be the cell built using the recipe $\rcp$, but at the $i$-th BCFW product, we delete the edge $e_1$ if $\tilde{f}_i=\tilde{a}_i$; $e_2$ if $\tilde{f}_i=\tilde{b}_i$; $e_3$ if $\tilde{f}_i=\tilde{c}_i$; $e_4$ if $\tilde{f}_i=\tilde{d}_i$; and $e_5$ if $\tilde{f}_i=\tilde{n}_i$ as in \cref{fig:condensation}.

\end{definition}

\begin{figure}[h]
\centering
\includegraphics[height=2in]{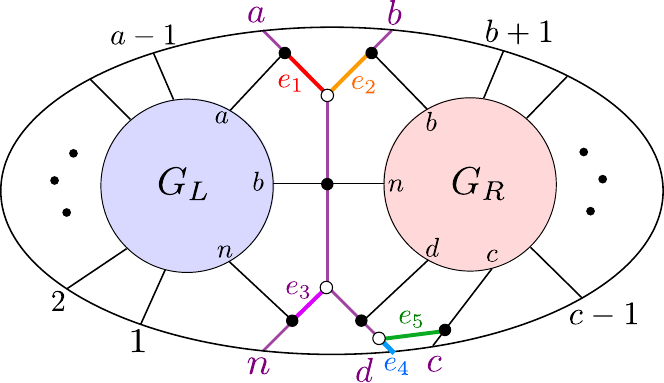}
\caption{}
\label{fig:condensation}
\end{figure}

\begin{definition}\label{def:rigid} Let $S_{\tilde{D}}$ be a general BCFW cell, with generalized chords $\tilde{D}=\{(\tilde{a}_j, \ldots, \tilde{n}_j)\}_{j=1}^k$. The $\tilde{f}_i$-condensation $\partial_{\tilde{f}_i} S_{\tilde{D}}$ of $S_{\tilde{D}}$
is \emph{rigid} if for all $\ell >i$, $\{\tilde{b}_\ell,\tilde{c}_\ell,\tilde{d}_\ell,\tilde{n}_\ell\}$ is coindependent (as in \cref{def:coind}) for $\partial_{\tilde{f}_i} S_{\tilde{D}^{(\ell)}_R}$, where $\tilde{D}^{(\ell)}_R$ is as in \cref{not:5tuples}. 
\end{definition}

Using the techniques of this paper, and extending the ones used for the standard BCFW tiles, the following statement can be shown.

\begin{claim}[Facets of general BCFW tiles]\label{claim:facetsgenBCFW}
    Let $S=S_{\tilde{D}}$ be a BCFW cell with recipe $\rcp$. If $S_{\tilde{D}}$ is  $\tilde{f}_i$-condensable and $S'=\partial_{\tilde{f}_i} S_{\tilde{D}}$ is rigid, then $Z_{S'}$ is a facet of $Z_{S}$. 
    
     Moreover, let $\rzeta_{i}$ be the coordinate cluster variable of $Z_{S}$ defined as
\begin{equation} \label{def:ftildezeta}
 \rzeta_{i}= \begin{cases}
 \ralpha^{\rcp}_i, \mbox{ if } \tilde{f}_i=\tilde{a}_i\\
\rbeta^{\rcp}_i, \mbox{ if }\tilde{f}_i=\tilde{b}_i\\
\rgamma^{\rcp}_i, \mbox{ if } \tilde{f}_i=\tilde{c}_i\\
\rdelta^{\rcp}_i, \mbox{ if } \tilde{f}_i=\tilde{d}_i\\
\repsilon^{\rcp}_i, \mbox{ if } \tilde{f}_i=\tilde{n}_i\\
\end{cases}  \mbox{ (see \cref{def:generalcluster}),}
\end{equation}
then the facet $Z_{S'}$ is cut out by the functionary $\rzeta_{i}(Y)$. Finally, all facets of $Z_S$ arise this way.
\end{claim}

\begin{remark}
It can be shown that in case $\partial_{\tilde\zeta_i}
S_{\tilde D}$ is not rigid, then for the minimal $l>i$ such that the condition in \cref{def:rigid} is not met, $\ralpha_l$ equals the BCFW coordinate $\bar\zeta_i$ of the $i$-th generalized chord which corresponds to $\tilde{f}_i$ according to \cref{def:ftildezeta}.   
\end{remark}

\begin{example}
Consider the example in \cref{fig:bcfw_tile}. 
 All the condensations of the condensable cases in \cref{ex:condensable} are rigid. Therefore $S_\rcp$ has $17$ facets and they are cut out by all the functionaries in \cref{ex:coord_clust}, 
 except for $\rdelta_2(Y),\repsilon_2(Y),\rbeta_4(Y)$, corresponding to the non-condensable cases in \cref{ex:condensable}.
\end{example}

We omit the proof of \cref{claim:facetsgenBCFW} as it is similar to the proof of \cref{prop:standard_facets} in the standard BCFW case, but the technical details are much lengthier.

\begin{remark}
    In the case of standard BCFW cells, the $\tilde{f}_i$-condensation is non-rigid only in the case of $\tilde{f}_i=\tilde{b}_i$ when $D_i$ is a sticky same-end child of a chord $D_p$. In this case, $\rbeta_i=\ralpha_p$ and $\rbeta_i=\ralpha_p=0$ does not cut out a facet.  
    The non-condensable cases correspond precisely to the remaining mutable variables $\Mut(D)$ (cf. \cref{def:frozmut}).
\end{remark}

\section{The spurion tile and tiling}\label{sec:spurion}

The amplituhedron $\Ank$ has a broad class of tiles, the \emph{BCFW tiles} (cf. \cref{def:BCFWtile}). 
Moreover, we can use BCFW tiles to tile $\Ank$ into a broad class of tilings, the \emph{BCFW tilings}, see \cite[Section 12]{even2023cluster}. 
We note that there are tilings made of BCFW tiles which are \emph{not} BCFW tilings (e.g. cf. \cite[Theorem 12.6]{even2023cluster}). However, there are also tiles which are \emph{not} BCFW tiles, and it turns out that they can also be used to tile $\Ank$. In this section we report the first example in the literature of a tiling containing a non-BCFW tile.

\subsection{Spurion tiles}
The simplest case of a tiling with non BCFW tiles is for $n=9$ and $k=2$, i.e. for $\mathcal{A}_{9,2,4}(Z)$. Consider the positroid cell $S_{sp} \subset \Gr_{2,9}^{\geq 0}$ with 
plabic graph in \cref{fig:spurion}. 

\begin{figure}[h!]
\centering
	\includegraphics[width=0.2\textwidth]{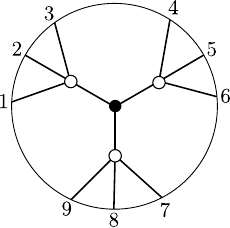}
	\caption{\label{fig:spurion} Plabic graph of the spurion cell $S_{{sp}} \subset \mbox{Gr}^{\geq 0}_{2,9}$.
	}
\end{figure}

A matrix $C_{sp}$ representing a point in $S_{sp}$ has triples of proportional columns whose labels are: $\{1,2,3\}, \{4,5,6\},\{7,8,9\}$. We denote such configuration of column vectors as $(123)(456)(789)$, see \cref{appendix:spurion}.
Therefore any such matrix representative  has rows of support at least $6$. We showed in \cite[Section 6]{even2023cluster} that points in a BCFW cells can be represented by matrices with at least one row of support $5$. 
Therefore, $S_{sp}$ is \emph{not} a BCFW cell and  we call it a \emph{spurion} cell. 
By writing a parametrization with functionaries, and applying techniques from \cite{even2023cluster}, it is possible to show that the amplituhedron map is injective on $S_{sp}$, hence $Z_{sp}:=\overline{\tilde{Z}({S}_{sp})}$ is a tile for $\mathcal{A}_{9,2,4}(Z)$, which we call a \emph{spurion} tile. 
This is an example of a non BCFW tile. 
Applying cyclic shifts to $S_{sp}$ ($Z_{sp}$), we can obtain two other spurion cells (tiles) for $\mathcal{A}_{9,2,4}(Z)$.

\subsection{A tiling containing the spurion} 
We are able to find a tiling $\mathcal{T}_{sp}$ of $\mathcal{A}_{9,2,4}(Z)$ containing a spurion tile. We report the collection of tiles in $\mathcal{T}_{sp}$ in \cref{appendix:spurion}. Moreover,  $\mathcal{T}_{sp}$ is a \emph{good}\footnote{meaning that internal facets of adjacent tiles match pairwise.} tiling of $\mathcal{A}_{9,2,4}(Z)$ and it is `close' to a good BCFW tiling $\mathcal{T}_{BCFW}$. We report the collection of $5$ tiles to substitute in order to go from $\mathcal{T}_{sp}$ to $\mathcal{T}_{BCFW}$ in \cref{appendix:spurion}. We present a sketch of a proof in \cref{skproof:spurion}.

\subsection{Spurion tiles and cluster algebras}
The spurion tile exhibits the same relationship to the cluster structure on $\Gr_{4,n}$ as BCFW tiles. Firstly, $Z_{sp}$ satisfies \emph{cluster adjacency} in \cite[Conjecture 7.17(i)]{even2023cluster}.
Indeed, $Z_{sp}$ has $9$ facets lying on the vanishing locus of the following collection $\mathcal{F}_{sp}$ of functionaries: $a_1(Y)=\llrr{1 23 | 65| 789 }$, $a_2(Y)=\llrr{1 23 | 64| 789 }$, $a_3(Y)=\llrr{1 23 | 54| 789 }$, together with their cyclic shifts $(\cyc^*)^3$ and $(\cyc^*)^6$. The functionaries (up to sign) in $\mathcal{F}_{sp}$ correspond to a collection $\AFacet(Z_{sp})$ of compatible cluster variables of $\Gr_{4,n}$ (see \cref{not:pluckfunc}). A seed $\tilde{\Sigma}_{sp}$ for $\mbox{Gr}_{4,n}$ containing $\AFacet(Z_{sp})$ was found in \cite[Figure 1]{Gurdougan2020ClusterPI}, see \cref{fig:spurionseed}.

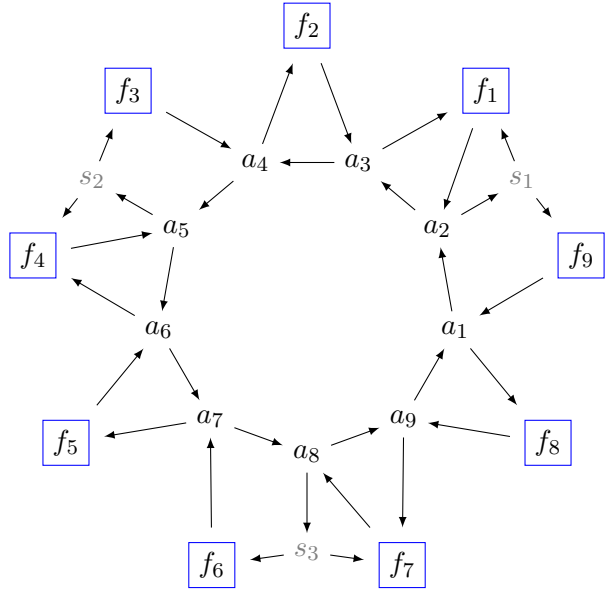
\begin{figure}\centering 
  
\begin{tikzpicture}
  \tikzstyle{frozen} = [draw=blue,fill=none,outer sep=2mm]
  \foreach \ni in {0,...,8}{
    \coordinate (n\ni) at (30+40*\ni:2);
    \pgfmathsetmacro{\nl}{int(Mod(\ni+1,9)+1)};
    \node (a\ni)  at (n\ni)  {$a_{\nl}$};
  }

  \foreach \ni in {0,...,8}{
    \coordinate (n\ni) at (30+20+40*\ni:3.7);
    \pgfmathsetmacro{\nl}{int(\ni+1)};
    \node[frozen] (f\ni) at (n\ni) {$f_{\nl}$};
  }

  \foreach \ni in {0,...,2}{
    \coordinate (n\ni) at (30+120*\ni:3.3);
    \pgfmathsetmacro{\nl}{int(\ni+1)};
    \node (s\ni) at (n\ni) {\color{gray}$s_{\nl}$};
  }

  \foreach \ni in {0,...,8}{
    \draw[-latex] let \n1={int(mod(\ni+1,9))} in (a\ni) -- (a\n1);
  }

  \foreach \ni in {0,...,8}{
    \draw[-latex]   (f\ni) -- (a\ni);    
  }
  \foreach \ni in {3,4,6,7,9,1}{
    \draw[-latex] let \n1={int(mod(\ni+1,9))} in let \n0={int(mod(\ni,9))} in (a\n1) -- (f\n0);
  }

  \foreach \ni in {1,2,3}{
    \draw[-latex] let \n3={int(mod(3*\ni,9))} in let \n1={int(mod(\ni,3))} in (s\n1) -- (f\n3);
    \draw[-latex] let \n3={int(mod(3*\ni,9))} in let \n1={int(mod(\ni,3))} in(a\n3) -- (s\n1);
    \draw[-latex] let \n3={int(mod(3*\ni-1,9))} in let \n1={int(mod(\ni,3))} in (s\n1) -- (f\n3);
  }\,
\end{tikzpicture}    
\caption{The seed $\tilde{\Sigma}_{sp}$, where: $a_1=\lr{123|65|789}$, $a_2=\lr{123|64|789}$, $a_3=\lr{123|54|789}$, $a_4=\lr{789|23|456}$, $a_5=\lr{789|13|456}$, $a_6=\lr{789|12|456}$, $a_7=\lr{456|89|123}$, $a_8=\lr{456|79|123}$, $a_7=\lr{456|78|123}$, $s_1=\lr{1456}$, $s_2=\lr{1237}$, $s_1=\lr{4789}$. The functionaries $a_2(Y), a_5(Y), a_8(Y)$ are positive on $\gto{sp}$ and all the others are negative. }
\label{fig:spurionseed}
  \end{figure}

Moreover, the open spurion tile $Z^{\circ}_{sp} \subset \Gr_{2,6}$ is fully determined by the functionaries in $\mathcal{F}_{sp}$ having a definite sign (see \cref{fig:spurionseed}). Therefore, the coordinate cluster variables $\xx_{sp}$ are exactly the ones in $\AFacet(Z_{sp})$ (containing the functionaries that cut out the facets of $Z_{sp}$). Let $\txx_{sp}$ denote the extended cluster of $\tilde{\Sigma}_{sp}$. We observe that all functionaries $x(Y)$ with $x$ cluster variables in $\txx_{sp}$ have a definite sign on $Z_{sp}$. Furthermore, the seed obtained from $\tilde{\Sigma}_{sp}$ by freezing $\AFacet(Z_{sp})$ is a \emph{signed seed} \cite[Definition 9.22]{even2023cluster}, hence $Z_{sp}$ also satisfies the \emph{positivity test} in \cite[Conjecture 7.17(ii)]{even2023cluster}.

\begin{remark}[Relation to Physics]
Spurion cells first appeared in \cite[Table 1]{abcgpt}. They are informally called `spurion' by physicists because they correspond to Yangian invariants (see, e.g. \cite[Remark 4.6]{even2023cluster}) which have only spurious poles, i.e. poles which cancel in the sum when computing the scattering amplitude. Geometrically, this is reflected in the fact the spurion tile, contrary to general BCFW tiles, does not have any facet which lie on the boundary of the amplituhedron. 

It had been an open problem to determine whether tree-level scattering amplitudes in $\mathcal{N}=4$ SYM could be expressed in terms of the spurion. By showing the amplituhedron $\Ank$ has tilings comprising the spurion tile, we solve this problem. The spurion tiling corresponds to a new expression of scattering amplitudes, which can not be obtained from physics via BCFW recursions. 

\end{remark}

\subsubsection{Sketch of a proof for the tiling with spurion} \label{skproof:spurion}
We now sketch a proof that the spurion tiling of \cref{appendix:spurion} is indeed a tiling.
\begin{itemize}

\item 
Let $S_{\nu} \subset \mbox{Gr}_{2,6}^{\geq 0}$ be the $9$ dimensional positroid cell labelled by the affine permutation $\nu=\{2,6,4,5,8,7,9,12,10\}$.
$S_{\nu}$ has exactly $10$ facets $S_1,\ldots,S_{10} \subset \overline{S}_{\nu}$ that map injectively to $\mathcal{A}_{9,2,4}(Z)$, giving the tiles $Z_{S_1}, \ldots, Z_{S_{10}}$. $Z_{S_1}$ is a spurion tile, and the remaining nine are BCFW tiles, five of which, $Z_{S_6}, \ldots, Z_{S_{10}}$, are part of the BCFW tiling $\mathcal{T}_{\tiny BCFW}$.
We now perform a \emph{flip} on $\mathcal{T}_{\tiny BCFW}$ by replacing the tiles $\gt{S_6},\ldots, \gt{S_{10}}$ with $\gt{S_1},\ldots, \gt{S_{5}}$. Let $\mathcal{T}_{sp}$ be the resulting collection of tiles. We claim that $\mathcal{T}_{sp}$ is a tiling of  $\mathcal{A}_{9,2,4}(Z)$ (which contains the spurion tile $\gt{S_1}$). 

In order to show the claim, it is enough to prove that $\{\gto{S_i}\}_{i=1}^5$ are pairwise disjoint and that
\begin{equation}\label{eq:unionflip}
    \bigcup_{i=1}^5 \gt{S_i}=\bigcup_{i=6}^{10} \gt{S_i}.
\end{equation} Let $\mathcal{F}'$ ($\mathcal{F}''$) denote the left (right) hand side of \cref{eq:unionflip}.

\item The tiles $\gt{S_6},\ldots,\gt{S_{10}}$ have the following facets: $15$ `external' facets $\gt{B_1},\ldots, \gt{B_{15}}$, which cover the boundary of $\mathcal{F''}$; $10$ `internal' facets, each of which belongs to a pair of tiles among $\gt{S_6},\ldots,\gt{S_{10}}$ which lie on opposite sides of it. Similarly, the tiles $\gt{S_1},\ldots,\gt{S_{5}}$ have the same $15$ external facets $\gt{B_1},\ldots, \gt{B_{15}}$ and $10$ internal facets $\gt{B'_1},\ldots, \gt{B'_{10}}$, each of which belongs to a pair of tiles among $\gt{S_1},\ldots,\gt{S_{5}}$.

\item One can show that the functionaries vanishing on the internal facets serve as separating functionaries for all pairs of tiles in $\{\gt{S_i}\}_{i=1}^5$. In particular, if $\gt{B'_{i}}$ is a facet of both $\gt{S_j}$ and $\gt{S_{r}}$, one can show the facet functionary of $\gt{B'_{i}}$ has definite opposite sign on $\gto{S_j}$ and $\gto{S_{r}}$ by using the Cauchy-Binet expansion for twistors (see, for example, \cite[Lemma 2.16]{even2023cluster})  and Pl\"ucker relations. Moreover, using similar techniques, one can show that each external facet $\gt{B_i}$ belongs to a pair of tiles $\gt{S_{j'}} \subset \mathcal{F}'$ and $\gt{S_{j''}} \subset \mathcal{F}''$ and the corresponding facet functionary has definite same sign on $\gto{S_{j'}}$ and $\gto{S_{j''}}$.
\item 

The previous arguments and a topological argument shows that the collection $\{\gt{S_i}\}^5_{i=1}$ tiles $\mathcal{F}'$, whose boundary is $\partial\mathcal{F}''$. Moreover, locally both $\mathcal{F}'$ and $\mathcal{F}''$ lie on the same side of such boundary. Since $\mathcal{F}',\mathcal{F''}$ are of the same dimension of the amplituhedron, by standard algebraic topology arguments (e.g. those of \cite[Section 8]{even2021amplituhedron}), one can conclude that $\mathcal{F}'=\mathcal{F}''$. The claim follows. \qed
\end{itemize}

\section{Standard BCFW tiles as positive parts of cluster varieties}\label{sec:pospart}

In this section, we provide a birational map from $\Gr_{k, k+4}$ to a cluster variety $\Vcal_D$ which maps an open standard BCFW tile $\gto{D}$ bijectively to the positive part of $\Vcal_D$. The tile seed $\check{\Sigma}_D$ defining $\Vcal_D$ is quasi-homomorphic to the seed $\Sigma_D$ of 
\cite[Definition 9.8]{even2023cluster}. Throughout this section, we fix a chord diagram $D \in \CD $. In a mild abuse of notation, we use the terminology ``domino variable" also for the functionary $x(Y)$ corresponding to a domino cluster variable $x \in \xx(D)$.

First, recall we have two sets of functions which determine a point of the tile: the $5k$ coordinate functionaries and the $5k-t$ domino variables, where $t$ is the number of chords of $D$ which are sticky same-end children. It will be useful to express the coordinate functionaries of $\gto{D}$ in terms of the domino variables $\Irr(D)$. By definition, the coordinate functionaries are (signed) Laurent monomials in the domino variables. In the next proposition, we give explicit formulas for these Laurent monomials, up to sign. The signs may be computed using \cite[Proposition 8.10]{even2023cluster} and the fact that all coordinate functionaries are positive on the tile (cf. \cref{thm:BCFW-tile-and-sign-description}).

For a chord $D_i$ in a chord diagram $D$, we set
$E_i := \prod_{\ell} \repsilon_\ell$
where the product is over all ancestors of $D_i$ which contribute to the expression $| c_i~d_i \rchn_i n \rangle$ (cf. \cite[Notation 8.3]{even2023cluster}). We define $E_i'$ identically, but with the product over ancestors contributing to 
$| b_i~c_i \rchn_i n \rangle$ .

\begin{proposition}\label{prop:coord-fcnaries-in-domino-var}
	Let $D \in \CD$ be a chord diagram. Then we have the following expressions for the coordinate functionaries of $\gt{D}$ in terms of the domino variables:
	\[\calpha_i(Y)= \pm\frac{\ralpha_i}{E_i},\qquad \cbeta_i(Y)= \pm \frac{(\rbeta_i)(\ralpha_p)}{E_i}, \qquad \cdelta_i(Y) = \pm \frac{\rdelta_i (\ralpha_{p})}{E_i'}, \qquad \cepsilon_i(Y)= \pm \repsilon_i,\]
	\[\cgamma_i = \pm \frac{\rgamma_i (\ralpha_p)}{E_i (\rdelta_p)(\rbeta_j)(\repsilon_p)(\repsilon_g)}\]
	where $(\ralpha_p)$ appears if $D_i$ has a sticky parent $D_p$; $(\rbeta_i)$ appears \emph{unless} $D_i$ has a sticky and same-end parent; $(\rdelta_p)$ appears if $D_i$ has a same-end parent $D_p$; $(\rbeta_j)$ appears if $D_j$ is right head-to-tail sibling of $D_i$; $(\repsilon_p)$ appears if $(\rbeta_j)$ appears and $D_i$ has a sticky parent $D_p$ which is not same-end to $D_j$; and $(\repsilon_g)$ appears if $D_i$ has a same-end parent $D_p$ and $D_p$ has a sticky but not same-end parent $D_g$.
\end{proposition}

\cref{prop:coord-fcnaries-in-domino-var} can be proved using the explicit formulas for domino variables \cite[Theorem 8.4]{even2023cluster} and \cite[Lemma 8.7]{even2023cluster} on factorization under promotion.

\begin{example}\label{ex:coord-in-domino-var}For the chord diagram $D$ in \cref{cd-example}, the formulas for coordinate functionaries in terms of domino variables are:
\begin{center}
    {\renewcommand{\arraystretch}{1.4}
\begin{tabular}{|c|| c| c| c|c|c|}
 \hline
 $i$ & $\calpha_i$ & $\cbeta_i$  & $\cgamma_i$ & $\cdelta_i$ & $\cepsilon_i$\\ [0.5ex] 
 \hline \hline
 $1$ & $\frac{\ralpha_1}{\repsilon_3}$ & $-\frac{\rbeta_1}{\repsilon_3}$ & $\frac{\rgamma_1}{\rbeta_2\repsilon_3}$ & $-\frac{\rdelta_1}{\repsilon_3}$ & $\repsilon_1$ \\
 \hline
 $2$ & $-\ralpha_2$ & $\rbeta_2$ & $\frac{\rgamma_2}{\rdelta_3\rbeta_6}$ & $\frac{\rdelta_2}{\repsilon_3}$ & $\repsilon_2$ \\
 \hline
 $3$  & $-\ralpha_3$& $\rbeta_3$ & $-\frac{\rgamma_3}{\rbeta_6}$ & $\rdelta_3$ & $\repsilon_3$\\
  \hline
 $4$ & $\frac{\ralpha_4}{\repsilon_6}$ & $-\frac{\ralpha_5}{\repsilon_6}$ & $\frac{\rgamma_4 \ralpha_5}{\rdelta_5\repsilon_6^2}$ & $-\frac{\rdelta_4\ralpha_5}{\repsilon_5 \repsilon_6}$ & $\repsilon_4$\\
  \hline
$5$ & $-\frac{\ralpha_5}{\repsilon_6}$ & $\frac{\rbeta_5 \ralpha_6}{\repsilon_6}$ & $\frac{\rgamma_5\ralpha_6}{\repsilon_6}$ & $-\frac{\rdelta_5 \ralpha_6}{\repsilon_6}$ & $\repsilon_5$ \\
\hline
 $6$ &  $\ralpha_6$ & $-\rbeta_6$ & $\rgamma_6$ & $-\rdelta_6$ & $\repsilon_6$\\
  \hline
\end{tabular}}
\end{center}

\end{example}

Note that both the set of domino variables and the set of coordinate functionaries give redundant descriptions of the tile, which is $4k$ dimensional. We will use \cref{lem:good-scaling} to 
rescale the domino variables $\xx(D)$ by (signed) Laurent monomials in $\Froz(D)$
to obtain $4k$ ``tile variables." The tile variables form a coordinate system for $\gto{D}$, are positive on $\gto{D}$, and will comprise the cluster variables of $\check{\Sigma}_D$.

We perform this scaling in two steps. First, for a domino variable $\rzeta_i(Y)$, let $s$ be the sign of $\rzeta_i(Y)$ on the open tile $\gto{D}$ (cf. 
\cite[Proposition 8.10]{even2023cluster}
) and define the \emph{signed domino variable} as 
$\szeta_i(Y):= s \cdot \rzeta_i(Y)$. 
Note that each coordinate functionary is a Laurent monomial in the signed domino variables, given by the formulas in \cref{prop:coord-fcnaries-in-domino-var} by replacing each domino variable with a signed domino variable and deleting the signs. We denote by $\hxx(D)$ the set of signed domino variables.

The second step of the scaling is more involved. The next proposition identifies the correct scaling factor $m(\szeta_i)$ for each signed domino variable $\szeta_i$, which will be a Laurent monomial in the $\sgamma_i$. The proof of this proposition gives an algorithm to determine the scaling factor.

We use the notation $\PP[X]$ to denote the group of Laurent monomials in the variables $X$. 
\begin{lemma}\label{lem:good-scaling}
	Let $\Gamma:= \{\sgamma_i: D_i \text{ does not have a sticky same-end parent}\}$. There exists a unique group homomorphism $m: \PP[\hxx(D)] \to \PP[\Gamma]$ such that 
	\begin{enumerate}
		\item for $\sgamma_i \in \Gamma$, $m(\sgamma_i)$ is $ \sgamma_i^{-1}$.
		\item for each $i \in [k]$, the image $m(\czeta_i)$ of the coordinate functionary $\czeta_i(Y)$ is equal for all $\czeta \in \{\alpha, \beta, \gamma, \delta, \epsilon\}$.
	\end{enumerate}

Moreover, the degree of $m(\szeta_i)$ in twistor coordinates is equal to the degree of $\szeta_i^{-1}$ in twistor coordinates for all $\szeta_i \in \hhxx(D)$. 
\end{lemma}

\begin{proof}
	A group homomorphism is uniquely determined by the images of $\hxx(D)$. We will determine $m$ on the signed domino variables $\szeta_i(Y)$ for $i=k, k-1, \dots, 1$, in that order. For the rest of this proof, ``degree" means ``degree in twistor coordinates."
	
	We begin with the signed domino variables for the chord $D_k$. Note that $\sgamma_k \in \Gamma$ since $D_k$ is a top chord. So (1) is satisfied if and only if $m(\sgamma_k)= \sgamma_k^{-1}$. Since $D_k$ is a top chord, \cref{prop:coord-fcnaries-in-domino-var} implies that $\szeta_k$ is equal to the coordinate functionary
	$\czeta_k$. Thus (2) is satisfied if and only if $m(\szeta_k)= \sgamma_k^{-1}$ for $\zeta \in \{\alpha, \beta, \gamma, \delta, \epsilon\}$. We see that when (1) and (2) hold, the degree of $\szeta_k^{-1}$ is equal to the degree of $\sgamma_k^{-1}$.
	
	Now, assume for all $\ell >i$ and all signed domino variables $\szeta_\ell$ that there is a unique choice of image $m(\szeta_\ell)$ so that (1) and (2) hold for $\ell$, and the statement about degrees holds. We will show that there is also a unique choice of each image $m(\szeta_i)$ so that (1) and (2) also hold for $i$, and that for this choice, the statement about degrees holds.

	\noindent \textbf{Case 1}: If $\sgamma_i \notin \Gamma$ then (1) is vacuously true. Since $D_i$ is a sticky same-end child of its parent $D_p$, we see from 
	\cref{prop:coord-fcnaries-in-domino-var} 
	that the coordinate functionary $\cbeta_i$ is a Laurent monomial in signed domino variables $\szeta_\ell$ where $\ell >i$. Thus the image $m(\cbeta_i)$ is determined by the values of $m(\szeta_\ell)$. For (2) to hold, we must have $m(\cbeta_i)=m(\czeta_i)$ for all other coordinate functionaries $\czeta_i$. Again by \cref{prop:coord-fcnaries-in-domino-var}, $\czeta_i= \szeta_i \cdot M$ where $M$ is a Laurent monomial in signed domino variables for $\ell >i$. So (2) holds if and only if $m(\szeta_i)= m(\beta_i)/m(M).$
	 
	 For the statement about degrees, notice first that the coordinate functionaries $\czeta_i$ are degree 1, because they are promotions of twistor coordinates and promotion preserves degree. The assumption on the degrees of $m(\szeta_\ell)$ implies that the degree of $m(\cbeta_i)$ is -1. Since $\zeta_i= \szeta_i \cdot M$, the degree of $\szeta_i$ is $1- \deg(M)$. On the other hand, $m(\szeta_i)= m(\beta_i)/m(M)$ implies that the degree of $m(\szeta)$ is $-1 -\deg m(M)$, which is equal to $-1+\deg(M)$ by the assumption on the degrees of $m(\szeta_\ell)$. So we have the desired equality of degrees.
	
	\noindent \textbf{Case 2}: If $\sgamma_i \in \Gamma$, then (1) holds if and only if $m(\sgamma_i)= \sgamma_i^{-1}$. The statement about degrees clearly holds for $\sgamma_i$. The choice of $m(\sgamma_i)$ completely determines the image $m(\cgamma_i)$ of the coordinate functionary $\cgamma_i$, using \cref{prop:coord-fcnaries-in-domino-var}. Similar reasoning as the above case shows that there is a unique choice of $m(\szeta_i)$ so that (2) holds, and that the statement about degrees holds for this choice.
	
\end{proof}

\begin{definition}[Tile variables and seeds]\label{def:tilevars}
	Let $m$ be as in \cref{lem:good-scaling}. For each signed domino variable $\szeta_i(Y) \in \hxx(D) \setminus \Gamma$, we define the \emph{tile variable} as $\tzeta_i(Y) := m(\szeta_i(Y)) \cdot \szeta_i(Y)$. We denote by $\hhxx(D)$ the set of tile variables. 
 We define the \emph{tile seed} $\check{\Sigma}_D=(\hhxx(D), \hhQ_D)$  as the seed obtained from $\Sigma_D$ by deleting $\{\rgamma_i:\rgamma_i \notin \Gamma\}$, and replacing each domino variable $\rzeta_i$ by the corresponding tile variable $\tzeta_i(Y)$. Finally, we let $\mathcal{A}(\check{\Sigma}_D)$ be the associated cluster algebra, which we call \emph{tile cluster algebra}.
\end{definition}

Each tile variable is positive on $\gto{D}$, there are exactly $4k~=~\dim \gto{D}$ tile variables, and each tile variable is degree 0 in the twistor coordinates. It will sometimes be convenient to extend the definition of tile variables to $\szeta_i \in \Gamma$; in this case $\tzeta_i(Y) := 1.$

\begin{example}[Tile cluster variables] 
\label{tile-variables-formulas}
For the chord diagram $D$ in \cref{cd-example}, the domino variables 
$$ \ralpha_2,\; \ralpha_3,\;  
\ralpha_5 = 
\rbeta_4,\;  \rbeta_1,\;  \rbeta_6,\;  \rgamma_2,\;  \rdelta_1,\;  \rdelta_5,\;  \rdelta_6 $$
are negative on the tile $\gto{D}$, and all others are positive (cf. \cref{ex:stdBCFW_postest}). So the signed domino variable $\szeta_i$ coincides with the domino variable $\rzeta_i$ unless $\rzeta_i$ is one of the variables listed above.
To obtain the tile cluster variable $\tzeta_i$ for $D$, multiply $\szeta_i$ by the monomial $m(\szeta_i)$ listed in the table below. 
\setlength{\tabcolsep}{2pt}

\begin{center}
    {\renewcommand{\arraystretch}{1}
\begin{tabular}{|c|| c| c| c|c|c|}
 \hline
 $i$ & $m(\salpha_i)$ & $m(\sbeta_i)$  & $m(\sgamma_i)$ & $m(\sdelta_i)$ & $m(\sepsilon_i)$\\ [0.5ex] 
 \hline \hline
 $1$ & ${\sgamma_2\sgamma_6}({\sgamma_1\sgamma_3})^{-1}$ & ${\sgamma_2\sgamma_6}({\sgamma_1\sgamma_3})^{-1}$ & ${\sgamma_1}^{-1}$ & ${\sgamma_2\sgamma_6}({\sgamma_1\sgamma_3})^{-1}$ & ${\sgamma_2}{\sgamma_1}^{-1}$ \\
 \hline
 $2$ & ${\sgamma_3}({\sgamma_2\sgamma_6})^{-1}$ & ${\sgamma_3}({\sgamma_2\sgamma_6})^{-1}$ & ${\sgamma_2}^{-1}$ & ${\sgamma_2}^{-1}$ & ${\sgamma_3}({\sgamma_2\sgamma_6})^{-1}$ \\
 \hline
 $3$  & ${\sgamma_6}({\sgamma_3})^{-1}$& ${\sgamma_6}({\sgamma_3})^{-1}$ & ${\sgamma_3}^{-1}$ & ${\sgamma_6}({\sgamma_3})^{-1}$ & ${\sgamma_6}({\sgamma_3})^{-1}$\\
  \hline
 $4$ & $({\sgamma_5 \sgamma_6})^{-1}$ & $({\sgamma_5 \sgamma_6})^{-1}$ & $({\sgamma_5 \sgamma_6})^{-1}$ & ${\sgamma_5 }^{-1}$ & ${\sgamma_5 }^{-1}$\\
  \hline
$5$ & $({\sgamma_5 \sgamma_6})^{-1}$ & ${\sgamma_5 }^{-1}$ & ${\sgamma_5 }^{-1}$ & ${\sgamma_5 }^{-1}$ & ${\sgamma_5 }^{-1}$ \\
\hline
 $6$ &  ${\sgamma_6}^{-1}$ & ${\sgamma_6}^{-1}$ & ${\sgamma_6}^{-1}$ & ${\sgamma_6}^{-1}$ & ${\sgamma_6}^{-1}$\\
  \hline
\end{tabular}}

The tile seed $\check{\Sigma}_D$ is displayed on the left in \cref{fig:tile_seed}.
\end{center}
\end{example}

\begin{figure}
    \centering
\includegraphics[width=0.9\textwidth]{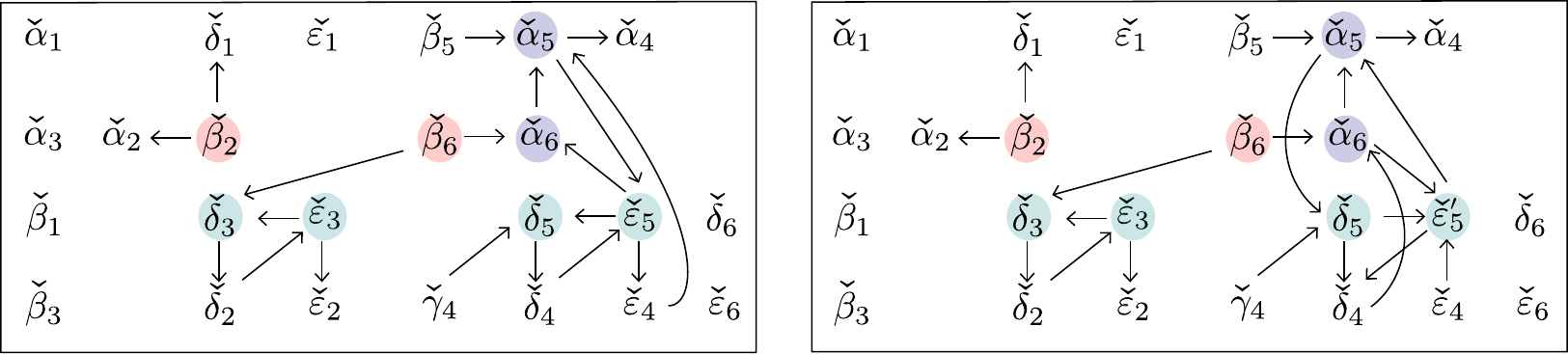}
    \caption{(Left): the tile seed $\check{\Sigma}_D$ for $D$ in \cref{cd-example}. See \cref{domino-variables-formulas,tile-variables-formulas} for the formulas for the tile variables $\tzeta_i$. (Right): the mutation of $\check{\Sigma}_D$ at $\tepsilon_5$.}
    \label{fig:tile_seed}
\end{figure}

As the next result shows, the tile variables give coordinates on the open tile.

\begin{proposition}\label{prop:tile-biregular-pos-reals}
	The map $f: \gto{D} \mapsto \R_+^{\hhxx(D)}$ sending a point $Y \mapsto (\tzeta_i(Y))$ to its list of tile variables is a bijection.
\end{proposition}

\begin{proof}
 	
	We first show that each point in $\R_+^{\hhxx(D)}$ has a preimage in $\gto{D}$. 
Recall that \cref{prop:coord-fcnaries-in-domino-var} gives formulas for each coordinate functionary $\zeta_i(Y)$ as a Laurent monomial $N_{\zeta_i}(\szeta_j(Y))$ in the signed domino variables. 
	We define a Laurent monomial map $$F: \R_+^{\hhxx(D)} \to (\R_+)^{5k}$$ sending $(\tzeta_i) \in \R_+^{\hhxx(D)}$ to $(\zeta'_i:=N_{\zeta_i}(\tzeta_j))$, where the latter set ranges over all coordinate functionaries.  That is, we evaluate the Laurent monomials $N_{\zeta_i}$ for coordinate functionaries in terms of signed domino variables on the tuple $(\tzeta_i)$. (We set $\tzeta_j:=1$ if $\szeta_j(Y) \in \Gamma$.) For a point $p \in \R_+^{\hhxx(D)}$, define $M_p:= \mtx_D(F(p))$ to be the BCFW matrix using $\{\zeta'_i\}$ as BCFW coordinates. We claim that $Y_p:=\tZ(M_p) \in \gto{D}$ is a preimage of $p$ under $f$. That is, the tile variables of $Y_p$ are precisely $p$.
	
	Recall that the rowspan of the BCFW matrix depends only on the projection of $F(p)$ to $(\chR)^k$. We define a vector $q \in (\R_+)^{5k}$ whose entries are $\zeta'_i$ if $D_i$ has a sticky same-end parent and are $\zeta'_i/\gamma'_i$ otherwise. By construction, $q$ and $F(p)$ project to the same point.
	So the rowspan of $M_p$ is equal to the rowspan of $\mtx_D(q)$, and thus (the rowspan of) $Y_p$ is also equal to (the rowspan of) $Y_q:=\tZ(\mtx_D(q))$. \cref{thm:BCFW-tile-and-sign-description}, and in particular the proof of \cite[Proposition 11.15]{even2023cluster},  
 implies that the coordinate functionaries of $Y_q$ are exactly equal to the BCFW coordinates of $\mtx_D(q)$; that is, the coordinate functionaries of $Y_q$ are the entries of the vector $q$. Moreover, the twistor coordinates of $Y_p$ and $Y_q$ differ by a global scalar. Because coordinate functionaries are degree 1 in twistors, the coordinate functionaries of $Y_p$ and $Y_q$ also differ by a global scalar. So $\zeta_i(Y_p)=c \cdot \zeta'_i/\gamma'_i$ if $D_i$ does not have a sticky same-end parent and $\zeta_i(Y_p)= c \cdot \zeta'_i$ otherwise.
	
	We need to show that $\tzeta_i(Y_p)$, a function evaluated on $Y_p$, is equal to $\tzeta_i$, which is either a coordinate of $p$ or equal to 1. We will show this for $i=k, k-1, \dots, 1$.
	
	 For $i=k$, since $D_k$ is a top chord, for any $Y \in \gto{D}$
	\[\zeta_k(Y)= \szeta_k(Y) \quad \text{so} \quad \tzeta_k(Y)=\zeta_k(Y)/ \gamma_k(Y).\] 
	Setting $Y=Y_p$, we obtain $\tzeta_k(Y_p)= \zeta'_k/\gamma'_k$. In this case, according to the definition of $F$, we have $\tzeta_k=\zeta'_k$. In particular, $\gamma'_k=1$. So we have $\tzeta_k(Y_p)= \zeta'_k=\tzeta_k$.
	
	Assume $\tzeta_\ell=\tzeta_\ell(Y_p)$ for $\ell>i$. 
	
	\noindent \textbf{Case 1:} Suppose that $D_i$ has a sticky same-end parent. For any $Y \in \gto{D}$, we have that $N_{\beta_i}(\tzeta_j(Y))= m(\beta_i(Y)) \beta_i(Y)$ and the only tile variables appearing in the Laurent monomial on the left hand side are for chords $D_\ell$ with $\ell > i$. So, for $Y=Y_p$, we have $m(\beta_i(Y_P)) \cdot c \beta'_i= N_{\beta_i}(\tzeta_j)=\beta'_i$, implying that $m(\beta_i(Y_P))= c^{-1}$. For $\zeta_i \neq \beta_i$, we have $$N_{\zeta_i}(\tzeta_j(Y_P))= m(\beta_i(Y_P)) \zeta_i(Y_P)= \zeta'_i= N_{\zeta_i}(\tzeta_j).$$
	In the second equality, we use property (2) of the map $m$.
	Since $N_{\zeta_i}(\tzeta_j(Y))$ is $\tzeta_i(Y)$ times tile variables for $\ell>i$ and $\tzeta_\ell=\tzeta_\ell(Y_p)$ for $\ell>i$, the above string of equalities implies that $\tzeta_i(Y_P)$ is equal to $\tzeta_j$.
	
	\noindent \textbf{Case 2:} Suppose $D_i$ does not have a sticky same-end parent. Then $\tgamma_i(Y)=1=\tgamma_i$, since $\tgamma_i(Y)\in \Gamma$. This means that $\gamma_i'=N_{\gamma_i}(\tzeta_j(Y_p))$. On the other hand, $N_{\gamma_i}(\tzeta_j(Y_p))=m(\gamma_i(Y_p)) \gamma_i(Y_p)= m(\gamma_i(Y_p)) c$, so $c=\gamma_i'/m(\gamma_i(Y_p))$. For $\zeta_i \neq \gamma_i$, we have 
	$$N_{\zeta_i}(\tzeta_j(Y_p))= m(\gamma_i(Y_p)) \zeta_i(Y_p)= c\cdot  m(\gamma_i(Y_p)) \zeta'_i/\gamma_i'= \zeta_i' = N_{\zeta_i}(\tzeta_j).$$ 
	Again, in the second equality, we use property (2) of the map $m$. By a similar argument as in the first case, this shows that $\tzeta_i(Y_p)=\tzeta_i$.
	
	This shows that $Y_p$ is a preimage of $p$ in $\gto{D}$. For uniqueness, note that the tile variables determine the coordinate functionaries up to a scalar for each $i$. So another preimage $Y'$ would have coordinate functionaries $\czeta_i(Y')$ which can only differ from $\czeta_i(Y_p)$ by a scalar $c_i$. However, this implies that the twistor matrix $\mtx_D(Y')$ has the same rowspan as the twistor matrix $\mtx_D(Y_p)$, and thus $Y'= \tZ(\mtx_D(Y'))$ is equal to $Y_p= \tZ(\mtx_D(Y_p))$.

\end{proof}

One may upgrade \cref{prop:tile-biregular-pos-reals} to a statement about the cluster variety $\Vcal_D$ corresponding to the tile seed $\check{\Sigma}_D$ as follows.

\begin{theorem}\label{prop:birat-to-torus}
	Let $f : \Gr_{k, k+4} \dashrightarrow \Vcal_D$ be the map $Y \mapsto (\tzeta_i(Y))$ sending a point to its list of tile variables. Then $f$ is a birational map which maps $\gto{D}$ onto the positive part of $\Vcal_D$.
\end{theorem}

\begin{proof}
	Let $T_D \subset \Gr_{k, k+4}$ be the subset where all tile variables are well-defined and nonvanishing. Note that $T_D$ is open and nonempty, as it contains $\gto{D}$. The map $f$ is well-defined on $T_D$, and the tile coordinates are rational functions in the Pl\"ucker coordinates of $Y$, so $f$ is rational. Note that $f(T_D)$ is contained in the cluster torus $T_{\check{\Sigma}_D}=(\C^*)^{\hhxx(D)} \subset \Vcal_D$.
	
	In the proof of \cref{prop:tile-biregular-pos-reals}, we constructed an inverse to $f$ on the positive part $\R_+^{\hhxx(D)}$ of $\Vcal_D$. This inverse extends to an open subset of the cluster torus $T_{\check{\Sigma}_D}$. Indeed, for $p \in T_{\check{\Sigma}_D}$, define $M_p$ and $Y_p$ as in the proof of \cref{prop:tile-biregular-pos-reals}. The matrix $M_p$ is full-rank by e.g. \cite{MullerSpeyerTwist}, as it is the path matrix of a plabic graph with nonzero complex edge weights. However, $Y_p$ may or may not be full rank. Let $T' \subset T_{\check{\Sigma}_D}$ be the subset of points $p$ such that the coordinate functionaries of $Y_p$ are well-defined and non-vanishing. The coordinate functionaries of $Y_p$ are rational functions in the coordinates of $p$; if they are all well-defined and non-vanishing, then in particular $Y_p$ has at least one nonvanishing twistor coordinate, and so is full rank. The tile variables can be expressed as Laurent monomials in the signed domino variables, and so also as Laurent monomials in coordinate functionaries. Thus, if the coordinate functionaries of $Y_p$ are non-vanishing, so are the tile variables. This implies for $p \in T'$,  $Y_p \in T_D$. Note that $T'$ contains the positive part of $\Vcal_D$, and so is open in $\Vcal_D$.
	
	We claim that $p \mapsto Y_p$ is the inverse of $f$ on $T'$. The argument is very similar to the proof of \cref{prop:tile-biregular-pos-reals}. We outline the additional arguments needed. First, allowing the BCFW coordinates to vary over $(\Gr_{1,5})^k$ rather than $(\chR)^{k}$, the BCFW matrices will parametrize a torus containing $S_D$ \cite{MullerSpeyerTwist}. Second, for any point $Y \in \Gr_{k, k+4}$ which has all non-vanishing coordinate functionaries, the proof of \cite[Proposition 11.15]{even2023cluster}
 shows that the unique pre-image of $Y$ in this torus is given by the twistor matrix $\mtx_D(Y)$. That is, the BCFW coordinates of this unique pre-image are exactly the coordinate functionaries of $Y$. With these facts in hand, the proof of \cref{prop:tile-biregular-pos-reals} goes through identically for $p \in T'$. 	As the Pl\"ucker coordinates of $Y_p$ are rational functions in the coordinates of $p$, $p \mapsto Y_p$ is rational.

	Finally, \cref{prop:tile-biregular-pos-reals} shows that $f$ maps $\gto{D}$ onto the positive part of $\Vcal_D$.

\end{proof}

It would be interesting to upgrade \cref{prop:birat-to-torus} to a biregular map $T_D \to T_{\check{\Sigma}_D}$, or to an embedding $\Vcal_D \hookrightarrow \Gr_{k, k+4}$.

For each cluster in the tile cluster algebra $\A(\check{\Sigma}_D)$, \cref{prop:birat-to-torus} gives 
a way to describe $\gto{D}$ as a semi-algebraic set, this time using dimension-many inequalities:  
\begin{corollary}[Positivity test]\label{cor:tilecoordinates}
	We have
 \begin{align*}
     \gto{D}
     &= \{Y \in \Gr_{k, k+4}: x(Y)>0 \text{ for all } x \text{ in any fixed cluster of } \A(\check{\Sigma}_D)\}
 \end{align*}
In particular, $Y \in \Gr_{k, k+4}$ is in $\gto{D}$ if and only if all tile variables are positive on $Y$.
\end{corollary}

\begin{proof}
 All cluster variables in $\A(\check{\Sigma}_D)$ are positive on $\gto{D}$ by construction, so it suffices to show the right hand side is contained in the left-hand side. If $Y$ is in the right-hand side, then $f(Y)$ is in the positive part of $\Vcal_D$. The inverse of $f$ maps the positive part to $\gto{D}$, so $Y \in \gto{D}$.
\end{proof}

\section{Canonical forms of BCFW tiles from cluster algebra}\label{sec:canonical}
 In this section we use the cluster structure for BCFW tiles to compute the canonical form of such tiles purely in terms of cluster
variables for $\mbox{Gr}_{4,n}$.
\subsection{Background on Positive Geometry}
\begin{definition}[\cite{ABL}]
Let $X$ be a $d$-dimensional 
complex irreducible algebraic variety which is defined over $\mathbb{R}$, and let $X^{\geq 0}$ be a closed\footnote{We always use the Euclidean topology, unless specified otherwise (e.g. in the case of Zariski topology).} semialgebraic subset of $X(\mathbb{R})$, whose interior $X^{> 0}$ is a $d$-dimensional oriented real manifold.
Let $C_1 \dots C_r$ be the irreducible components of the Zariski-closure of the boundary  $X^{\geq 0} \setminus X^{>0}$, and for $1\leq i \leq r$ let $C_i^{\geq 0}$ denote the closure of the interior of 
$C_i \cap X^{\geq 0}$.  We say that 
$(X,X^{\geq 0})$
is a \emph{positive geometry} of dimension $d$ if there exists a unique nonzero rational $d$-form 
$\Omega(X,X^{\geq 0})$ called the \emph{canonical form}, satisfying the recursive axioms:
\begin{itemize}
\item If $d=0$, then $X=X_{\geq 0} = \mbox{pt}$ is a point, and we define $\Omega=\pm 1$ depending on the orientation.
 \item If $d>0$, then we require that 
 $\Omega(X, X_{\geq 0})$ has poles only along the boundary components $C_i$, these poles are simple, and for 
 each $1\leq i \leq r$, we have 
 that $(C_i,C_i^{\geq 0})$ is a positive geometry of dimension $d-1$, called a \emph{facet} of $(X,X^{\geq 0})$, and 
    \begin{equation*}
        \mbox{Res}_{C_i} \Omega(X,X^{\geq 0})=\Omega(C_i,C_i^{\geq 0}).
    \end{equation*} \vspace{-2em}
\end{itemize}
\end{definition}

\begin{example}[$d=1$]  $(\mathbb{P}^1, [a,b])$, with the canonical form $\Omega=\frac{b-a}{(x-a)(b-x)} \mbox{d}x$ is a positive geometry (closed interval). Its facets are: $(\{a\},\{a\}), (\{b\},\{b\})$ and $\mbox{Res}_{a} \Omega=1, \mbox{Res}_{b} \Omega=-1$.
\end{example}

\begin{example}[$d=2$] $(\mathbb{P}^2, \square_{1234})$, where $\square_{1234}$ is a quadrilateral with vertices $v_1=(0, 0); v_2=(2,0); v_3=(1, 2), v_4=(0, 1)$, see \cref{fig:pos_geo}. The canonical form is: 
\begin{equation}\label{eq:canform_square}
   \Omega(\mathbb{P}^2, \square_{1234})=\frac{y-4x -4}{xy(y-x-1)(2x+y-4)} \mbox{d}x \wedge \mbox{d}y.   
    \end{equation}

The facets are: $(\mathbb{P}^1, [v_1,v_2])$, $(\mathbb{P}^1, [v_2,v_3])$, $(\mathbb{P}^1, [v_3,v_4])$, $(\mathbb{P}^1, [v_4,v_1])$.
   \begin{equation*}
 \mbox{Res}_{[v_1,v_2]} \Omega(\mathbb{P}^2, \square_{1234})=\frac{2}{x(2-x)} \mbox{d}x= \Omega(\mathbb{P}^1, [v_1,v_2]).
   \end{equation*}
 $(\mathbb{P}^2, \mbox{half disk})$ with $\Omega=\frac{1}{y(x^2+y^2-1)} \mbox{d}x \wedge \mbox{d}y$ is a positive geometry.
A closed disk is \emph{not} a positive geometry.
For more positive geometries in $d=2$ see the work on \emph{planar polypols} \cite{Kohn}.
\end{example}

\begin{figure}[ht]
    \centering
\includegraphics[width=0.8\textwidth]{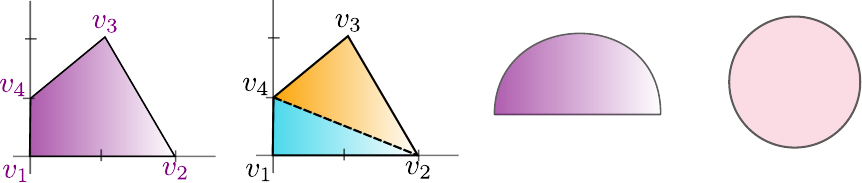}
    \caption{From left to right: the quadrilateral $\square_{1234}$; the tiling of $\square_{1234}$ into the triangles $\Delta_{124}, \Delta_{234}$; half disk; closed disk.}
    \label{fig:pos_geo}
\end{figure}

\begin{definition}
 Let $(X,X^{\geq 0})$ be a positive geometry. A collection $\{(X_i,X_i^{\geq 0})\}_{i \in \CCC}$ of positive geometries is a \emph{tiling} of $(X,X^{\geq 0})$ if:
\begin{itemize}
    \item the interiors $X_i^{>0}$ are pairwise disjoint;
   \item the union $\cup_i X_i^{\geq 0}$ equals $X^{\geq 0}$;
   \item the orientation of each $X_i^{>0}$ agrees with $X^{>0}$.
   \end{itemize}
\end{definition}

\begin{heuristic}\cite{ABL} \label{heuristic:tilings_pos}
Let $(X,X^{\geq 0})$ be a positive geometry and the collection $\{(X_i,X_i^{\geq 0})\}_{i \in \CCC}$ be a tiling of $(X,X^{\geq 0})$. Then
\begin{equation}
  \Omega(X,X^{\geq 0})=\sum_{i \in \CCC} \Omega(X_i,X_i^{\geq 0}).  
\end{equation}
\end{heuristic}

\begin{example}
    $(\mathbb{P}^2, \square_{1234})$ can be tiled by the two triangles $(\mathbb{P}^2, \Delta_{124})$ and $(\mathbb{P}^2, \Delta_{234})$ with vertices $(v_1,v_2,v_4)$ and $(v_2,v_3,v_4)$ respectively, see \cref{fig:pos_geo}. Their canonical forms are:
    \begin{equation*}
    \Omega(\Delta_{124})=\frac{2}{x y (2-x-2y)} \mbox{d} x  \wedge \mbox{d} y, \quad \Omega(\Delta_{234})=\frac{9}{(1+x-y) (4-2x-y) (2-x-2y)} \mbox{d} x  \wedge \mbox{d} y.
\end{equation*}
    Then $\Omega(\square_{1234})=\Omega(\Delta_{124})+\Omega(\Delta_{234})$, cf. \cref{eq:canform_square}. Moreover, the (`spurious') pole along the facet $(24)$ cut out by $2-x-2y=0$ cancels in the sum. Indeed, $(24)$ is not a facet of $\square_{1234}$.

\end{example}

\begin{theorem}\cite{ABL, KohnRanestad}
    Let $\mathcal{P}$ be a projective pointed polyhedral cone (or projective polytope) in $\mathbb{P}^m$. Then  $(\mathbb{P}^m, \mathcal{P})$ is a positive geometry. Moreover,
\begin{equation*}
\Omega(\mathbb{P}^m, \mathcal{P})=\frac{N(x)}{D(x)} \mbox{d}^m x,    
\end{equation*}
  where $D(x)$ is the product of linear forms defining facets of $\mathcal{P}$, and $N(x)$ is the \emph{adjoint} of $\mathcal{P}$.
\end{theorem}
The \emph{adjoint} is a polynomial that cancels the `unwanted' poles outside the polyope, i.e. it cuts out the hypersurface which passes through the \emph{residual hyperplane arrangement} of $\mathcal{P}$.

\begin{theorem}\cite{postnikov, KLS, Lam:2022yly}
\label{th:GrposGeo}
    $(\mbox{Gr}_{k,n}(\mathbb{C}), \mbox{Gr}^{\geq 0}_{k,n})$ is a positive geometry with canonical form:
\begin{equation*}
    \Omega(\Gr_{k,n}(\mathbb{C}), \Gr^{\geq 0}_{k,n})= \frac{\mbox{d}^{k (n-k)} C}{\lr{1, \ldots ,k} \lr{2, \ldots ,k+1}\ldots \lr{n, 1, \ldots ,k-1}},
\end{equation*}
where $\lr{I}$ denotes the Pl\"ucker coordinate of a point $C \in \mbox{Gr}^{\geq 0}_{k,n}$. Moreover, the faces $(\Pi_S(\mathbb{C}),\bar{S})$ are positive geometries, where $S \subset \mbox{Gr}^{\geq 0}_{k,n}$ is a positroid cell and $\Pi_S(\mathbb{C})$ is its Zariski closure in $\mbox{Gr}_{k,n}(\mathbb{C})$, called the \emph{positroid variety} of $S$.
\end{theorem}

\subsection{The canonical form of the amplituhedron}

Both (cyclic) polytopes and the positive Grassmannian are positive geometries.  These objects can also be seen as special cases of amplituhedra (in particular, the  amplituhedra
$\mathcal{A}_{n,1,m}(Z)$ and $\mathcal{A}_{n,n-m,m}(Z)$, respectively). Since the amplituhedron
$\mathcal{A}_{n,k,m}(Z)$ is a subset of $\Gr_{k,k+m}$,
it is natural to conjecture the following.

\begin{conjecture}\cite{ABL}
    The amplituhedron $(\Gr_{k,k+m}(\mathbb{C}), \mathcal{A}_{n,k,m}(Z))$ is a positive geometry.
\end{conjecture}

In order to find the canonical form of the amplituhedron, one method is to tile $\mathcal{A}_{n,k,m}(Z)$ and sum over the canonical forms of the tiles (cf. \cref{heuristic:tilings_pos}). 

\begin{definition}[Candidate canonical form of a tile]\label{def:cand_can_form} Let $Z_S$ be a tile of $\mathcal{A}_{n,k,m}(Z)$.  As the amplituhedron map $\tilde{Z}$ is injective on the open tile $\gto{S}$, we can define its inverse $\tilde{Z}^{-1}: \gto{S} \rightarrow S$. Then let us consider the pullback of the canonical form of the positroid cell under $\tilde{Z}^{-1}$:
\begin{equation}\label{eq:can_form_tile}
  \tilde{\Omega}(Z_{S})=( \tilde{Z}^{-1})^* \Omega(\Pi_S(\mathbb{C}),\bar{S}).
\end{equation}    
We call  $\tilde{\Omega}(Z_{S})$ the\footnote{we will always consider it up to a global sign, which is not relevant for our paper and  depends on the orientation.} \emph{candidate canonical form of the tile} $Z_S$.
\end{definition}
By \cref{th:GrposGeo}, $(\Pi_S(\mathbb{C}),\bar{S})$ is positive geometry and has a canonical form $\Omega(\Pi_S(\mathbb{C}),\bar{S})$. Moreover, by \cite[Section 6]{arkani-hamed_trnka} and \cite[Section 8.2]{Galashin_Lam_2020}, \cref{eq:can_form_tile} is well-defined.

Each positroid cell $S$ has a positive parameterization \cite{postnikov}, i.e. there is a diffeomorphism $h: S \to \mathbb{R}_+^{mk}$ which sends a matrix representative $C$ in $S$ to a collection of positive coordinates $(\alpha_1,\ldots,\alpha_{mk})$ in $\mathbb{R}_+^{mk}$. In this case, if we denote $\phi=h \circ \tilde{Z}^{-1}$, then
\begin{equation}\label{eq:canformpar}
\tilde{\Omega}(Z_{S})=\phi^*\bigwedge_{i=1}^{mk} \mbox{d log} (\alpha_i).
\end{equation}

\begin{conjecture}[Tiles are positive geometries] Let $Z_{S}$ be a tile of $\mathcal{A}_{n,k,m}(Z)$. Then $(\Gr_{k,k+m}(\mathbb{C}),Z_{S})$ is a positive geometry and its canonical form $\Omega(\Gr_{k,k+m}(\mathbb{C}),Z_{S})$ is the candidate canonical form $\tilde{\Omega}(Z_S)$ in \cref{def:cand_can_form}.
\end{conjecture}

\begin{conjecture}[Canonical form from tilings]
Let $\{Z_{S}\}_{S \in \mathcal{C}}$ be a tiling of $\mathcal{A}_{n,k,m}(Z)$. Then the canonical form of the amplituhedron $\mathcal{A}_{n,k,m}(Z)$ is obtained as
\begin{equation} \label{eq:canformampl}
 \Omega(\Gr_{k,k+m}(\mathbb{C}),\mathcal{A}_{n,k,m}(Z))=\sum_{S \in \mathcal{C}} \Omega(\Gr_{k,k+m}(\mathbb{C}),Z_{S}). 
\end{equation}    
In particular, the right hand side of \cref{eq:canformampl} is independent of the tiling.
\end{conjecture}

\begin{remark}
Clearly finding tilings of the amplituhedron and inverting the amplituhedron map on tiles are crucial step for computing the canonical form of the amplituhedron, and hence scattering amplitudes. 
In this paper and in \cite{even2023cluster} we inverted the amplituhedron map \cite[Theorem 7.7]{even2023cluster} on \emph{BCFW tiles} and proved the existence of a large family of tilings, the \emph{BCFW tilings} \cite[Theorem 12.3]{even2023cluster}. 
It then follows from \cite{Mason:2009qx,Arkani-Hamed:2009ljj,Arkani-Hamed:2012zlh,arkani-hamed_trnka} that tree-level scattering amplitudes in $\mathcal{N}=4$ SYM expressed via BCFW recursions are computed by the sum of the candidate canonical forms of the tiles in a BCFW tiling of $\mathcal{A}_{n,k,4}(Z)$.   
\end{remark}

\begin{proposition}[Canonical form of tiles from coordinate functionaries] Let $Z_\rcp$ be a BCFW tile and $\left([\calpha_i(Y): \cbeta_i(Y):\cgamma_i(Y): \cdelta_i(Y): \cepsilon_i(Y)]\right)_{i=1}^k$ its associated coordinate functionaries as in \cite[Definition 7.1]{even2023cluster}. Then the candidate canonical form $\tilde{\Omega}(Z_\rcp)$ of $Z_\rcp$ is given by:
\begin{equation} \label{eq:canform_coordsfunc}
  \tilde{\Omega}(Z_{S})=\bigwedge_{i=1}^k \mbox{dlog} \frac{\beta_i(Y)}{\alpha_i(Y)} \wedge \mbox{dlog} \frac{\gamma_i(Y)}{\alpha_i(Y)} \wedge \mbox{dlog} \frac{\delta_i(Y)}{\alpha_i(Y)} \wedge \mbox{dlog} \frac{\epsilon_i(Y)}{\alpha_i(Y)}.
\end{equation}
Analogously, for each $i \in [k]$, we could have chosen any other coordinate functionary $\zeta_i(Y)$ instead of $\alpha_i(Y)$ to divide the others by. 
\end{proposition}
\begin{proof}
 Given a BCFW tile $Z_\rcp$, the inverse of the amplituhedron map $\tilde{Z}^{-1}$ sends a point $Y$ in $\gto{\rcp}$ to a point in $\mbox{Gr}^{\geq 0}_{k,n}$ represented by the twistor matrix $\twmt_\rcp(Y)$ \cite[Definition 7.1]{even2023cluster}. Moreover, there is a positive parametrization of $S_\rcp$  in terms of BCFW parameters $\left([\calpha_i: \cbeta_i:\cgamma_i: \cdelta_i: \cepsilon_i]\right)_{i=1}^k$ in $(\mbox{Gr}^{>0}_{1,5})^{k}$ \cite[Proposition 6.22]{even2023cluster}, or equivalently in terms of e.g. $\left(\frac{\cbeta_i}{\calpha_i},\frac{\cgamma_i}{\calpha_i},\frac{\cdelta_i}{\calpha_i},\frac{\cepsilon_i}{\calpha_i}\right)_{i=1}^k$ in $\mathbb{R}^{4k}_+$. Composing this with $\tilde{Z}^{-1}$ gives a diffeomorphism $g:\gto{S} \rightarrow \mathbb{R}^{4k}_+$ that sends $Y \in \gto{S}$ to the (ratios of) coordinate functionaries  $\left(\frac{\beta_i(Y)}{\alpha_i(Y)},\frac{\gamma_i(Y)}{\alpha_i(Y)},\frac{\delta_i(Y)}{\alpha_i(Y)},\frac{\epsilon_i(Y)}{\alpha_i(Y)}\right)_{i=1}^k$. Hence can obtain the candidate canonical form of the tile $Z_S$ as:
\begin{equation*} 
  \tilde{\Omega}(Z_{S})=g^* \bigwedge_{i=1}^k \mbox{dlog} \frac{\beta_i}{\alpha_i} \wedge \mbox{dlog} \frac{\gamma_i}{\alpha_i} \wedge \mbox{dlog} \frac{\delta_i}{\alpha_i} \wedge \mbox{dlog} \frac{\epsilon_i}{\alpha_i}=\bigwedge_{i=1}^k \mbox{dlog} \frac{\beta_i(Y)}{\alpha_i(Y)} \wedge \mbox{dlog} \frac{\gamma_i(Y)}{\alpha_i(Y)} \wedge \mbox{dlog} \frac{\delta_i(Y)}{\alpha_i(Y)} \wedge \mbox{dlog} \frac{\epsilon_i(Y)}{\alpha_i(Y)}.
\end{equation*}   
\end{proof}

\begin{example}
    For the BCFW tile $S_{\rcp}$ in \cref{fig:bcfw_tile}, the coordinate functionaries $\{\zeta_i(Y)\}$ are in \cref{ex:coord_func}. Then we can compute the canonical form of $S_{\rcp}$ in terms of $\{\zeta_i(Y)\}$ by \cref{eq:canform_coordsfunc}. 
\end{example}

\begin{proposition}[Canonical form of tiles from tile variables and clusters]
Let $Z_D$ be a standard BCFW tile. Let  $\hhxx(D)=\{\tzeta_i(Y)\}_{i=1}^{4k}$ be its collection of tile variables and $\A(\check{\Sigma}_D)$ its associated cluster algebra as in \cref{def:tilevars}. Then the candidate canonical form $\tilde{\Omega}(Z_D)$ of $Z_D$ is given by:
\begin{equation}\label{eq:canform_tilevars}
 \tilde{\Omega}(Z_{D})=\bigwedge_{\tzeta_i(Y) \in \hhxx(D)} \mbox{dlog } \tzeta_i(Y).
\end{equation}
Moreover, for each fixed cluster $\hhxx=\{x_i\}_{i=1}^{4k}$ in $\A(\check{\Sigma}_D)$, the form $\tilde{\Omega}(Z_D)$ is given by:
\begin{equation}
 \tilde{\Omega}(Z_{D})=\bigwedge_{x_i \in \hhxx} \mbox{dlog } x_i(Y).   
\end{equation}
\end{proposition}
The proof easily follows from \cref{prop:birat-to-torus}, and the fact there is a bijection $f:\gto{D} \rightarrow \mathbb{R}^{4k}_+$ that sends $Y 
\in \gto{D}$ to the collection $\hhxx(D)=\{\tzeta_i(Y)\}$ of $4k$ tile variables. Each tile variable $\tzeta_i(Y)$ is a signed ratio of cluster variables for $\mbox{Gr}_{4,n}$, in particular of domino variables $\xx(D)$, see \cref{prop:coord-fcnaries-in-domino-var}. The same argument holds if instead of $\hhxx(D)$, we consider an arbitrary cluster $\hhxx$ in $\A(\check{\Sigma}_D)$.

\begin{example}
    For the BCFW tile $Z_{D}$ in \cref{cd-example-again}, the tile variables $\tzeta_i(Y)$ were computed in \cref{domino-variables-formulas}. Then we can compute the candidate canonical form $\tilde{\Omega}(Z_{D})$ in terms of $\tzeta_i(Y)$ by \cref{eq:canform_tilevars}. Moreover, we can also compute $\tilde{\Omega}(Z_{D})$ by using a different cluster obtained e.g. by mutating the tile seed $\check{\Sigma}_D$ at $\tepsilon_5$, see \cref{fig:tile_seed}. The collection of cluster variables then would have $\tepsilon'_5$ instead of $\tepsilon_5$, where
    $$\tepsilon_5'= \frac{\lr{ABC \br 89 \br DEF}}{\sgamma_5 \sgamma_6}=  \frac{\lr{ABC \br 89 \br DEF}}{\lr{8~9~A~D} \lr{8~9~E~F}} $$
    and we use $A,B,C,D,E,F$ for $10,11,12,13,14,15$. Note that any sequence of mutations applied to $\check{\Sigma}_D$ will give cluster variables which are cluster variables for $\Gr_{4,n}$ times a Laurent monomial in the $\sgamma_i$.
\end{example}

\appendix
\newpage
\section{List of tiles in a spurion tiling}\label{appendix:spurion}
A good tiling $\mathcal{T}_{sp}=\{Z_{\pi}\}$ for $\mathcal{A}_{9,2,4}$ containing a spurion tile ($\# 28$ in the list).  For each tile $Z_{\pi}$ we display its affine permutation $\pi$ and the vector configuration of the columns of a matrix $C_{\pi}$ representing a point in $S_{\pi}$, where columns within the same bracket $(\ldots)$ are proportional.

\begin{spacing}{1.15}
\setlength{\columnsep}{-1cm}
\begin{multicols}{2}
\hspace{-5.2em}
\noindent\(\begin{array}{||c|c|c}
\hline
\# & \mbox{Affine Permutation} & \mbox{Vector Conf.} \\
 \hline
 1 & \{1,4,5,6,7,11,12,8,9\} & 
\begin{array}{c}
 (2)(3)(4)(5)(6)(7) \\
\end{array} \\
 2 & \{1,4,5,7,6,8,12,11,9\} &  
\begin{array}{c}
 (8,2)(3)(4)(5,6)(7) \\
\end{array}
  \\
 3 & \{1,4,5,7,8,6,11,12,9\} & 
\begin{array}{c}
 (2)(3)(4)(5)(7)(8) \\
\end{array}
 \\
 4 & \{1,4,5,8,6,7,12,9,11\} & 
\begin{array}{c}
 (8,9,2)(3)(4)(5,6,7) \\
\end{array}
  \\
 5 & \{1,4,5,8,7,6,9,12,11\} & 
\begin{array}{c}
 (9,2)(3)(4)(5,7)(8) \\
\end{array}
  \\
 6 & \{1,4,5,8,9,6,7,11,12\} & 
\begin{array}{c}
 (2)(3)(4)(5)(8)(9) \\
\end{array}
  \\
 7 & \{1,5,3,6,7,8,11,13,9\} & 
\begin{array}{c}
 (2)(4)(5)(6)(7)(8) \\
\end{array}
  \\
 8 & \{1,5,3,6,8,7,9,13,11\} & 
\begin{array}{c}
 (9,2)(4)(5)(6,7)(8) \\
\end{array}
  \\
 9 & \{1,5,3,6,8,9,7,11,13\} & 
\begin{array}{c}
 (2)(4)(5)(6)(8)(9) \\
\end{array}
  \\
 10 & \{1,5,4,6,7,11,8,12,9\} & 
\begin{array}{c}
 (2)(3,4)(5)(6)(7,8) \\
\end{array}
  \\
 11 & \{1,5,4,6,8,11,7,9,12\} & 
\begin{array}{c}
 (2)(3,4)(5)(6)(8,9) \\
\end{array}
  \\
 12 & \{1,5,4,7,6,8,11,12,9\} & 
\begin{array}{c}
 (2)(3,4)(5,6)(7)(8) \\
\end{array}
  \\
 13 & \{1,5,4,8,6,7,9,12,11\} & 
\begin{array}{c}
 (3,4)(5,6,7)(8)(9,2) \\
\end{array}
  \\
 14 & \{1,5,4,8,6,9,7,11,12\} & 
\begin{array}{c}
 (2)(3,4)(5,6)(8)(9) \\
\end{array}
  \\
 15 & \{1,6,3,4,7,8,9,11,14\} & 
\begin{array}{c}
 (2)(5)(6)(7)(8)(9) \\
\end{array}
  \\
 16 & \{1,6,4,5,7,8,12,9,11\} & 
\begin{array}{c}
 (3,4,5)(6)(7)(8,9,2) \\
\end{array}
  \\
 17 & \{1,6,4,5,8,7,9,12,11\} & 
\begin{array}{c}
 (3,4,5)(6,7)(8)(9,2) \\
\end{array}
  \\
 18 & \{1,6,5,4,8,7,9,11,12\} & 
\begin{array}{c}
 (2)(3,5)(6,7)(8)(9) \\
\end{array}
  \\
 19 & \{2,3,5,9,6,7,8,13,10\} & 
\begin{array}{c}
 (9,1,2,3)(4)(5,6,7,8) \\
\end{array}
  \\
 20 & \{2,4,5,9,6,8,7,12,10\} & 
\begin{array}{c}
 (9,1,2)(3)(4)(5,6,8) \\
\end{array}
  \\
 21 & \{2,4,5,9,8,6,7,10,12\} & 
\begin{array}{c}
 (1,2)(3)(4)(5,8)(9) \\
\end{array}
  \\
 22 & \{2,4,6,9,5,7,8,12,10\} & 
\begin{array}{c}
 (9,1,2)(3)(4)(6,7,8) \\
\end{array}
  \\
 23 & \{2,5,3,6,9,7,8,13,10\} & 
\begin{array}{c}
 (9,1,2)(4)(5)(6,7,8) \\
\end{array}
  \\
 24 & \{2,5,3,6,9,8,7,10,13\} & 
\begin{array}{c}
 (1,2)(4)(5)(6,8)(9) \\
\end{array}
  \\
 25 & \{2,5,4,6,9,10,7,8,12\} & 
\begin{array}{c}
 (1,2)(3,4)(5)(6)(9) \\
\end{array} \\
\hline
\end{array}\)

\noindent\(\begin{array}{||c|c|c||}
\hline
\# & \mbox{Affine Permutation} & \mbox{Vector Conf.} \\
\hline
 26 & \{2,5,4,9,6,8,7,10,12\} & 
\begin{array}{c}
 (1,2)(3,4)(5,6,8)(9) \\
\end{array}
  \\
 27 & \{2,6,4,5,7,9,12,8,10\} & 
\begin{array}{c}
 (3,4,5)(6)(7)(9,1,2) \\
\end{array}
  \\
 28 & \{2,6,4,5,9,7,8,12,10\} & 
\begin{array}{c}
 (3,4,5)(6,7,8)(9,1,2) \\
\end{array}
  \\
 29 & \{2,7,4,5,6,8,12,9,10\} & 
\begin{array}{c}
 (3,4,5,6)(7)(8,9,1,2) \\
\end{array}
  \\
 30 & \{2,7,4,5,8,6,9,12,10\} & 
\begin{array}{c}
 (3,4,5)(7)(8)(9,1,2) \\
\end{array}
  \\
 31 & \{3,4,5,9,10,6,7,8,11\} & 
\begin{array}{c}
 (1)(2)(3)(4)(5)(9) \\
\end{array}
  \\
 32 & \{3,5,4,6,10,9,7,8,11\} & 
\begin{array}{c}
 (1)(2)(3,4)(5)(6,9) \\
\end{array}
  \\
 33 & \{3,5,6,4,9,10,7,8,11\} & 
\begin{array}{c}
 (1)(2)(3)(5)(6)(9) \\
\end{array}
  \\
 34 & \{3,6,4,5,9,8,7,11,10\} & 
\begin{array}{c}
 (9,1)(2)(3,4,5)(6,8) \\
\end{array}
  \\
 35 & \{3,6,4,5,10,7,8,11,9\} & 
\begin{array}{c}
 (1)(2)(3,4,5)(6,7,8) \\
\end{array}
  \\
 36 & \{3,6,4,9,5,8,7,10,11\} & 
\begin{array}{c}
 (1)(2)(3,4)(6,8)(9) \\
\end{array}
  \\
 37 & \{3,6,5,4,7,8,11,10,9\} & 
\begin{array}{c}
 (8,1)(2)(3,5)(6)(7) \\
\end{array}
  \\
 38 & \{3,6,5,4,8,7,11,9,10\} & 
\begin{array}{c}
 (8,9,1)(2)(3,5)(6,7) \\
\end{array}
  \\
 39 & \{3,6,5,4,9,7,8,11,10\} & 
\begin{array}{c}
 (9,1)(2)(3,5)(6,7,8) \\
\end{array}
  \\
 40 & \{3,7,5,4,8,6,9,11,10\} & 
\begin{array}{c}
 (9,1)(2)(3,5)(7)(8) \\
\end{array}
  \\
 41 & \{4,6,3,5,9,8,7,10,11\} & 
\begin{array}{c}
 (1)(2)(4,5)(6,8)(9) \\
\end{array}
  \\
 42 & \{5,6,3,4,7,9,8,11,10\} & 
\begin{array}{c}
 (9,1)(2)(5)(6)(7,8) \\
\end{array}
  \\
 43 & \{5,6,3,4,8,9,7,10,11\} & 
\begin{array}{c}
 (1)(2)(5)(6)(8)(9) \\
\end{array}
  \\
 44 & \{5,7,3,4,6,8,9,11,10\} & 
\begin{array}{c}
 (9,1)(2)(5,6)(7)(8) \\
\end{array}
  \\
 45 & \{6,3,4,5,9,7,8,11,10\} & 
\begin{array}{c}
 (2,3,4,5)(6,7,8)(9,1) \\
\end{array}
  \\
 46 & \{6,3,4,9,5,7,8,10,11\} & 
\begin{array}{c}
 (1)(2,3,4)(6,7,8)(9) \\
\end{array}
  \\
 47 & \{6,3,5,4,7,8,11,9,10\} & 
\begin{array}{c}
 (2,3,5)(6)(7)(8,9,1) \\
\end{array}
  \\
 48 & \{6,4,3,5,9,7,8,10,11\} & 
\begin{array}{c}
 (1)(2,4,5)(6,7,8)(9) \\
\end{array}
  \\
 49 & \{6,5,3,4,7,9,8,10,11\} & 
\begin{array}{c}
 (1)(2,5)(6)(7,8)(9) \\
\end{array}
  \\
 50 & \{6,7,3,4,5,8,9,10,11\} & 
\begin{array}{c}
 (1)(2)(6)(7)(8)(9) \\
\end{array}
  \\
  \hline
\end{array}\)
\end{multicols}
\end{spacing}

Substituting the $5$ tiles $\# 28,34,35,38,46$ in $\mathcal{T}_{sp}$ with $5$ tiles in the table below, we obtain a good BCFW tiling $\mathcal{T}_{BCFW}$.
\vspace{0.5em}
\begin{spacing}{1.15}
\noindent\(\begin{array}{||c|c||}
\hline
 \mbox{Affine Permutation} & \mbox{Vector Conf.} \\
 \hline
  \{1,6,4,5,9,7,8,11,12\} & 
\begin{array}{c}
 (2)(3,4,5)(6,7,8)(9)\\
\end{array} \\
 \{3,6,4,5,9,7,11,8,10\} &  
\begin{array}{c}
 (9,1)(2)(3,4,5)(6,7) \\
\end{array}
  \\
  \{3,6,4,9,5,7,8,11,10\} & 
\begin{array}{c}
 (9,1)(2)(3,4)(6,7,8) \\
\end{array}
 \\
  \{3,7,4,5,9,6,8,11,10\} & 
\begin{array}{c}
 (9,1)(2)(3,4,5)(7,8) \\
\end{array}
  \\
 \{4,6,3,5,9,7,8,11,10\} & 
\begin{array}{c}
 (9,1)(2)(4,5)(6,7,8) \\
\end{array}
\\
\hline
\end{array}\)
\end{spacing}

\bibliographystyle{alpha}
	\bibliography{ClusterTilesPromotion}
\end{document}